\documentclass[11pt,reqno]{amsart}

\usepackage{amsmath, amsthm, amssymb}
\usepackage{amsfonts}
\usepackage{bm,mathrsfs}
\usepackage{enumitem}
\usepackage{mathrsfs}

\usepackage{hyperref}

\usepackage[ansinew]{inputenc}
\usepackage[dvips]{epsfig}
\usepackage{graphicx}
\usepackage[english]{babel}

\usepackage{thmtools}
\theoremstyle{plain}
\declaretheorem[title=Theorem, parent=section]{theorem}
\declaretheorem[title=Lemma,parent=section]{lemma}
\declaretheorem[title=Proposition,parent=section]{proposition}
\declaretheorem[title=Corollary,parent=section]{corollary}
\declaretheorem[title=Problem, parent=section]{problem}

\theoremstyle{remark}
\declaretheorem[title=Remark, parent=section]{remark}

\theoremstyle{definition}
\declaretheorem[title=Definition,parent=section]{definition}
\declaretheorem[title=Example, sibling=theorem]{example}
\declaretheorem[title=Assumption, numbered=no]{assumption*}

\numberwithin{equation}{section}

\usepackage[backgroundcolor=white, bordercolor=blue,
linecolor=blue]{todonotes}

\usepackage{dsfont}
\usepackage{bbm}
\usepackage{mathrsfs}

\newcommand{\average}{{\mathchoice {\kern1ex\vcenter{\hrule height.4pt
width 6pt depth0pt} \kern-9.7pt} {\kern1ex\vcenter{\hrule
height.4pt width 4.3pt depth0pt} \kern-7pt} {} {} }}

\begin{document}
\allowdisplaybreaks
\title{Regularity for degenerate or singular fully nonlinear elliptic equations with Hamiltonian terms and applications}

\author{Jiangwen Wang}
\author{Feida Jiang$^*$}

\address{School of Mathematics and Shing-Tung Yau Center of Southeast University, Southeast University, Nanjing 211189, P.R. China}
\email{\url{jiangwen\_wang@seu.edu.cn}}

\address{School of Mathematics and Shing-Tung Yau Center of Southeast University, Southeast University, Nanjing 211189, P.R. China; Shanghai Institute for Mathematics and Interdisciplinary Sciences, Shanghai 200433, P.R. China}
\email{\url{jiangfeida@seu.edu.cn}}

\date{\today}
	\thanks{*corresponding author}



\allowdisplaybreaks

\begin{abstract}
In this paper, we investigate regularity properties for viscosity solutions to a general class of degenerate or singular fully nonlinear elliptic equations with Hamiltonian terms,
\[
\left\{
\begin{alignedat}{2}
\Phi(|Du|,x)F(D^{2}u,x)+H(|Du|,x)
&   =f(x) \quad    &&    \text{in }  \quad  \Omega,\\
u(x) &  = g(x) \quad    &&       \text{on } \quad   \partial\Omega.
\end{alignedat}
\right.
\]
Here, $F$ is uniformly elliptic, while $\Phi$ and $H$ satisfy suitable structural and growth conditions allowing for both degenerate and singular regimes.

Our first result establishes sharp global $C^{1,\alpha}$ regularity for this general class of equations, thereby providing a unified regularity framework for both degenerate and singular regimes.

We next develop an oscillation-based approach for fully nonlinear equations with unbalanced variable degeneracy and Hamiltonian terms. Under a H\"older-type decay assumption on $ f $, we derive sharp boundary $C^{1,\beta'}$ estimates with an explicit exponent, and establish a quantitative non-degeneracy estimate in the singular region.

Finally, as applications of the general theory, under a suitable viscosity curvature condition on the level sets of the solution, we establish global $ C^{1,\frac{1}{3}} $ regularity for the infinity-Poisson and global $ C^{1,\frac{1}{p-1}} $ regularity for the $ p $-Poisson with $ p>2 $. These results may provide a new perspective on these two long-standing open problems.     
\end{abstract}

\allowdisplaybreaks

\maketitle
\tableofcontents

\section{Introduction}\label{Intro}

This paper is concerned with global gradient regularity for viscosity solutions to degenerate or singular fully nonlinear elliptic equations with Hamiltonian terms, under the Dirichlet boundary condition
\begin{equation}
\label{Main:eq1}
\left\{
     \begin{alignedat}{2}
         \Phi(|Du|,x) F(D^{2}u,x)+ H(|Du|, x)  & = f(x)         \quad   &&   \text{in}  \quad   \Omega     ,    \\
          u(x)   &  = g(x)   \quad  &&   \text{on}  \quad    \partial \Omega.        \\
     \end{alignedat}
     \right.
\end{equation}
Here, \( F : \mathrm{Sym}(n) \times \Omega \to \mathbb{R} \) is a uniformly \((\lambda, \Lambda)\)-elliptic operator in the sense of \hyperref[A1]{\bf (A1)}, \( \Phi : \Omega \times [0, \infty) \to [0, \infty) \) is a continuous map whose gradient degeneracy and singularity are described in \hyperref[A3]{\bf (A3)}, \( f \) and \( g \) are sufficiently regular functions satisfying the conditions in \hyperref[A5]{\bf (A5)}, $ H: \mathbb{R}^n \times \Omega \to \mathbb{R} $ fulfilling the assumptions in \hyperref[A4]{\bf (A4)}, and \( \Omega \) is a \( C^2 \)-domain as specified in \hyperref[A6]{\bf (A6)}. We recall that, as a consequence of the Krylov-Safonov theory \cite{CC95}, viscosity solutions to the homogeneous equation
\[
F(D^2 u) = 0 \quad \text{in } \Omega,
\]
with \( F \) uniformly \((\lambda, \Lambda)\)-elliptic, belong to \( C^{1,\alpha_{0}}_{\mathrm{loc}}(\Omega) \) for some constant \( \alpha_{0} \equiv \alpha_{0}(n, \lambda, \Lambda) \in (0, 1) \).

\vspace{2mm}

When Hamiltonian terms $ H  $ vanishes, equation \eqref{Main:eq1}   incorporates an inhomogeneous degenerate or singular term, patterned after the integrand of an Uhlenbeck-type functional:
\begin{equation}\label{Section1:eq2}
    v \mapsto \int_{\Omega} G(|Dv|, x) \, dx
\end{equation}
with an integral density $G : \Omega \times [0, \infty) \to [0, \infty)$. From a variational viewpoint, \eqref{Section1:eq2} is a general non-autonomous functional that includes several prototypical models relevant to regularity theory, such as $ p$-growth \cite{GG82}, $ p(x)$-growth \cite{AM01, AM05}, Orlicz growth \cite{BBL21}, $ (p,q)$-double phase \cite{CM15a, CM15b}, $ (p(x), q(x))$-double phase \cite{BL21}, borderline case of double phase \cite{BCM15}, and the references therein.

  \vspace{2mm}

Non-variational counterparts to these models take the form of degenerate fully nonlinear equations. A well-known prototype is
\begin{equation}\label{Section1:eq3}
|Du|^{p} F(D^{2}u) = f   \quad  \text{in}  \quad \Omega.
\end{equation}
The pioneering work was carried out by Imbert and Silvestre in \cite{IS13}, where they established interior $C^{1,\alpha}$ regularity for viscosity solutions to \eqref{Section1:eq3}. Subsequently, Ara\'{u}jo et al. \cite{ART15} proved optimal interior $C^{1,\alpha}$ regularity under the assumption that $F$ is convex or concave. Recently, the local $ C^{1, \alpha} $ regularity of viscosity solution to \eqref{Section1:eq3} has been the subject of intensive study, such as $ (p,q)$-double phase degeneracy \cite{De21}, $ (p(x), q(x))$-double phase degeneracy \cite{FRZ21}, $ p(x) $-degeneracy \cite{BPRT20}, general $ \Phi $-degeneracy \cite{BBLL24a} and so on. It is also worth mentioning that in the recent work \cite{HJMZ26}, the authors examined sharp interior $ C^{1,\alpha} $ regularity of solutions to fully nonlinear degenerate or singular equations with Hamiltonian terms
\begin{equation}\label{Section1:eq4}
  \Phi(|Du|,x) F(D^{2}u,x) + H(|Du|, x) = f   \quad  \text{in}  \quad \Omega.
\end{equation}
For further results on interior regularity, the interested reader is referred to \cite{APPT22, AN25a, KL26, N25, WJ26a, AN25b}.

  \vspace{2mm}

Much progress has been made in understanding the boundary regularity of solution to \eqref{Main:eq1} under the Dirichlet boundary condition. When problem \eqref{Main:eq1} has no Hamiltonian term and exhibits only a special degenerate structure, the global $ C^{1,\alpha} $ regularity results for degenerate fully nonlinear equations were developed in \cite{AS23, BD14} for $ \Phi(t,x) = t^{p} $ with $ p \geq 0 $, in \cite{BSRR23} for $\Phi(t,x) = t^{p(x)} + a(x)t^{q(x)} $ with $ 0 \leq p(\cdot) \leq q(\cdot) $, and in \cite{BBLL24b} for general $ \Phi $. Very recently, Wang-Jiang in \cite{WJ26c} established global $ C^{1} $ regularity of solutions to degenerate equation with form   
\begin{equation*}
   \big[\sigma_{1}(|Du|)+a(x)\sigma_{2}(|Du|)\big] F(D^{2}u) = f        \quad   \text{in}  \quad  \Omega
\end{equation*}
provided that $ \sigma_{2} : \mathbb{R}_{+} \rightarrow \mathbb{R}_{+} $ is a function whose inverse has a Dini-continuous modulus of continuity near the origin. To the best of our knowledge, higher order regularity results, particularly on global $ C^{1, \alpha}$ regularity, was still unknown for the above-mentioned fully nonlinear elliptic equation with Hamiltonian terms, let alone for those with degeneracy as \eqref{Main:eq1}.

In view of the above motivation and the previous results, the following problem naturally emerges:   

\begin{problem}\label{Section1:problem1}
Can we establish the optimal global $ C^{1, \alpha} $ regularity for viscosity solutions of \eqref{Main:eq1}?  
\end{problem}

To address Problem~\ref{Section1:problem1}, we first introduce the structural assumptions needed to state our main result:  

\label{A1} {\bf (A1)} {(Uniform ellipticity of $ F$).}
$F$ is uniformly elliptic with $F(\mathrm{O}_n, x) = 0$; namely, there exist constants $0 < \lambda \leq \Lambda$ such that, for all $\mathrm{M}, \mathrm{N} \in \mathrm{Sym}(n)$,
\[
  \mathscr{P}^{-}_{\lambda,\Lambda}(\mathrm{M}-\mathrm{N}) \leq F(\mathrm{M}, x) - F(\mathrm{N}, x)
  \leq \mathscr{P}^{+}_{\lambda,\Lambda}(\mathrm{M}-\mathrm{N}),
\]
where $ x \in \Omega$ and $\mathscr{P}^{\pm}$ denote the Pucci extremal operators \cite{CC95}, i.e.,
\[
  \mathscr{P}^{+}_{\lambda,\Lambda}(\mathrm{X})
  := \lambda \sum_{e_{i}<0} e_{i}(\mathrm{X}) + \Lambda \sum_{e_{i}>0} e_{i}(\mathrm{X}),
  \quad \text{and} \quad
  \mathscr{P}^{-}_{\lambda,\Lambda}(\mathrm{X})
  := \lambda \sum_{e_{i}>0} e_{i}(\mathrm{X}) + \Lambda \sum_{e_{i}<0} e_{i}(\mathrm{X});
\]

\label{A2} {\bf (A2)} {(H\"older continuity of $ F$ in the spatial variable).} We assume a uniform continuity on the coefficients of $ F $, namely, there exists a constant $ \mathrm{C}_{F} > 0 $ and $ 0 < \theta_{F} < 1 $ such that
\begin{equation*}
  \mathrm{osc}_{F}(x,y):= \sup_{\mathrm{M} \in \mathrm{Sym}(n)\setminus\{0\}} \frac{|F(\mathrm{M}, x)-F(\mathrm{M}, y)|}{||\mathrm{M}||} \leq \mathrm{C}_{F} |x-y|^{\theta_{F}},
\end{equation*}
for all $ x, y \in \Omega $.

\label{A3} {\bf (A3)} {(The monotonicity of $ \Phi$).} The function \( \Phi : [0, \infty) \times \Omega \to [0, \infty) \) is a continuous map satisfying the following properties:   

(i) there exist constants \( s(\Phi) \geq i(\Phi) > -1 \) such that the map \( t \mapsto \frac{\Phi(t,x)}{t^{i(\Phi)}} \) is almost non-decreasing with constant \( L \geq 1 \) in \( (0, \infty) \) in the sense that
\[\frac{\Phi(t,x)}{t^{i(\Phi)}} \leq L \frac{\Phi(s,x)}{s^{i(\Phi)}} \quad \text{whenever} \quad 0 < t \leq s < \infty \text{ and } x \in \Omega,\]
and the map \( t \mapsto \frac{\Phi(t,x)}{t^{s(\Phi)}} \) is almost non-increasing with constant \( L \geq 1 \) in \( (0, \infty) \) in the sense that
\[L \frac{\Phi(t,x)}{t^{s(\Phi)}} \geq \frac{\Phi(s,x)}{s^{s(\Phi)}} \quad \text{whenever } 0 < t \leq s < \infty \text{ and } x \in \Omega;\]

(ii) there exist constants \( 0 < \nu_0 \leq \nu_1 \) such that \( \nu_0 \leq \Phi(1,x) \leq \nu_1 \) for all \( x \in \Omega \).  

\label{A4} {\bf (A4).} {(The growth of Hamiltonian term $ H$).} \( H : \mathbb{R}^n \times \Omega \to \mathbb{R} \) is continuous and there exist constants \( \mathcal{M}_{1}, \mathcal{M}_{2} > 0 \) and \( 0 < m \leq 1 + i(\Phi) \) such that
\[|H(t,x)| \leq \mathcal{M}_{1} + \mathcal{M}_{2}|t|^m\]
for every \( t \in \mathbb{R}^n, x \in \Omega \).

\label{A5} {\bf (A5).} {(The regularity of $ f $ and $ g$).} The source term \( f \) belongs to \( C^{0}(\Omega) \cap L^\infty(\Omega) \) and $ g \in C^{1, \beta_{g}}(\partial \Omega) $ for some $ 0 < \beta_{g} < 1 $.

\label{A6} {\bf (A6).} { (The regularity of $ \Omega$).} $ \Omega \subset \mathbb{R}^{n} $ is a bounded $ C^{2}$-domain.  

The assumption \hyperref[A6]{\bf (A6)} was inspired by the approach introduced in \cite{BD14, BBLL24b}. Specifically, we may assume that $ 0 \in \partial \Omega $, and there exists a ball $ B_{1} \subset \mathbb{R}^{n} $ and $ \varphi \in C^{2}(\mathbb{R}^{n-1}) $ such that $ \varphi(0) = 0$, $  \nabla \varphi(0) = 0 $ and
\begin{equation*}
  \Omega \cap B_{1}  \subset \{ y \in B_{1}: y_{n} > \varphi(y')   \},  \quad    \partial \Omega \cap B_{1} =  \{ y \in B_{1}: y_{n} = \varphi(y')   \}.
\end{equation*}

The first main result is the following global gradient estimate.

\begin{theorem}[{Global $ C^{1, \alpha} $ regularity}]   
\label{Thm1}
Suppose that the assumptions \hyperref[A1]{\bf (A1)}--\hyperref[A6]{\bf (A6)} hold. Let $ u \in C^{0}(\overline{\Omega}) $ be a viscosity solution to \eqref{Main:eq1}, and $ \alpha $ be chosen to satisfy
\begin{equation*}
\alpha \in \left\{
     \begin{alignedat}{2}
         & (0,\alpha_{0}) \cap \bigg(0, \frac{1}{1+s(\Phi)}\bigg] \cap (0,\beta_{g})           \quad  &&  \text{if} \quad i(\Phi) \geq 0    ,      \\
          &   (0,\alpha_{0}) \cap \bigg(0, \frac{1}{1+s(\Phi)-i(\Phi)}\bigg] \cap (0,\beta_{g}) \quad  &&  \text{if} \quad  -1 < i(\Phi) <0.       \\
     \end{alignedat}
     \right.
\end{equation*}
Then $ u \in C^{1,\alpha}(\overline{\Omega}) $. More precisely,

$ (i) $ if $ 0< m < 1+ i(\Phi) $, then
\begin{equation*}
  ||u||_{C^{1,\alpha}(\overline{\Omega})} \leq \mathrm{C} \bigg( 1+ ||u||_{L^{\infty}(\Omega)}  +  ||g||_{C^{1,\beta_{g}}(\partial \Omega)}  + \bigg(\frac{\mathcal{M}_{2}}{\nu_{0}} \bigg)^{\frac{1}{1+i(\Phi)-m}} +  \bigg(\frac{||f||_{L^{\infty}(\Omega)}+\mathcal{M}_{1}}{\nu_{0}}\bigg)^{\frac{1}{1+i(\Phi)}} \bigg)
\end{equation*}
where the constant $ \mathrm{C} $ depends on $ n, \lambda, \Lambda, i(\Phi), L, \alpha, m, \theta_{F} $ and $ C_{F} $.

$ (ii) $ if $ m = 1+ i(\Phi) $, then
\begin{equation*}
  ||u||_{C^{1,\alpha}(\overline{\Omega})} \leq \mathrm{C} \bigg( 1+ ||u||_{L^{\infty}(\Omega)} +  ||g||_{C^{1,\beta_{g}}(\partial \Omega)}   \bigg).
\end{equation*}
where the constant $ \mathrm{C} $ depends on $ \nu_{0}, ||f||_{L^{\infty}(\Omega)}, \mathcal{M}_{1} $ and $ \mathcal{M}_{2} $.
\end{theorem}

Next, if the operator $ F $ is convex (or concave) and $ g \in C^{1,1}(\partial \Omega) $, by the classical Evans-Krylov theory \cite{Eva82, K82, K83}, one obtains
\begin{corollary}\label{Section1:coro1}
Suppose the assumptions of Theorem \ref{Thm1} are in force. Assume further the operator $ F $ is convex (or concave) and $ g \in C^{1,1}(\partial \Omega) $. Then $ u \in C^{1,\alpha}(\overline{\Omega})$, where
\begin{equation*}
\alpha = \left\{
     \begin{alignedat}{2}
         & \frac{1}{1+s(\Phi)}            \quad  &&  \text{if} \quad i(\Phi) \geq 0    ,         \\
          &   \frac{1}{1+s(\Phi)-i(\Phi)}  \quad  &&  \text{if} \quad  -1 < i(\Phi) <0.         \\
     \end{alignedat}
     \right.
\end{equation*}
\end{corollary}

\begin{remark}
Theorem~\ref{Thm1} applies, in particular, to several concrete degenerate models with Hamiltonian terms. Typical examples include
\[
\big[|Du|^{p}+a(x)|Du|^{q}\big]F(D^{2}u,x)
+h(x)|Du|^{m}=f(x),
\]
\[
\big[|Du|^{p}+a(x)|Du|^{p}\log(1+|Du|)\big]F(D^{2}u,x)
+h(x)|Du|^{m}=f(x),
\]
and, more generally,
\[
\big[\Phi(|Du|)+a(x)\Psi(|Du|)\big]F(D^{2}u,x)
+h(x)|Du|^{m}=f(x),
\]
where $\Phi$ and $\Psi$ are $N$-functions exhibiting Orlicz-type growth. These examples illustrate that our result covers power-type, logarithmically perturbed, and Orlicz-type degeneracy laws in the presence of Hamiltonian terms.
\end{remark}

We next revisit the boundary regularity problem from a different
perspective. Inspired by the oscillation-based approach developed
in \cite{AS23}, we ask whether this method can be adapted to equations
with variable double-phase degeneracy and Hamiltonian terms. This
would not only extend the corresponding results in
\cite{AS23, BSRR23}, but also provide an alternative approach to the
boundary regularity obtained from Theorem~\ref{Thm1} for this
particular class of equations. This motivates the following problem:

\begin{problem}\label{Section1:problem2}
Can the oscillation-based approach of \cite{AS23} be adapted to obtain
sharp boundary regularity for degenerate equations with variable
double-phase exponents and Hamiltonian terms?
\end{problem}

More precisely, we consider the concrete model
\begin{equation}\label{Section1:eq5}
\left\{
     \begin{alignedat}{2}
          \big[|Du|^{p(x)}+ a(x)|Du|^{q(x)}\big]F(D^{2}u,x)+ h(x)|Du|^{m} & = f(x)         \quad   &&   \text{in} \ \ \Omega     ,    \\
          u(x)   &  = g(x)   \quad  &&   \text{on} \ \   \partial \Omega,      \\
     \end{alignedat}
     \right.
\end{equation}
under the following assumptions:

\label{A3a} {\bf (A3a).} $ p(x), q(x) \in C^{0}(\Omega) $, $ 0 < p(x) \leq q(x) $ and $ p_{\mathrm{min}}:=\min_{x \in \overline{\Omega}} p(x)     $, $ p_{\mathrm{max}}:=\max_{x \in \overline{\Omega}} p(x)     $;

\label{A4a} {\bf (A4a).} $ h(x) \in C^{0}(\Omega) $, $ 0 \leq a(x) \in C^{0}(\Omega) $, and $ 0 < m \leq 1+  p_{\mathrm{min}} $;

\label{A5a} {\bf (A5a).} We suppose that $ f $ is H\"older continuous and $ f(0) = 0 $, namely,
\begin{equation*}
   |f(x)| \leq \mathcal{K}_{2}|x|^{\theta_{f}}, \quad  x \in \Omega
\end{equation*}
holds for some $   \mathcal{K}_{2} > 0 $ and $ \theta_{f} \in (0,1) $. Moreover, $ g \in C^{1, \beta_{g}}(\partial \Omega) $ for some $ 0 < \beta_{g} < 1 $.

Before providing the boundary regularity for \eqref{Section1:eq5}, we denote $ \Omega_{r}^{+}(x):= \Omega \cap B_{r}(x) $, $ \Omega'_{r}(x):= \partial \Omega \cap B_{r}(x) $. In particular, $ \Omega_{r}^{+}:= \Omega \cap B_{r} $ and $ \Omega'_{r}:= \partial \Omega \cap B_{r} $.

\begin{theorem}[{Boundary $ C^{1,\beta'}$ regularity}]
\label{Thm2}
Suppose that the assumptions \hyperref[A1]{\bf (A1)}, \hyperref[A2]{\bf (A2)}, \hyperref[A3a]{\bf (A3a)}, \hyperref[A4a]{\bf (A4a)} and \hyperref[A5a]{\bf (A5a)} hold. Let $ u $ be a viscosity solution to \eqref{Section1:eq5}. Then for
\begin{equation}\label{Section1:eq6}
  \beta'=\min \bigg\{ \beta_{g}, \frac{1+\theta_{f}}{1+p_{\mathrm{max}}}, \frac{1}{1+p_{\mathrm{max}}-m} \bigg\},
\end{equation}
and for some constant $ \mathrm{C} $, it holds
\begin{equation}\label{}
  |u(x)-u(x_{0}) - Du(x_{0})\cdot (x-x_{0})| \leq \mathrm{C} A'|x-x_{0}|^{1+\beta'}
\end{equation}
for $ x \in  \Omega_{1}^{+} $ and $ x_{0} \in \Omega_{1}^{'} $, where
\[
A' :=
\begin{cases}
1 + \|u\|_{L^\infty(\Omega_{1}^{+})} +  ||h||_{L^\infty(\Omega_{1}^{+})}^{\frac{1}{1+p_{\mathrm{min}}-m}}        +  \mathcal{K}_{2}^{\frac{1}{1+p_{\mathrm{min}}}} +   ||g||_{C^{1,\beta_{g}}(\Omega_{1}^{'})}  & \text{for } \ m < 1 + p_{\mathrm{min}}, \\[1em]
1 + \|u\|_{L^\infty(\Omega_{1}^{+})} + ||g||_{C^{1,\beta_{g}}( \Omega_{1}^{'})} & \text{for } \ m = 1 + p_{\mathrm{min}},
\end{cases}
\]
\end{theorem}

\begin{remark}
In fact, arguing as in \cite[Section 4]{AS23}, we can also obtain the sharp global $ C^{1, \beta} $ regularity of viscosity solution to \eqref{Section1:eq5} for $ \beta=\min \big\{\beta', \alpha_{0}^{-}    \big\} $.
\end{remark}

Hereafter, we denote the {\it singular region} of the existing solutions by
\begin{equation*}
  \mathcal{S}_{r,\gamma'}(u, \Gamma'):= \{x_{0} \in \Gamma' \Subset \Omega: |Du(x_{0})| \leq  r^{\gamma'}, \text{for} \quad  0 \leq r \ll 1  \quad \text{and} \quad  \gamma' \in (0,1)   \}.
\end{equation*}

A direct application of Theorems~\ref{Thm1} and~\ref{Thm2} yield the interior $ C^{1, \widehat{\beta}}(\Omega) $ result for problem~\eqref{Section1:eq5}, where
$$ \widehat{\beta} = \min \left\{\alpha_{0}^{-}, \frac{1+\theta_{f}}{1+p_{\mathrm{max}}} \right\} .$$
More precisely, if $ x_{0} \in  \mathcal{S}_{r, \widehat{\beta}}(u, \Gamma') $, then near $ x_{0} $ there holds
\begin{equation*}
  \sup_{x \in B_{r}(x_{0})} |u(x)| \leq |u(x_{0})|  +  \mathrm{C} r^{1+\widehat{\beta}}.
\end{equation*}

Nevertheless, from a geometric viewpoint, it is qualitatively essential for deriving the sharp lower bound estimate--the counterpart to the upper bound--for such PDEs with nonhomogeneous degeneracy. This is precisely the {\it non-degeneracy property} of solutions. Thus, a natural question arises:
\begin{problem}\label{Section1:problem3}
Under the assumption $\inf_{x\in\Omega} f(x)>0$, can we establish a quantitative non-degeneracy estimate for solutions to~\eqref{Section1:eq5} in the singular region?
\end{problem}

The following theorem provides an affirmative answer to
problem~\ref{Section1:problem3}.

\begin{theorem}[{Non-degeneracy}]
\label{Thm3}
Suppose that the assumptions \hyperref[A1]{\bf (A1)}, \hyperref[A2]{\bf (A2)}, \hyperref[A3a]{\bf (A3a)}, \hyperref[A4a]{\bf (A4a)} hold. Let $ u $ be a viscosity solution to
\begin{equation}\label{Section1:introthm3}
   [|Du|^{p(x)}+ a(x)|Du|^{q(x)}] F(D^{2}u,x)+ h(x)|Du|^{m} = f(x)        \quad    \text{in} \quad   \Omega
\end{equation}
with \( f \in C^{0}(\Omega) \cap L^\infty(\Omega) \) and $ \inf_{x \in \Omega} f(x) > 0 $. Given $ x_{0} \in \mathcal{S}_{r,\gamma'}(u, \Gamma') $, there exists a constant $ \mathrm{c}=\mathrm{c}(m, n, \lambda, \Lambda, p_{\mathrm{min}}, q_{\mathrm{min}}, ||a||_{L^{\infty}(\Omega)}, ||h||_{L^{\infty}(\Omega)})$ such that
\begin{equation*}
  u(x_{0}) + \mathrm{c}r^{\frac{2+p_{\mathrm{min}}}{1+p_{\mathrm{min}}}}  \leq \sup_{\partial B_{r}(x_{0})} u(x)  \quad  \text{for all} \quad   r \in (0, 2^{-1}),
\end{equation*}
provided
\begin{equation*}
  \frac{2+p_{\mathrm{min}}}{1+p_{\mathrm{min}}} \leq \frac{2+p_{\mathrm{min}}-m}{1+p_{\mathrm{max}}-m}
\end{equation*}
holds.
\end{theorem}

\begin{remark}
Such quantitative information is pivotal for addressing various analytic and geometric problems, particularly in the study of blow-up procedures and related weak geometric and free boundary analysis (see \cite{SJW25, WJ26b}). To the best of our knowledge, this work provides the first non-degeneracy result for nonlinear degenerate models with the Hamiltonian terms in the existing literature. Furthermore, Theorem~\ref{Thm3} recovers the conclusion of \cite[Theorem 1.4]{BSRR23} as a special case by setting $ h(x) = 0 $.
\end{remark}

\subsection{Ideas of proof of Theorems \ref{Thm1}--\ref{Thm3}}

We first notice that establishing the boundary $ C^{1, \alpha} $ regularity of the viscosity solution $ u$ to \eqref{Main:eq1} is equivalent to showing that the graph of $ u $ can be approximated by an affine function with an error bounded by $ \mathrm{C}r^{1+\alpha} $ on every $ {B_{r}\cap \{y_{n}>\varphi(y')\}} $. Our method builds on the flatness improvement scheme of \cite{ART15}. However, the interplay between the general degeneracy $  \Phi(|Du|, \cdot)$ and the growth of the Hamiltonian term $ H(|Du|, \cdot) $ creates a delicate balance and calls for a new analysis tailored to our setting. For this purpose, we describe this idea step by step as follows:

\emph{Step 1-a}. We first characterize the boundary behavior of viscosity solutions to \eqref{Main:eq1} in terms of $ C^{2} $-distance function. To achieve this, we need a general  comparison principle (Proposition \ref{Section3:prop1}) as well as a delicate construction of auxiliary functions (Lemma~\ref{Section3:lemma2}). Based on this, we attempt to obtain the compactness of viscosity solutions to the perturbed problem
\begin{equation}
\label{Section1:pert1}
\left\{
     \begin{alignedat}{2}
         \Phi(|Du+\xi|,y) F(D^{2}u,y)+ H(|Du+\xi|, y)  & = f         \quad   &&   \text{in} \quad  B_{1} \cap \{y_{n} > \varphi(y')\}     ,    \\
          u   &  = g    \quad  &&   \text{on} \quad  B_{1} \cap \{y_{n} = \varphi(y')\} ,        \\
     \end{alignedat}
     \right.
\end{equation}
where $ \xi \in \mathbb{R}^{n} $ is an arbitrary vector. To this end, it is essential to elucidate how the competitive interplay between the degeneracy law $  \Phi(|Du+\xi|,\cdot)$ and the Hamiltonian growth $ H(|Du+\xi|, \cdot) $ modulates the elliptic structure of \eqref{Section1:pert1}. In the small-gradient regime, the equation \eqref{Section1:pert1} is primarily governed by its degenerate profile, rendering it amenable to the techniques developed in \cite{IS13}. Conversely, away from the critical set $ \mathcal{C}(u)$, the equation qualitatively transitions into a uniformly elliptic with a gradient-dependent lower-order term:
\begin{equation}
\label{Section1:pert2}
  F(D^2u, y) + \frac{H(|Du+\xi|, y)}{\Phi(|Du+\xi|, y)} = \frac{f(y)}{\Phi(|Du+\xi|, y)},
\end{equation}
where
\begin{equation*}
  \frac{H(|Du+\xi|, y)}{\Phi(|Du+\xi|, y)} \lesssim |Du+\xi|^{m-i(\Phi)},   \quad    \text{and}   \quad \frac{f(y)}{\Phi(|Du+\xi|, y)}  \lesssim  |Du+\xi|^{-i(\Phi)}.
\end{equation*}

This perspective underscores the analytical complexity of controlling the Hamiltonian growth amidst gradient perturbations. To address this challenge, we implement systematic scaling and approximation frameworks inspired by \cite{AN25b}.

In the regime where $m \leq i(\Phi)$, the heuristic form \eqref{Section1:pert2} suggests that the equation maintains a predominantly elliptic character for large gradients, allowing us to provide boundary H\"{o}lder regularity of solutions to \eqref{Section1:pert1} via Crandall-Ishii-Lions Lemma (\cite[Theorem 3.2]{CIL92}). Conversely, when $i(\Phi) < m \leq 1 + i(\Phi)$, additional structural condition \eqref{Section3:eq8} controls the growth of the Hamiltonian term $ \mathcal{M}_{2}|Du+\xi|^{m-i(\Phi)} $, which in turn allows us to obtain boundary Lipschitz regularity of solutions to \eqref{Section1:pert1} for the small-gradient regime by combining Lemma~\ref{Section3:lemma6} with Ishii-Lions's method.

\emph{Step 1-b}. Building on the compactness result ({\it Step 1-a}), we establish a boundary version of the approximation lemma in the degenerate setting ($i(\Phi) \geq 0$), namely, the solution to \eqref{Section1:pert1} can be approximated by the function $ h$ of homogeneous $ (\lambda, \Lambda)$-uniform elliptic problem. In the end, by means of the scaling technique and the iteration method, we obtain the boundary $ C^{1, \alpha'} $ regularity of solutions to \eqref{Section1:pert1}.

\emph{Step 1-c}. For the singular case ($ -1 < i(\Phi) < 0 $), the main result of Theorem \eqref{Thm1} follows from the degenerate case once we establish that the boundary Lipschitz regularity of viscosity solution $ u$ of \eqref{Main:eq1}; see Proposition~\ref{Section3:Prop4} below.

The proof of Theorem~\ref{Thm2} mainly relies on the gradient oscillation estimates for degenerate elliptic equations established in \cite{ATU17, ATU18, AS23, WYJ25, N25}. We organize our analysis into two parts as follows:

\emph{Step 2-a}. We first treat boundary points with small gradient at the relevant scale: Theorem~\ref{Section5:Thm1} establishes the desired regularity whenever, at a given boundary point, the gradient is sufficiently small on the relevant scale. As in the preceding arguments, the proof hinges on tangential techniques.

\emph{Step 2-b}. If the gradient $ |Du| $ at a boundary point $ x_{0}$ exhibits the behavior prescribed by Theorem~\ref{Section5:Thm1}, then the conclusion is immediate. Otherwise, the boundary point $ x_{0}$ lies at a quantifiable distance from the critical set $ \mathcal{C}(u)$, then after a rescaling tailored to the value of $ |Du| $ at the point, equation \eqref{Section1:eq5} becomes uniformly elliptic on a domain of fixed size, allowing us to invoke the optimal regularity estimates for uniformly elliptic equations established in \cite[Theorem 2.1]{AS23}.

The proof of Theorem~\ref{Thm3} depends crucially on the general comparison principle. In particular, we need to prove that the auxiliary function $ \Theta(x):= \mathrm{C}_{0} |x|^{\beta''} $, where $ \beta'':=\frac{2+p_{\mathrm{min}}}{1+p_{\mathrm{min}}} $, is a super-solution of \eqref{Section1:introthm3}; the choice of the constant $ \mathrm{C}_{0}$ is the key point here.

\subsection{Main contributions of the paper}
The primary objective of this work is to extend the non-variational elliptic regularity theory to a genuinely degenerate framework for elliptic equations with Hamiltonian terms. The main contributions of this manuscript are summarized below:

\begin{enumerate}[label=(\roman*)]

\item Theorem~\ref{Thm1} provides a unified treatment of global regularity for equations of this type, embracing the frameworks of \cite{BBLL24b, BSRR23, BD14}, and can also be considered as a boundary counterpart of \cite{AN25b, ART15, BBLL24a, BPRT20, De21, FRZ21, HJMZ26, IS13};

\vspace{2mm}

\item Theorem~\ref{Thm2} extend regarding non-variational scenario, previous results from \cite{AS23, N25}, and to some extent, we employ an improved oscillation-estimate method rather than the approximation approach \cite{BSRR23} and the flatness improvement program \cite{BD14};

\vspace{2mm}

\item To the best of our knowledge, Theorem~\ref{Thm3} provides the first non-degeneracy result for nonlinear degenerate models with the Hamiltonian terms in the existing literature. Furthermore, Theorem~\ref{Thm3} also recovers the conclusion of \cite[Theorem 1.4]{BSRR23} as a special case by setting $ h(x) = 0 $;

\vspace{2mm}

\item As two striking applications of Corollary~\ref{Section1:coro1}, under a suitable viscosity curvature condition on the level sets of the solution, we establish global $ C^{1,\frac{1}{3}} $ regularity for infinity-Poisson and global $ C^{1, \frac{1}{p-1}} $ regularity for the $ p$-Poisson equation ($ p >2$), thereby offering a promising path toward two long-standing open problems.   
\end{enumerate}

\subsection{Structure of the paper}
The remainder of this paper is organized as follows. Section \ref{Section 2} presents the definition of viscosity solutions to \eqref{Main:eq1} and several auxiliary lemmas. In Section \ref{Section 3}, we establish the boundary H\"{o}lder regularity ($ i(\Phi) \geq 0 $) and Lipschitz regularity ($ -1< i(\Phi) <0, \xi=0 $) of viscosity solution to perturbation equation \eqref{Section3:eq1} by using scaling technique, barrier function, comparison principle and Ishii-Lions method. Section~\ref{Section 4} is devoted to presenting the proof of Theorem~\ref{Thm1}. Sections~\ref{Section 5} and~\ref{Section 6} are concerned to prove the boundary $ C^{1, \beta'} $ (Theorem~\ref{Thm2}) regularity and non-degeneracy property (Theorem~\ref{Thm3}) of solutions to \eqref{Section1:eq5}. Finally, in Section~\ref{Section 7}, we are devoted to presenting two relevant applications of our findings; see Theorems~\ref{Section 7:thm1} and \ref{Section 7:thm2} for details.

\subsection{Notations}
We summarize below the basic notation used throughout the paper.
\begin{itemize}

\item For $ r >0 $, $ B_{r}(x) $ denotes the open ball of radius $ r $ centered at $ x $. We simply use $ B_{r} $ to denote the open ball $ B_{r}(0)$.

\item  We denote $ \Omega_{r}^{+}(x):= \Omega \cap B_{r}(x) $, $ \Omega'_{r}(x):= \partial \Omega \cap B_{r}(x) $. In particular, $ \Omega_{r}^{+}:= \Omega \cap B_{r} $ and $ \Omega'_{r}:= \partial \Omega \cap B_{r} $.

\item $ \textbf{I}\mathrm{d_{n}} $ denotes the $n\times n$ identity matrix.

\item $ \text{Sym}(n) $ denotes the space of all $ n \times n $ symmetric matrices in $ \mathbb{R}^{n} $.

\item $ f  \lesssim  g $ denotes there exists a constant $ \mathrm{C} > 0 $ such that $  f \leq \mathrm{C}g $.

\item $ \mathrm{diam}(\Omega) $ denotes the diameter of domain $ \Omega $.

\item $ \mathrm{C} $ shall denote a generic positive constant which may vary in different inequalities.
\end{itemize}

{\bf Acknowledgments.} F. Jiang has been supported by the National Natural Science Foundation of China (No. 12271093) and the Jiangsu Provincial Scientific Research Center of Applied Mathematics (Grant No. BK20233002), and Shanghai Institute for Mathematics and Interdisciplinary Sciences (SIMIS) under grant number SIMIS-ID-2025-AD.

\section{Preliminaries}
\label{Section 2}
In this section, we first review the definition of a viscosity solution to equation \eqref{Main:eq1}, and then present several auxiliary lemmas to be used in the proof of Theorems~\ref{Thm1}--\ref{Thm2}: the Alexandroff-Bakelman-Pucci (ABP) estimate, the Cutting lemma, and a stability result for problem \eqref{Main:eq1}.

\subsection{The definition of viscosity solution} In what follows, we shall focus on viscosity solutions of the equation
\begin{equation}\label{Section 2:eq00}
  G(x, Du, D^{2}u):= \Phi(|Du|,x) F(D^{2}u,x)+ H(|Du|, x) - f  \quad   \text{in}   \quad   \Omega.
\end{equation}
We now provide the following definition of a viscosity solution $ u $ of \eqref{Section 2:eq00}:

\begin{definition}\label{subsection 2.1}
A function $ u \in C^{0}(\Omega) $ is a viscosity super-solution to \eqref{Section 2:eq00} if for $ x_{0} \in \Omega $,

\begin{itemize}

\item either for any $ \varphi \in C^{2}(\Omega) $ such that $ u-\varphi $ has a local minimum at $ x_{0} $ and $ D \varphi(x_{0}) \neq 0 $, one has
\begin{equation*}
  G(x, D\varphi(x_{0}), D^{2}\varphi(x_{0}))  \leq 0;
\end{equation*}

\item or there exists a $ \delta > 0 $ such that $ v$ is constant in $ B_{\delta}(x_{0}) $ and $ 0 \leq f(x) $ for all $ x \in B_{\delta}(x_{0}) $.
\end{itemize}

Similarly, $ u \in C^{0}(\Omega) $ is called is a viscosity sub-solution of \eqref{Section 2:eq00} if for $ x_{0} \in \Omega $,

\begin{itemize}

\item either for any $ \varphi \in C^{2}(\Omega) $ such that $ u-\varphi $ has a local maximum at $ x_{0} $ and $ D \varphi(x_{0}) \neq 0 $, one has
\begin{equation*}
  G(x, D\varphi(x_{0}), D^{2}\varphi(x_{0}))  \geq 0;
\end{equation*}

\item or there exists a $ \delta > 0 $ such that $ v$ is constant in $ B_{\delta}(x_{0}) $ and $ 0 \geq f(x) $ for all $ x \in B_{\delta}(x_{0}) $.
\end{itemize}

Finally, we say that $ u $ is a viscosity solution to \eqref{Section 2:eq00} if it is both a viscosity sub-solution and a viscosity super-solution.
\end{definition}

\subsection{Auxiliary lemmas}
In the sequel, we will present an Alexandroff-Bakelman-Pucci (ABP) estimate adapted to our context of fully nonlinear degenerate models with the Hamiltonian terms, see \cite[Theorem 1.3]{BFO26}. Such an estimate is pivotal for establishing universal bounds for viscosity solutions in terms of the problem data.

\begin{theorem}[{ABP estimate}]
\label{Section 2:thm1}
Assume that the assumptions \hyperref[A1]{\bf (A1)}--\hyperref[A2]{\bf (A2)} are in force. Let $ u \in C(\overline{\Omega}) $ be a viscosity solution to \eqref{Main:eq1}. Then if $ 0 < m \leq 1+i(\Phi)$, there exist
constants
\begin{equation*}
  \mathrm{C}_{1} = \mathrm{C}_{1}(n, i(\Phi), s(\Phi), m, \lambda, \Lambda, L, \nu_{0}, |\Omega|, \mathcal{M}_{1}, \mathcal{M}_{2}) \geq 1, \quad \mathrm{C}_{2} = \mathrm{C}_{2}(n, s(\Phi), m) >0
\end{equation*}
and
\begin{equation*}
  \mu:= \frac{n(1+s(\Phi))}{m}
\end{equation*}
such that, for $ \mu_{0} = (\min\{1, ||f||_{L^{n}(\Omega)}^{\mu}\})^{-1} $, we have
\begin{equation*}
  ||u||_{L^{\infty}(\Omega)} \leq ||g||_{L^{\infty}(\partial \Omega)}  +   \mathrm{C}_{2} \mathrm{C}_{1}^{\mu_{0}} \mathrm{diam}(\Omega) \max \bigg\{||f||_{L^{n}(\Omega)}^{\frac{1}{m}} +  ||f||_{L^{n}(\Omega)}^{\frac{s(\Phi)+1}{m(i(\Phi)+1)}} \bigg\}.
\end{equation*}
In particular, if $ m = i(\Phi)+1 = s(\Phi) +1 $ or $ ||f||_{L^{n}(\Omega)} \geq 1 $, we have
\begin{equation*}
  ||u||_{L^{\infty}(\Omega)} \leq ||g||_{L^{\infty}(\partial \Omega)}   +   \mathrm{C}_{1} \mathrm{C}_{2} \mathrm{diam}(\Omega) ||f||_{L^{n}(\Omega)}.
\end{equation*}
\end{theorem}

The following result is the ``Cutting Lemma''. Although the degenerate term $ \Phi(\cdot, x)$ in the present setting is more general than \cite[Lemma 6]{IS13}, the argument developed for \cite[Lemma 6]{IS13} carries over to this situation.

\begin{lemma}[Cutting lemma]
\label{Section 2:lem1}
Suppose that the conditions \hyperref[A1]{\bf (A1)}--\hyperref[A3]{\bf (A3)} hold. Let $ u $ be a viscosity solution to
\begin{equation*}
  \Phi(|Du|,x) F(D^{2}u, x) = 0   \quad   \text{in}  \quad  B_{1}
\end{equation*}
with $ i(\Phi) \geq 0 $. Then $ u $ is a viscosity solution to
\begin{equation*}
  F(D^{2}u, x) = 0    \quad   \text{in}  \quad  B_{1}.
\end{equation*}
\end{lemma}

On the other hand, in the singular case $ -1 < i(\Phi) < 0 $, following the proof of \cite[Proposition 2.1]{BBLL24a}, one can prove the following assertion:

\begin{lemma}
\label{Section2:lemma33}
Suppose that the assumptions \hyperref[A1]{\bf (A1)}--\hyperref[A4]{\bf (A4)} hold. Let $ u $ be a viscosity solution to
\begin{equation*}
  \Phi(|Du|,x) F(D^{2}u,x)+ H(|Du|, x)   = f(x)    \quad   \text{in}  \quad   \Omega
\end{equation*}
with $ -1 < i(\Phi) < 0 $ and $  f \in C^{0}(\Omega) \cap L^{\infty}(\Omega) $. Then, $ u $ is a classical solution to
\begin{equation*}
  |Du|^{-i(\Phi)}\Phi(|Du|,x) F(D^{2}u,x)+ |Du|^{-i(\Phi)} H(|Du|, x)   = |Du|^{-i(\Phi)} f(x)    \quad   \text{in}  \quad   \Omega.
\end{equation*}
\end{lemma}

The following lemma plays a crucial role in the proof of Theorem \ref{Thm1}. To prove this, one can follow the lines of proof of \cite[Lemma 2.6]{BBLL24b} or \cite[Theorem 4.1]{BBLL24a}.

\begin{lemma}[{Stability}]
\label{Section2:lem2}
Let $ \{g_{j}\}_{j \in \mathbb{N}} $ be a sequence of Lipschitz continuous functions such that $ g_{j} \rightarrow g_{\infty} $. Assume that $ u \in C^{0}(B_{1} \cap \{y_{n}>\varphi(y')\})$ is a uniformly bounded viscosity solution to
\begin{equation*}
\left\{
     \begin{alignedat}{2}
        \Phi_{j}(|Du_{j}+\xi_{j}|,y) F_{j}(D^{2}u_{j},y)+ H_{j}(|Du_{j}+\xi_{j}|, y) & = f_{j}(y)        \quad  &&  \text{in} \ \ B_{1} \cap \{y_{n} > \varphi(y')\}    ,    \\
          u_{j}(y)  & = g_{j}(y)    \quad  && \text{on}   \ \  B_{1}  \cap \{y_{n} = \varphi(y')\},        \\
     \end{alignedat}
     \right.
\end{equation*}
where $ \{\xi_{j}\}_{j \in \mathbb{N}} \subset \mathbb{R}^{n} $, $ \{f_{j}\}_{j \in \mathbb{N}} \subset C^{0}(B_{1} \cap \{y_{n} > \varphi(y')\})$, and $ \{F_{j}\}_{j \in \mathbb{N}} \subset C^{0}(\mathrm{Sym}(n), \mathbb{R}) $ is uniformly $ (\lambda, \Lambda) $ elliptic. Suppose further that $ \xi_{j} \rightarrow \xi_{\infty} $, $ F_{j} \rightarrow F_{\infty} $, $ f_{j} \rightarrow 0 $(uniformly) and $ H_{j}\rightarrow 0 $(uniformly). Then one can extract a subsequence from $ \{u_{j}\}_{j \in \mathbb{N}} $ which converges uniformly to $ u_{\infty} $ on $ \overline{B_{r} \cap \{y_{n} > \varphi(y')\} } $ for any $ 0 <r <1 $. Furthermore, the limit $ u_{\infty} $ solves
\begin{equation*}
\left\{
     \begin{alignedat}{2}
       F(D^{2}u_{\infty}) & = 0       \quad  &&  \text{in} \ \ B_{r} \cap \{y_{n} > \varphi(y')\}    ,        \\
          u_{\infty}(y)  & = g_{\infty}(y)    \quad  && \text{on}   \ \  B_{r}  \cap \{y_{n} = \varphi(y')\}.        \\
     \end{alignedat}
     \right.
\end{equation*}
\end{lemma}

We finish this section by providing the interior regularity results shown in \cite[Theorem 1.1]{HJMZ26}.

\begin{theorem}\label{Section2:thm3}
Let $ u \in C^{0}(\Omega) $ be a viscosity solution of
\begin{equation*}
  \Phi(|Du|,x) F(D^{2}u,x)+ H(|Du|, x)  = f  \quad  \text{in}  \quad \Omega,
\end{equation*}
under the assumptions \hyperref[A1]{\bf (A1)}--\hyperref[A4]{\bf (A4)} with $ f \in L^{\infty}(\Omega) \cap C^{0}(\Omega)$. Then $ u \in C^{0,\widetilde{\beta'}}_{ \mathrm{loc}}(\Omega)$ for all $ \widetilde{\beta'} > 0 $ fulfilling
\begin{equation*}
\widetilde{\beta'} \in \left\{
     \begin{alignedat}{2}
         & (0,\alpha_{0}) \cap \bigg(0, \frac{1}{1+s(\Phi)}\bigg]           \quad  &&  \text{if} \quad i(\Phi) \geq 0    ,      \\
          &   (0,\alpha_{0}) \cap \bigg(0, \frac{1}{1+s(\Phi)-i(\Phi)}\bigg]  \quad  &&  \text{if} \quad  -1 < i(\Phi) <0.       \\
     \end{alignedat}
     \right.
\end{equation*}
More precisely, for any $ \Omega' \Subset \Omega $, there holds

$ (i) $ if $ 0< m < 1+ i(\Phi) $, then
\begin{equation*}
  ||u||_{C^{1,\widetilde{\beta'}}(\Omega')} \leq \mathrm{C} \bigg( 1+ ||u||_{L^{\infty}(\Omega)}  + \bigg(\frac{\mathcal{M}_{2}}{\nu_{0}} \bigg)^{\frac{1}{1+i(\Phi)-m}} +  \bigg(\frac{||f||_{L^{\infty}(\Omega)}+\mathcal{M}_{1}}{\nu_{0}}\bigg)^{\frac{1}{1+i(\Phi)}} \bigg)
\end{equation*}
where the constant $ \mathrm{C} $ depends on $ n, \lambda, \Lambda, \widetilde{\beta'}, m, \theta_{F}, L $ and $ i(\Phi) $;

$ (ii) $ if $ m = 1+ i(\Phi) $, then
\begin{equation*}
  ||u||_{C^{1,\widetilde{\beta'}}(\Omega')} \leq \mathrm{C} \big( 1+ ||u||_{L^{\infty}(\Omega)}    \big).
\end{equation*}
where the constant $ \mathrm{C} $ depends on $ \nu_{0}, ||f||_{L^{\infty}(\Omega)}, \mathcal{M}_{1} $ and $ \mathcal{M}_{2} $.
\end{theorem}

\section{Boundary H\"{o}lder regularity and Lipschitz regularity}\label{Section 3}

In this section, to obtain further Lipschitz estimate up to the boundary as in \cite{BBLL24b, BSRR23}, we shall consider a bounded viscosity solution to
\begin{equation}
\label{Section3:eq1}
\left\{
     \begin{alignedat}{2}
         \Phi(|Du+\xi|,y) F(D^{2}u,y)+ H(|Du+\xi|, y)  & = f         \quad   &&   \text{in} \quad  B_{1} \cap \{y_{n} > \varphi(y')\}     ,    \\
          u   &  = g    \quad  &&   \text{on} \quad  B_{1} \cap \{y_{n} = \varphi(y')\} ,        \\
     \end{alignedat}
     \right.
\end{equation}
where $ \xi \in \mathbb{R}^{n} $ and $ \varphi $ as defined in Section \ref{Intro}. Firstly, we exploit the scaling properties of the shifted operator in \eqref{Section3:eq1} to reduce the problem to a smallness regime.

\begin{remark}[{Smallness regime}]\label{Section3:rmk1}
In the proof of Theorem \ref{Thm1}, we can confine the proof to
\begin{align*}
  & ||u||_{L^{\infty}(B_{1}\cap \{y_{n}>\varphi(y')\})} \leq 1, \quad ||g||_{C^{1,\beta_{g}}(B_{1} \cap \{y_{n} = \varphi(y')\} )} \leq 1, \quad \text{and}    \\
   & \quad \max \big\{||f||_{L^{\infty}(B_{1}\cap \{y_{n}>\varphi(y')\})}, ||\mathrm{osc}_{F}||_{L^{\infty}(B_{1}\cap \{y_{n}>\varphi(y')\})}, \mathcal{M}_{1}, \mathcal{M}_{2}  \big\} \leq \epsilon_{0},
\end{align*}
for some constant $ 0 < \epsilon_{0} \ll 1 $, and also $ \nu_{0} = \nu_{1} =1 $ in \hyperref[A3]{\bf (A3)}. In what follows, for a fixed $ B_{r}(x) \Subset B_{1} $, we define $ v: B_{1} \cap  \{y_{n}>\widetilde{\varphi}(y')\} \rightarrow \mathbb{R} $ by
\begin{equation*}
  v(y):=\frac{u(x+ry)}{K}
\end{equation*}
for a function $ \widetilde{\varphi} $ and a positive constant $ K \geq 1 \geq r $ to be determined later. Direct calculation yield that $ v $ is a viscosity solution to
\begin{equation*}
\left\{
     \begin{alignedat}{2}
         \widetilde{\Phi}(|Dv+\widetilde{\xi}|,y) \widetilde{F}(D^{2}v,y)+ \widetilde{H}(|Dv+\widetilde{\xi}|, y)  & = \widetilde{f}         \quad   &&   \text{in} \ \ B_{1} \cap \{y_{n} > \widetilde{\varphi}(y')\}        ,    \\
             v&  = \widetilde{g}    \quad  &&   \text{on} \ \  B_{1} \cap \{y_{n} = \widetilde{\varphi}(y')\} ,        \\
     \end{alignedat}
     \right.
\end{equation*}
where
\begin{equation*}
\left\{
     \begin{aligned}
& \widetilde{F}(\mathrm{X}, y):= \frac{r^{2}}{K} F \bigg(\frac{K}{r^{2}}\mathrm{X}, x+ry \bigg), \quad  \widetilde{\Phi}(t,y):= \frac{\Phi(\frac{K}{r}t, x+ry)}{\Phi(\frac{K}{r}, x+ry)};   \\
& \widetilde{H}(t,y):=  \frac{r^{2}H(\frac{K}{r}t, x+ry)}{K\Phi(\frac{K}{r}, x+ry)};    \\
& \widetilde{f}(y):=\frac{r^{2}f(x+ry)}{K\Phi(\frac{K}{r}, x+ry)}     ,   \quad  \widetilde{g}(y):= \frac{g(x+ry)}{K};    \\
& \widetilde{\xi}:= \frac{r}{K}\xi , \quad \text{and}  \quad  \widetilde{\varphi}(y'):= \frac{\varphi(ry'+x')-x_{n}}{r}.
\end{aligned}
     \right.
\end{equation*}

Notice that $ \widetilde{F} $ is a uniformly $ (\lambda, \Lambda)$ elliptic operator, the map $ t \mapsto \frac{\widetilde{\Phi}(t,y)}{t^{i(\Phi)}} $ is also almost non-decreasing and still satisfies $ \frac{\widetilde{\Phi}(t,y)}{t^{i(\Phi)}} \leq L \frac{\widetilde{\Phi}(s,y)}{s^{i(\Phi)}} $ whenever $ 0 < t \leq s < \infty $, for some constant $ L > 0 $. Similarly, the map $ t \mapsto \frac{\widetilde{\Phi}(t,y)}{t^{s(\Phi)}} $ is also almost non-increasing and still fulfills $ L \frac{\widetilde{\Phi}(t,y)}{t^{s(\Phi)}} \geq \frac{\widetilde{\Phi}(s,y)}{s^{s(\Phi)}} $ whenever $ 0 < t \leq s < \infty $, for some constant $ L > 0 $. Here $ s(\Phi) \geq i(\Phi) \geq -1 $ as in \hyperref[A3]{\bf (A3)} and $ \widetilde{\Phi}(1,y) = 1 $ for all $ y \in B_{1} $.

Furthermore, a routine calculation reveals that
\begin{align*}
  \mathrm{osc}_{\widetilde{F}}(y,0):= \sup_{\mathrm{M} \in \mathrm{Sym}(n)\setminus\{0\}} \frac{|\widetilde{F}(\mathrm{M}, y)-\widetilde{F}(\mathrm{M}, 0)|}{||\mathrm{M}||} & = \sup_{\mathrm{M} \in \mathrm{Sym}(n)\setminus\{0\}} \frac{|F(\frac{K}{r^{2}}\mathrm{M}, x+ry)-F(\frac{K}{r^{2}}\mathrm{M}, x)|}{\frac{K}{r^{2}}||\mathrm{M}||}   \\
  & = \mathrm{osc}_{F}(x+ry, x),
\end{align*}
which yields that
\begin{equation*}
  ||\mathrm{osc}_{\widetilde{F}}||_{L^{\infty}(B_{1}\cap \{y_{n}>\widetilde{\varphi}(y')\})} \leq  ||\mathrm{osc}_{F}(\cdot, x)||_{L^{\infty}(B_{1}\cap \{y_{n}>\varphi(y')\})} \leq \mathrm{C}_{F} r^{\theta_{F}}.
\end{equation*}

It is immediate from the choice of $ r $ that $ ||D^{2}\widetilde{\varphi}||_{\infty} \leq ||D^{2}\varphi||_{\infty} $ and
\begin{equation*}
  ||\widetilde{g}||_{C^{1,\beta_{g}}(B_{1} \cap \{y_{n} = \widetilde{\varphi}(y')\})} \leq \frac{1}{K} ||g||_{C^{1,\beta_{g}}(\partial \Omega)}.
\end{equation*}

In addition, we combine \hyperref[A3]{\bf (A3)}, \hyperref[A4]{\bf (A4)} and $ \frac{K}{r} \gg  1 $, we have
\begin{align*}
  |\widetilde{H}(t,y)| = \frac{r^{2}}{K} \frac{H(\frac{K}{r}t, x+ry)}{|\Phi(\frac{K}{r}, x+ry)|} & \leq \frac{r^{2}}{K} \frac{1}{|\Phi(\frac{K}{r}, x+ry)|} \bigg[\mathcal{M}_{1}+\mathcal{M}_{2}\left(\frac{K}{r}\right)^{m}|t|^{m}\bigg]   \\
  & \leq \frac{Lr^{2+i(\Phi)}}{\nu_{0}K^{1+i(\Phi)}}\bigg[\mathcal{M}_{1}+\mathcal{M}_{2}\left(\frac{K}{r}\right)^{m}|t|^{m}\bigg]:=\widetilde{\mathcal{K}}_{1} + \widetilde{\mathcal{K}}_{2}|t|^{m}   \\
\end{align*}
and
\begin{equation*}
  ||\widetilde{f}||_{L^{\infty}(B_{1}\cap \{y_{n}>\widetilde{\varphi}(y')\})} \leq \frac{Lr^{2+i(\Phi)}}{\nu_{0}K^{1+i(\Phi)}} ||f||_{L^{\infty}(\Omega)}.
\end{equation*}

Now, for given $ \epsilon_{0} \in (0,1)$, which will be sufficiently small but fixed. Then, we select
\[
K :=
\begin{cases}
1 + \|u\|_{L^\infty(\Omega)} + \left[\frac{L}{\nu_0} \left(\|f\|_{L^\infty(\Omega)} + \mathcal{M}_{1}\right)\right]^{\frac{1}{1+i(\Phi)}} + \left(\frac{L\mathcal{M}_{2}}{\nu_0}\right)^{\frac{1}{1+i(\Phi)-m}} + ||g||_{C^{1,\beta_{g}}(\partial \Omega)}    \\[1em]   &    \text{for } \ m < 1 + i(\Phi),     \\[1em]
1 + \|u\|_{L^\infty(\Omega)} + ||g||_{C^{1,\beta_{g}}(\partial \Omega)}     &  \text{for } \ m = 1 + i(\Phi),  \\[1em]
\end{cases}
\]
and
\[
r :=
\begin{cases}
\min\left\{1,  \left(\frac{\epsilon_{0}}{C_F}\right)^{\frac{1}{\theta_{F}}}, \epsilon_{0}^{\frac{1}{2+i(\Phi)}}, \epsilon_{0}^{\frac{1}{2+i(\Phi)-m}}\right\} & \text{for } \ m < 1 + i(\Phi), \\[1.5em]
\min\left\{1, \left(\frac{\epsilon_{0}}{C_F}\right)^{\frac{1}{\theta_{F}}}, \left(\frac{\epsilon_{0}\nu_0}{L\left(\|f\|_{L^\infty(\Omega)} + \mathcal{M}_{1}\right)}\right)^{\frac{1}{2+i(\Phi)}}, \frac{\epsilon_{0}\nu_0}{L\mathcal{M}_{2}}\right\} & \text{for } \ m = 1 + i(\Phi).
\end{cases}
\]
With such choice, we arrive at
\[
\|v\|_{L^{\infty}(B_{1}\cap \{y_{n}>\widetilde{\varphi}(y')\})}  \leq \frac{\|u\|_{L^\infty(\Omega)}}{K} \leq 1, \quad  ||\widetilde{g}||_{C^{1,\beta_{g}}(B_{1} \cap \{y_{n} = \widetilde{\varphi}(y')\})} \leq 1,
\]
and
\[ \max\left\{||\mathrm{osc}_{\widetilde{F}}||_{L^{\infty}(B_{1}\cap \{y_{n}>\widetilde{\varphi}(y')\})}, \|\tilde{f}\|_{L^{\infty}(B_{1}\cap \{y_{n}>\widetilde{\varphi}(y')\})}, \widetilde{\mathcal{K}}_{1}, \widetilde{\mathcal{K}}_{2}\right\} \leq \epsilon_{0}.
\]
Therefore, $v$ solves an equation possessing the same structure as \eqref{Section3:eq1} and $v$ is in the smallness regime.
\end{remark}

The following lemma describes the boundary behavior of a viscosity solution $ u $ in terms of a distance function $ d $.

\begin{lemma}\label{Section3:lemma2}
Suppose that the assumptions \hyperref[A1]{\bf (A1)}--\hyperref[A6]{\bf (A6)} hold. Let $ g \in C^{1,\beta_{g}}(\partial \Omega) $ and $ d $ denote the distance to the hypersurface $ \{y_{n} = \varphi(y')\} $. Then for every $r \in (0,1)$, there exists $ C_{\varphi, B_{r}} $, depending on $ ||f||_{L^{\infty}(B_{1}\cap \{y_{n}>\varphi(y')\})} $, $ \lambda, \Lambda $, $ \beta_{g} $, $ \varphi $, $ \mathrm{Lip}_{g}(\partial \Omega) $, such that if $ u $ is a viscosity solution of
\begin{equation*}
\left\{
     \begin{alignedat}{2}
         \Phi(|Du|,y) F(D^{2}u,y)+ H(|Du|, y)  & = f         \quad   &&   \text{in} \quad  B_{1} \cap \{y_{n} > \varphi(y')\}     ,    \\
          u   &  = g    \quad  &&   \text{on} \quad  B_{1} \cap \{y_{n} = \varphi(y')\} ,        \\
     \end{alignedat}
     \right.
\end{equation*}
and $\|u\|_{L^\infty(B_1 \cap \{y_n > \varphi(y')\})} \leq 1$, then
\[
|u(y', y_n) - g(y')| \leq \mathrm{C}_{\varphi, B_{r}} d(y)   \quad   \text{in } B_r \cap \{y_n > \varphi(y')\}.
\]
\end{lemma}

Before proving Lemma~\ref{Section3:lemma2}, we need the following a comparison principle:

\begin{proposition}[{Comparison principle I}]
\label{Section3:prop1}
Let $ \Omega $ be a bounded $ C^{2} $ domain in $ \mathbb{R}^{n} $. Suppose that $ f $ and $ g $ are continuous in $ \overline{\Omega}$, and \( H : \mathbb{R}^n \times \Omega \to \mathbb{R} \) is continuous. Assume that $ \omega \in C^{0}(\overline{\Omega}) \cap C^{2}(\Omega)$ is a viscosity super-solution of
\begin{equation*}
  \Phi(|D\omega|,x) F(D^{2}\omega,x)+ H(|D\omega|, x) = f  \quad \text{in} \quad \Omega,
\end{equation*}
and $ u \in C^{0}(\overline{\Omega})$ is a viscosity sub-solution of
\begin{equation*}
  \Phi(|Du|,x) F(D^{2}u,x)+ H(|Du|, x) = g  \quad \text{in} \quad \Omega.
\end{equation*}
Suppose that $ f < g $ in $ \overline{\Omega}$. If $ u \leq \omega $ on $ \partial \Omega$, then $ u \leq  \omega $ in $ \Omega $.
\end{proposition}

For the sake of brevity, we omit the proof of this result, which is well-known to specialists; we refer the interested reader to \cite[Proposition 4.1]{BD16} or \cite[Lemma 4.2]{BBLL24b} for further details.

We are now ready to prove Lemma~\ref{Section3:lemma2}.

\begin{proof}[{Proof of Lemma~\ref{Section3:lemma2}}]
Let $ r < 1 $, and $ \Omega_{r} $ be a bounded $ C^{2} $-domain which contain $ \{ y_{n} > \varphi(y')\} \cap B_{\frac{1+r}{2}} $ and is contained in $ B_{1} \cap \{y_{n} > \varphi(y')\} $. Our main goal is to show that there exists a positive constant $ \mathrm{C}_{\varphi, B_{r}} $ such that
\begin{equation}\label{Section3:eq2}
  |u(y', y_{n})- g(y')| \leq  \mathrm{C}_{\varphi, B_{r}} d(y, \partial \Omega),   \quad   \forall y \in \Omega_{r}.
\end{equation}

Recall that the distance function $ d $ is of class $ C^{2} $ in a neighborhood of the boundary and semi-concave elsewhere. Consequently, there exists a constant $ \mathrm{C}_{1} $ such that $ D^{2}d \leq \mathrm{C}_{1} $. For simplicity, we denote by $ D_{\Omega_{r}} $ the diameter of $ \Omega_{r} $. Let $ \widetilde{g} $ be a viscosity solution to
\begin{equation*}
\left\{
     \begin{alignedat}{2}
           \mathscr{P}^{+}_{\lambda,\Lambda}(D^{2}\widetilde{g})& = f         \quad     &&   \text{in} \ \ B_{1} \cap \{y_{n}> \varphi(y')\}    ,    \\
           \widetilde{g}  &  = g    \quad  &&   \text{on} \ \  B_{1} \cap \{y_{n} = \varphi(y')\}.        \\
     \end{alignedat}
     \right.
\end{equation*}
It is well known that $ \widetilde{g} \in C^{1,\beta_{\widetilde{g}}}(\{y_{n} \geq \varphi(y')\}) \cap C^{2}(\{y_{n} > \varphi(y')\})   $, for some $ 0 < \beta_{\widetilde{g}}<1 $, see \cite{CC95}. Moreover, one has $ ||\widetilde{g}||_{\infty} \leq ||g||_{\infty} \leq 1   $ and $ ||\nabla \widetilde{g}||_{\infty} \leq \mathrm{C}||g||_{\infty}  $, for some constant $ \mathrm{C} $ depending on $ \lambda, \Lambda, n, \Omega $.

Let $ k $ be so that $ n \Lambda \mathrm{C}_{1}(1+ D_{\Omega_{r}}) < \frac{(k+1)\lambda}{2}$, and
\begin{equation}\label{Section3:bc1}
   \mathrm{c}_{1} > \mathcal{M}_{2}\mathrm{c}_{2}^{m},
\end{equation}
where
\begin{equation}\label{Section3:eq3}
  \mathrm{c}_{1}:= \frac{\nu_{0}\lambda k^{1+i(\Phi)}(k+1)}{4L(1+D_{\Omega_{r}})^{k+2}} \min \bigg\{ \left( 1+ \frac{1}{2(1+D_{\Omega_{r}})^{k+1}} \right)^{i(\Phi)}, \frac{1}{2^{i(\Phi)}(1+D_{\Omega_{r}})^{(k+1)i(\Phi)}}   \bigg\}
\end{equation}
and
\begin{equation}
 \mathrm{c}_{2}:= k \bigg[1+\frac{1}{2(1+D_{\Omega_{r}})^{1+k}}   \bigg].
\end{equation}
Moreover, we select $ M $ sufficiently large so that
\begin{equation}\label{Section3:eq4}
  Mk > 2\bigg(1+||\nabla \widetilde{g}||_{\infty}+\frac{2^{7}}{1-r}\bigg)(1+D_{\Omega_{r}})^{1+k}
\end{equation}
and
\begin{equation}\label{Section3:bc3}
\left\{
     \begin{alignedat}{2}
                   & M^{1+i(\Phi)}\mathrm{c}_{1} - \mathcal{M}_{1} - \mathcal{M}_{2}\mathrm{c}_{2}^{m} M^{m}  \geq ||f||_{\infty}      \quad     &&   \text{for} \ \  0< m< 1+i(\Phi)  ,    \\
           &   M^{1+i(\Phi)}(\mathrm{c}_{1}-\mathcal{M}_{2}\mathrm{c}_{2}^{m})-\mathcal{M}_{1} \geq  ||f||_{\infty}      \quad  &&   \text{for}   \ \  m=1+i(\Phi),         \\
     \end{alignedat}
     \right.
\end{equation}

The proof of \eqref{Section3:eq2} relies on the construction of upper and lower barriers. For this purpose, let $ \omega \in C^{2}(\Omega_{r}) $ be defined by
\begin{equation*}
\omega(y):=\left\{
     \begin{alignedat}{2}
                   & M \omega_{0}(y) + \widetilde{g}(y)       \quad     &&   \text{for} \ \  |y| \leq r   ,    \\
           &   M \omega_{0}(y) + \frac{2^{5}}{(1-r)^{3}}(|y|-r)^{3} + \widetilde{g}(y)      \quad  &&   \text{for}   \ \  |y|>r,         \\
     \end{alignedat}
     \right.
\end{equation*}
where  $ \omega_{0}(y):=  1- \frac{1}{(1+d(y))^{k}}    $. A routine calculation reveals that
\begin{equation*}
\nabla \omega(y)=\left\{
     \begin{alignedat}{2}
                   & M \nabla \omega_{0}(y) + \nabla \widetilde{g}(y)       \quad     &&   \text{for} \ \  |y| \leq r   ,    \\
           &   M \nabla \omega_{0}(y) +  \frac{3\times2^{5}}{(1-r)^{3}}(|y|-r)^{2}\frac{y}{|y|}+ \nabla \widetilde{g}(y)      \quad  &&   \text{for}   \ \  |y|>r,         \\
     \end{alignedat}
     \right.
\end{equation*}
where $ \nabla \omega_{0} = \frac{k \nabla d }{(1+d)^{k}}  $. Noticing that the properties of distance function $ d $, we see that
\begin{equation}\label{Section3:eq6}
  |\nabla\omega|  \geq \frac{Mk}{2(1+D_{\Omega_{r}})^{1+k}}  \quad   \text{and}  \quad  |\nabla\omega|  \leq  \bigg[1+\frac{1}{2(1+D_{\Omega_{r}})^{1+k}}   \bigg]Mk  \quad  \text{in}  \quad \Omega_{r}.
\end{equation}

Furthermore, it is easy to see that
\begin{equation*}
  \nabla^{2} \omega_{0}(y) = \frac{k\nabla^{2}d}{(1+d)^{1+k}} - k(k+1) \frac{\nabla d \otimes \nabla d}{(1+d)^{k+2}}.
\end{equation*}
By the properties of $ \mathscr{P}^{+}_{\lambda,\Lambda} $ and $ ||\nabla^{2} \widetilde{g}||_{\infty} \leq \mathrm{C} $, it reads
\begin{align}\label{Section3:eq7}
\begin{split}
  \mathscr{P}^{+}_{\lambda,\Lambda}(D^{2}\omega) \leq M  \mathscr{P}^{+}_{\lambda,\Lambda}(D^{2}\omega_{0}) + \mathrm{C} & \leq -\frac{\lambda M k(k+1)}{2(1+D_{\Omega_{r}})^{k+2}}  + \mathrm{C}  \\
  & \leq -\frac{\lambda M k(k+1)}{4(1+D_{\Omega_{r}})^{k+2}}
\end{split}
\end{align}
provided $ M $ is large. Noticing that
\begin{equation}\label{Section3:bc2}
\begin{split}
|D\omega|^{i(\Phi)} \geq \left\{
     \begin{alignedat}{2}
                   &  \left(\frac{kM}{2(1+D_{\Omega_{r}})^{k+1}}    \right)^{i(\Phi)}      \quad     &&   \text{for} \ \  i(\Phi) \geq 0   ,    \\
           &  \bigg[kM \left(1+\frac{1}{2(1+D_{\Omega_{r}})^{k+1}}     \right)\bigg]^{i(\Phi)}          \quad  &&   \text{for}   \ \  -1 < i(\Phi) < 0,         \\
     \end{alignedat}
     \right.
     \end{split}
\end{equation}

Now we combine \hyperref[A1]{\bf (A1)}, \hyperref[A3]{\bf (A3)}, \hyperref[A4]{\bf (A4)} and \eqref{Section3:bc1}--\eqref{Section3:bc2}, we obtain
\begin{align*}
\begin{split}
& \Phi(|D\omega|,x)  F(D^{2}\omega,x)+ H(|D\omega|, x)  \leq \Phi(|D\omega|,x) \mathscr{P}^{+}_{\lambda,\Lambda}(D^{2}\omega) + H(|D\omega|, x)         \\
& \leq - \frac{\nu_{0}}{L} |D \omega|^{i(\Phi)} \frac{k(k+1)\lambda M}{4(1+D_{\Omega_{r}})^{k+2}}  + \mathcal{M}_{1}  +  \mathcal{M}_{2} |D \omega|^{m}      \\
& \leq -M^{1+i(\Phi)}\mathrm{c}_{1}  + \mathcal{M}_{1} + \mathcal{M}_{2}\mathrm{c}_{2}^{m} M^{m}  \\
&  \overset{\eqref{Section3:bc3}}{\leq} -||f||_{\infty}.
\end{split}
\end{align*}

Next we will verify $ \omega \geq u $ on $ \partial \Omega_{r} $. Observed that on $ \partial \Omega_{r} \cap   \{y_{n} = \varphi(y')\} $, we readily see $ \omega \geq \widetilde{g} = u $. While on $ \Omega_{r} \cap B_{1} $, we have $ \omega \geq 4+\widetilde{g}(y) \geq 1 \geq u $ since $ \frac{1+r}{2} \leq |y| < 1 $. Finally, by the comparison principle (Proposition \ref{Section3:prop1}), we obtain $ \omega \geq u $ in $ \Omega_{r} $. That is to say,
\begin{equation*}
  u(y) - g(y') \leq  M \omega_{0}(y) + \widetilde{g}(y') - g(y') \leq \mathrm{C}_{0} d(y, \partial \Omega_{r}).
\end{equation*}
This proves the upper bound for $ u- g $. Similarly, we can show a lower bound for $ u- g $. Hence we finish the proof of Lemma \ref{Section3:lemma2}.
\end{proof}

With Lemma~\ref{Section3:lemma2} in hand, combined with Crandall-Ishii-Lions Lemma (\cite[Theorem 3.2]{CIL92}), we initiate the analysis for case $ 0 \leq i(\Phi) < m \leq 1+i(\Phi) $. In this setting, we impose an extra condition \eqref{Section3:eq8} to control the growth of the Hamiltonian term $ \mathcal{M}_{2}|Du+\xi|^{m-i(\Phi)} $.

\begin{proposition}[{Boundary H\"{o}lder regularity: degenerate case}]\label{Section3:Prop3}
Let $ g$ be a Lipschitz continuous function. Suppose that \hyperref[A1]{\bf (A1)}--\hyperref[A6]{\bf (A6)} are in force with $ 0 \leq i(\Phi) < m \leq 1+i(\Phi) $ and $ \nu_{0} = \nu_{1} =1 $. Let $ \xi \in \mathbb{R}^{n} $ and $ u $ be a viscosity solution to \eqref{Section3:eq1} with $ ||u||_{L^{\infty}(B_{1}\cap \{y_{n}>\varphi(y')\})} \leq 1 $. There exists a universal constant $ \mathrm{\kappa}_{0} > 0 $ such that if
\begin{equation}\label{Section3:eq8}
  \mathcal{M}_{2} \big(1+|\xi|^{m-i(\Phi)}\big) \leq  \kappa
\end{equation}
for some $ 0< \kappa \leq \kappa_{0} $. Then $ u \in C^{0,\gamma}(B_{r}\cap \{y_{n}>\varphi(y')\})$ and
\begin{equation*}
  ||u||_{C^{0,\gamma}(B_{r}\cap \{y_{n}>\varphi(y')\})} \leq \mathrm{C},
\end{equation*}
where $ \mathrm{C} > 0 $ is a constant depending on $ n, \lambda, \Lambda, L, m, i(\Phi), \mathrm{C}_{F}, \theta_{F}, \kappa_{0}, \mathcal{M}_{1}, r, ||f||_{L^{\infty}(\Omega)} $ and $ \mathrm{Lip}_{g}(\partial \Omega) $.
\end{proposition}

\begin{proof}
Let $ r < \frac{1}{2} $, and let $ r_{1} \in (r,\frac{1}{2}) $ be fixed. For $ x_{0} \in B_{r} \cap \{y_{n} > \varphi(y')\} $. We consider an auxiliary function
\begin{equation}\label{Section3:eq9}
  \Gamma(x,y):=u(x) - u(y) - L_{1}|x-y|^{\gamma} - L_{2}(|x-x_{0}|^{2} + |y-x_{0}|^{2}),
\end{equation}
where $ \gamma \in (0,1) $ and $ L_{1}, L_{2} >1 $. We claim that for $ L_{1}, L_{2} \gg 1 $ large enough,
\begin{equation}\label{Section3:eq10}
  \Gamma(x,y) \leq 0  \quad  \text{in}  \quad (B_{r_{1}} \cap \{y_{n} > \varphi(y') \})^{2}.
\end{equation}
Once this is established, it implies $ u \in C^{0,\gamma}(B_{r}\cap \{y_{n}>\varphi(y')\}) $.

Firstly, assume that $ y \in B_{r_{1}} \cap \{y_{n}>\varphi(y')\} $, then by using Lemma \ref{Section3:lemma2}, there exists a constant $ \widetilde{\mathrm{C}_{0}} $ such that
\[
|u(z) - g(z')| \leq  \widetilde{\mathrm{C}_{0}}   \mathrm{dist}(z, \partial\Omega) \quad \text{for } z \in B_{r_1} \cap \{y_n > \varphi(y')\},
\]
which implies that
\[
\begin{aligned}
|u(x) - u(y)| &\leq |u(x', x_n) - u(x', \varphi(x'))| + |u(x', \varphi(x')) - u(y', \varphi(y'))| \\
&\leq \widetilde{\mathrm{C}_{0}} \mathrm{dist}(x, \partial\Omega) + \mathrm{Lip}_g(\partial\Omega) |x' - y'| \leq (\widetilde{\mathrm{C}_{0}} + \mathrm{Lip}_g(\partial\Omega)) |x - y|.
\end{aligned}
\]
Therefore, recall $ \gamma \in (0,1) $, and if we choose \( \frac{L_{1}}{4} \geq \widetilde{\mathrm{C}_{0}} + \mathrm{Lip}_g(\partial\Omega) \), then
\[
\Gamma(x,y) \leq L_{1}\left( \frac{|x - y|}{4} - |x-y|^{\gamma} \right) - L_{2}\left( |x - x_0|^2 + |y - x_0|^2 \right) \leq 0.
\]

We now prove \eqref{Section3:eq10} by contradiction, suppose that there exists some point \((\hat{x}, \hat{y}) \in (B_{r_1} \cap \overline{\Omega})^{2}\) such that
\[
\Gamma(\hat{x}, \hat{y}) = \max_{(B_{r_1} \cap \overline{\Omega}) \times (B_{r_1} \cap \overline{\Omega})} \Gamma(x, y) > 0.
\]
Here, we also choose \( L_{2} > 32/(r_1 - r)^2\) and remember that $ ||u||_{L^{\infty}(B_{1}\cap \{y_{n}>\varphi(y')\})} \leq 1 $,   then we can easily check that
\begin{enumerate}[label=(\roman*)]

\item \label{1i} $ \hat{x} \neq \hat{y}$, \quad  \text{and}  \quad  $ |\hat{x}-\hat{y}|  \leq \frac{2}{\sqrt{L_{2}}} $;

\item  $ |\hat{x}-x_{0}|, |\hat{y}-x_{0}|  < \frac{r_{1}-r}{4} $   \quad   \text{and}  \quad  $ |\hat{x}-\hat{y}| \leq \frac{r_{1}-r}{2} <1$;

\item $ \hat{x}, \hat{y} \in B_{r_{1}} \cap \{y_{n}> \varphi(y')\} $.
\end{enumerate}

We are in a position to apply the Ishii-Lions lemma (\cite[Theorem 3.2]{CIL92}) to ensure the existence of a limiting subjet $(\xi_{\hat{x}}, \mathrm{X})$ of $u$ at $\hat{x}$ and a limiting superjet $(\xi_{\hat{y}}, \mathrm{Y})$ of $u$ at $\hat{y}$, such that the matrices $\mathrm{X}, \mathrm{Y} \in \text{Sym}(n) $ satisfy the matrix inequality
\[
\begin{pmatrix}\label{Section3:eq11}
\mathrm{X} & 0 \\
0 & -\mathrm{Y}
\end{pmatrix}
\leq
\begin{pmatrix}
\mathrm{A} & -\mathrm{A} \\
-\mathrm{A} & \mathrm{A}
\end{pmatrix}
+ (2L_2 + \epsilon)
\begin{pmatrix}
\textbf{I}\mathrm{d_{n}} & 0     \\
0 & \textbf{I}\mathrm{d_{n}}
\end{pmatrix}
\]
with $\epsilon \in (0,1)$, that only depends on $ ||\mathrm{A}|| $ and can be made sufficiently small. Here,
\begin{equation*}
\left\{
     \begin{aligned}
& \xi_{\hat{x}} := \gamma L_1 (\hat{x} - \hat{y}) |\hat{x} - \hat{y}|^{\gamma-2} + 2L_2 (\hat{x} - x_0);      \\
& \xi_{\hat{y}} := \gamma L_1 (\hat{x} - \hat{y}) |\hat{x} - \hat{y}|^{\gamma-2} - 2L_2 (\hat{y} - x_0) ;    \\
& \mathrm{A} := L_1 \gamma \left[ (\gamma - 2) |\hat{x} - \hat{y}|^{\gamma-4} \left( (\hat{x} - \hat{y}) \otimes (\hat{x} - \hat{y}) \right) + |\hat{x} - \hat{y}|^{\gamma-2} \textbf{I}\mathrm{d_{n}} \right].
\end{aligned}
     \right.
\end{equation*}

Furthermore, we have the following viscosity inequalities
\begin{equation}\label{Section3:eq12}
\left\{
     \begin{aligned}
& \Phi(|\xi_{\hat{x}} + \xi|, \hat{x}) F(\mathrm{X}, \hat{x}) + H(|\xi_{\hat{x}} + \xi|, \hat{x}) \geq f(\hat{x});      \\
& \Phi(|\xi_{\hat{y}} + \xi|, \hat{y}) F(\mathrm{Y}, \hat{y}) + H(|\xi_{\hat{y}} + \xi|, \hat{y}) \leq f(\hat{y}).
\end{aligned}
     \right.
\end{equation}

Using the properties of $ \hat{x} $ and $ \hat{y} $, it can be readily seen that
\begin{equation}\label{Section3:eq13}
  \frac{\gamma L_{1}}{2} |\widehat{x}-\widehat{y}|^{\gamma -1} \leq |\xi_{\hat{x}}|, |\xi_{\hat{y}}| \leq  2 \gamma L_{1} |\widehat{x}-\widehat{y}|^{\gamma -1}.
\end{equation}

We proceed to estimate the bounds of $ |\xi_{\hat{x}} + \xi| $ and $ |\xi_{\hat{y}} + \xi|$. Let $ \widetilde{\mathcal{M}_{0}}:=  \big( \frac{\kappa_{0}}{\mathcal{M}_{2}}   \big)^{\frac{1}{m-i(\Phi)}}$, then from \eqref{Section3:eq8} and $ m > i(\Phi) $, it reads
\begin{equation}\label{Section3:eq14}
  \widetilde{\mathcal{M}_{0}} \geq 1  \quad    \text{and}   \quad |\xi|\leq \widetilde{\mathcal{M}_{0}}.
\end{equation}
Choosing $ L_{1} > \frac{4\widetilde{\mathcal{M}_{0}}}{\gamma} $, then we combine \eqref{Section3:eq13}, \eqref{Section3:eq14} and $ |\hat{x}-\hat{y}|^{\gamma-1} > 1 $, we have
\begin{equation}\label{Section3:eq15}
  1 \leq \widetilde{\mathcal{M}_{0}}|\hat{x}-\hat{y}|^{\gamma-1} \leq |\xi_{\hat{x}}+\xi|, |\xi_{\hat{y}}+\xi| \leq 3 \gamma L_{1}|\hat{x}-\hat{y}|^{\gamma-1}.
\end{equation}

Next we can rewrite \eqref{Section3:eq12} as
\begin{equation}\label{Section3:eq16}
\left\{
     \begin{aligned}
&  F(\mathrm{X}, \hat{x}) + \frac{H(|\xi_{\hat{x}} + \xi|, \hat{x})}{\Phi(|\xi_{\hat{x}} + \xi|, \hat{x})} \geq \frac{f(\hat{x})}{\Phi(|\xi_{\hat{x}} + \xi|, \hat{x})};      \\
&  F(\mathrm{Y}, \hat{y}) + \frac{H(|\xi_{\hat{y}} + \xi|, \hat{y})}{\Phi(|\xi_{\hat{y}} + \xi|, \hat{y})} \leq \frac{f(\hat{y})}{\Phi(|\xi_{\hat{y}} + \xi|, \hat{y})}.
\end{aligned}
     \right.
\end{equation}
Subtracting the two inequalities in~\eqref{Section3:eq16}, and using \hyperref[A1]{\bf (A1)} and \hyperref[A2]{\bf (A2)}, we deduce
\begin{align}\label{Section3:eq17}
\begin{split}
  \frac{f(\hat{x})}{\Phi(|\xi_{\hat{x}} + \xi|, \hat{x})} & - \frac{H(|\xi_{\hat{x}} + \xi|, \hat{x})}{\Phi(|\xi_{\hat{x}} + \xi|, \hat{x})} + \frac{H(|\xi_{\hat{y}} + \xi|, \hat{y})}{\Phi(|\xi_{\hat{y}} + \xi|, \hat{y})} - \frac{f(\hat{y})}{\Phi(|\xi_{\hat{y}} + \xi|, \hat{y})}  \\
  & \leq F(\mathrm{X}, \hat{x}) - F(\mathrm{Y}, \hat{y}) \\
  & \leq \mathrm{C}_{F} |\hat{x}-\hat{y}|^{\theta_{F}} ||\mathrm{X}|| + \mathscr{P}^{+}_{\lambda,\Lambda}(\mathrm{X}-\mathrm{Y}).
  \end{split}
\end{align}

Note that the upper estimates for the norms $ ||\mathrm{X}|| $ and $ \mathscr{P}^{+}_{\lambda,\Lambda}(\mathrm{X}-\mathrm{Y}) $, and estimates for each term on the left side of \eqref{Section3:eq17} which can be found in \cite[Proposition 3.1]{HJMZ26}, hence it leads to
\begin{align*}
    &- 2L(\|f\|_{\infty} + \mathcal{M}_{1}) \\
    \leq & \mathrm{C}_1 + L_1 |\hat{x} - \hat{y}|^{\gamma-2} \left( \mathrm{C}_F\gamma |\hat{x} - \hat{y}|^{\theta_{F}} - 4\lambda\gamma(1 - \gamma) + 6L\mathcal{M}_{2} |\hat{x} - \hat{y}| \right),
\end{align*}
where $\mathrm{C}_1 := (\Lambda(n - 1) + \lambda) (4L_2 + 2) + \mathrm{C}_F(2L_2 + 1)$. Choose
\[
L_2 \ge \max \left\{ 4 \left( \frac{\mathrm{C}_F}{\lambda(1 - \gamma)} \right)^{2/{\theta_{F}}}, \left( \frac{12L\mathcal{M}_{2}}{\lambda\gamma(1 - \gamma)} \right)^2 \right\}.
\]
By this choice, we apply \ref{1i} to derive
\begin{equation}\label{Section3:eq18}
\mathrm{C}_F |\hat{x} - \hat{y}|^{\theta_{F}} \le \mathrm{C}_F \left( \frac{2}{\sqrt{L_2}} \right)^{\theta_{F}} \le \lambda(1 - \gamma), \quad 6L\mathcal{M}_{2} |\hat{x} - \hat{y}| \le \frac{12L\mathcal{M}_{2}}{\sqrt{L_2}} \le \lambda\gamma(1 - \gamma).
\end{equation}
Utilizing \eqref{Section3:eq18}, $|\hat{x} - \hat{y}| < 1$ and $\gamma < 1$, we further conclude that
\begin{equation*}
  2L_1 \lambda\gamma(1-\gamma) < \mathrm{C}_{1} + 2L(||f||_{\infty}+\mathcal{M}_{1}),
\end{equation*}
which leads to a contradiction provided that $ L_{1} $ is large enough.

This finishes the proof of \eqref{Section3:eq10}.
\end{proof}

Following the argument for Proposition~\ref{Section3:Prop3}, we further establish the boundary $ C^{0,1} $ regularity for viscosity solution to \eqref{Main:eq1} under the condition $ -1 < i(\Phi) < 0 $.

\begin{proposition}[{Boundary Lipschitz regularity: singular case}]\label{Section3:Prop4}
Let $ g$ be a Lipschitz continuous function. Suppose that \hyperref[A1]{\bf (A1)}--\hyperref[A6]{\bf (A6)} are in force with $ -1 < i(\Phi) < 0 $. Let $ u $ be a viscosity solution to \eqref{Main:eq1} with $ ||u||_{L^{\infty}(B_{1}\cap \{y_{n}>\varphi(y')\})} \leq 1 $. Then $ u \in C^{0,1}(B_{r}\cap \{y_{n}>\varphi(y')\})$ and
\begin{equation*}
  ||u||_{C^{0,1}(B_{r}\cap \{y_{n}>\varphi(y')\})} \leq \mathrm{C},
\end{equation*}
where $ \mathrm{C} > 0 $ is a constant depending on $ n, \lambda, \Lambda, L, m, i(\Phi), \mathrm{C}_{F}, \theta_{F}, \kappa_{0}, \mathcal{M}_{1}, \mathcal{M}_{2}, r, ||f||_{L^{\infty}(\Omega)} $ and $ \mathrm{Lip}_{g}(\partial \Omega) $.
\end{proposition}

\begin{proof}
Since the proof is similar to Proposition \ref{Section3:Prop3}, here we concentrate on the differences.

Let $ r_{1} \in (r,\frac{1}{2}) $ be fixed. For $ x_{0} \in B_{r} \cap \{y_{n} > \varphi(y')\} $. We consider an auxiliary function
\begin{equation}\label{Section3:eq19}
  \widetilde{\Gamma}(x,y):=u(x) - u(y) - L_{1} \omega(|x-y|) - L_{2}(|x-x_{0}|^{2} + |y-x_{0}|^{2}),
\end{equation}
for $ L_{1}, L_{2} >1 $, where
\[
 \omega(s) =
\begin{cases}
s - \omega_0 s^{1+\beta} & \text{if } 0 \leq s \leq s_0 := \left( \dfrac{1}{(1+\beta)\omega_0} \right)^{1/\beta}, \\[6pt]
\omega(s_0) & \text{if } s > s_0,
\end{cases}
\]
with $\beta \in (0, \theta_{F})$. Here we choose $\omega_0 \in \left(0, \dfrac{1}{1+\beta}\right)$ such that $s_0 \geq 1$. Our goal is to show
\begin{equation}\label{Section3:eq20}
  \widetilde{\Gamma}(x,y) \leq 0  \quad  \text{in}  \quad (B_{r_{1}} \cap \{y_{n} > \varphi(y') \})^{2}.
\end{equation}

We now prove \eqref{Section3:eq20} by contradiction, suppose that there exists some point \((\hat{x}, \hat{y}) \in (B_{r_1} \cap \overline{\Omega})^{2}\) such that
\[
\widetilde{\Gamma}(\hat{x},\hat{y}) := \max_{(B_{r_1} \cap \overline{\Omega}) \times (B_{r_1} \cap \overline{\Omega})} \widetilde{\Gamma}(x,y)> 0.
\]
As in the proof of Proposition \ref{Section3:Prop3}, we obtain a limiting subjet $(\xi_{\hat{x}}, \mathrm{X})$ of $u$ at $\hat{x}$ and a limiting superjet $(\xi_{\hat{y}}, \mathrm{Y})$ of $u$ at $\hat{y}$, such that the matrices $\mathrm{X}, \mathrm{Y} \in \text{Sym}(n) $ satisfying the matrix inequality
\[
\begin{pmatrix}\label{Section3:eq11}
\mathrm{X} & 0 \\
0 & -\mathrm{Y}
\end{pmatrix}
\leq
\begin{pmatrix}
\mathrm{A} & -\mathrm{A} \\
-\mathrm{A} & \mathrm{A}
\end{pmatrix}
+ (2L_2 + \epsilon)
\begin{pmatrix}
\textbf{I}\mathrm{d_{n}}   & 0     \\
0 & \textbf{I}\mathrm{d_{n}}
\end{pmatrix}
\]
with $\epsilon \in (0,1)$, that only depends on $||\mathrm{A}|| $ and can be made sufficiently small. Here,
\begin{equation*}
\left\{
     \begin{aligned}
& \xi_{\hat{x}} :=  L_1 \omega'(|\hat{x} - \hat{y}|) \frac{\hat{x} - \hat{y}}{|\hat{x} - \hat{y}|} + 2L_2 (\hat{x} - x_0);      \\
& \xi_{\hat{y}} := L_1 \omega'(|\hat{x} - \hat{y}|) \frac{\hat{x} - \hat{y}}{|\hat{x} - \hat{y}|}- 2L_2 (\hat{y} - x_0) ;    \\
& \mathrm{A} := L_1 \left[ \frac{\omega'(|\hat x - \hat y|)}{|\hat x - \hat y|} \textbf{I}\mathrm{d_{n}}+ \left( \omega''(|\hat x - \hat y|) - \frac{\omega'(|\hat x - \hat y|)}{|\hat x - \hat y|} \right) \frac{(\hat x - \hat y) \otimes (\hat x - \hat y)}{|\hat x - \hat y|^2} \right].
\end{aligned}
     \right.
\end{equation*}
Furthermore, we have the following viscosity inequalities
\begin{equation}\label{Section3:eq21}
\left\{
     \begin{aligned}
& \Phi(|\xi_{\hat{x}}|, \hat{x}) F(\mathrm{X}, \hat{x}) + H(|\xi_{\hat{x}}|, \hat{x}) \geq f(\hat{x});      \\
& \Phi(|\xi_{\hat{y}}|, \hat{y}) F(\mathrm{Y}, \hat{y}) + H(|\xi_{\hat{y}} |, \hat{y}) \leq f(\hat{y}).
\end{aligned}
     \right.
\end{equation}
We claim that
\begin{equation}\label{Section3:eq22}
  1< |\xi_{\hat{x}} |, |\xi_{\hat{y}} | < 2L_{1}.
\end{equation}
In effect, we choose $ L_{1} > \max \{8, L_{2}\} $ and $ L_{2} > \bigg[  \frac{2^{3+\beta}\omega_{0}(1+\beta)}{5}  \bigg]^{\frac{2}{\beta}} $, then we have $ |\xi_{\hat{x}} | , |\xi_{\hat{y}} | < 2L_{1} $ and
\begin{equation*}
  |\xi_{\hat{x}}| \geq L_{1} \omega'(|\hat{x}-\hat{y}|) - \frac{L_{2}}{4}  \geq \frac{3}{8}L_{1} -  \frac{L_{2}}{4} >  \frac{L_{1}}{8} > 1.
\end{equation*}

Next we can rewrite \eqref{Section3:eq22} as
\begin{equation*}
\left\{
     \begin{aligned}
&  F(\mathrm{X}, \hat{x}) + \frac{H(|\xi_{\hat{x}}|, \hat{x})}{\Phi(|\xi_{\hat{x}}|, \hat{x})} \geq \frac{f(\hat{x})}{\Phi(|\xi_{\hat{x}}|, \hat{x})};      \\
&  F(\mathrm{Y}, \hat{y}) + \frac{H(|\xi_{\hat{y}}|, \hat{y})}{\Phi(|\xi_{\hat{y}}|, \hat{y})} \leq \frac{f(\hat{y})}{\Phi(|\xi_{\hat{y}}|, \hat{y})}.
\end{aligned}
     \right.
\end{equation*}
As in the proof of Proposition \ref{Section3:Prop3}, we deduce \begin{align}\label{Section3:eq23}
\begin{split}
  \frac{f(\hat{x})}{\Phi(|\xi_{\hat{x}}|, \hat{x})} & - \frac{H(|\xi_{\hat{x}}|, \hat{x})}{\Phi(|\xi_{\hat{x}}|, \hat{x})} + \frac{H(|\xi_{\hat{y}}|, \hat{y})}{\Phi(|\xi_{\hat{y}} |, \hat{y})} - \frac{f(\hat{y})}{\Phi(|\xi_{\hat{y}} |, \hat{y})}  \\
  & \leq F(\mathrm{X}, \hat{x}) - F(\mathrm{Y}, \hat{y}) \\
  & \leq \mathrm{C}_{F} |\hat{x}-\hat{y}|^{\theta_{F}} ||\mathrm{X}|| + \mathscr{P}^{+}_{\lambda,\Lambda}(\mathrm{X}-\mathrm{Y}).
  \end{split}
\end{align}
Observe that the upper estimates for the norms $ ||\mathrm{X}|| $ and $ \mathscr{P}^{+}_{\lambda,\Lambda}(\mathrm{X}-\mathrm{Y}) $, and estimates for each term on the left side of \eqref{Section3:eq17} which can be found in \cite[Proposition 3.2]{HJMZ26}, hence it yields
\begin{equation}\label{Section3:eq24}
-4L \left( \|f\|_{\infty} + \mathcal{M}_{1} \right) L_1^{-i(\Phi)} \leq \mathrm{C}_1 + L_1 |\hat{x} - \hat{y}|^{\beta-1} \left[ (\mathrm{C}_F + 4L\mathcal{M}_{2}) |\hat{x} - \hat{y}|^{\theta_{F}-\beta} - 4\lambda\omega_0\beta(1+\beta) \right],
\end{equation}
where \( \mathrm{C}_1 := (\Lambda(n-1) + \lambda)(4L_2 + 2) + \mathrm{C}_F(2L_2 + 1) \). Choosing \( L_2 \geq 4\left(\frac{\mathrm{C}_F + 4L\mathcal{M}_{2}}{2\lambda\omega_0\beta(1+\beta)}\right)^{2/(\theta_{F}-\beta)} \) such that
\[
(\mathrm{C}_F + 4L\mathcal{M}_{2}) |\hat{x} - \hat{y}|^{\theta_{F}-\beta} \leq (\mathrm{C}_F + 4L\mathcal{M}_{2})\left(\frac{2}{\sqrt{L_2}}\right)^{\theta_{F}-\beta} \leq 2\lambda\omega_0\beta(1+\beta),
\]
where we have used \ref{1i} of Proposition  \ref{Section3:Prop3} in the first inequality.

Recall that $ 0< \beta < \theta_{F} < 1 $ and $ |\hat{x} - \hat{y}| < 1 $, then \eqref{Section3:eq24} becomes
\begin{equation*}
  2\lambda\omega_0\beta(1+\beta)L_{1} \leq \mathrm{C}_{1}  +  4L \left( \|f\|_{\infty} + \mathcal{M}_{1} \right) L_1^{-i(\Phi)},
\end{equation*}
which does not hold for sufficiently large $ L_{1} $, when $ -1 < i(\Phi) < 0 $.

This completes the proof.
\end{proof}

To prove the boundary version of approximation lemma, see Lemma \ref{Section4:lem1}, we rely on the following near-boundary $ C^{0,1}$-estimate for the vector-switching equation in the degenerate setting ($ 0 < m \leq i(\Phi) $).

\begin{proposition}[{Boundary Lipschitz estimates for large $ |\xi|$: degenerate case}]\label{Section3:Thm5}
Let $ g$ be a Lipschitz continuous function. Suppose that \hyperref[A1]{\bf (A1)}--\hyperref[A6]{\bf (A6)} are in force with $ 0 < m \leq i(\Phi) $ and $ \nu_{0} = \nu_{1} =1 $. Let $ u $ be a viscosity solution to \eqref{Section3:eq1} with
\begin{equation*}
  ||u||_{L^{\infty}(B_{1}\cap \{y_{n}>\varphi(y')\})} \leq 1  \quad   \text{and}  \quad    ||f||_{L^{\infty}(B_{1}\cap \{y_{n}>\varphi(y')\})} \leq  \epsilon_{0}.
\end{equation*}
Then for all $ r \in (0,1) $, there exists $ A_{0} = A_{0}(n, \lambda, \Lambda, i(\Phi), s(\Phi), r, \epsilon_{0}, \mathcal{M}_{1}, \mathcal{M}_{2}, \mathrm{Lip}_{g}(\partial \Omega)) > 0 $, such that if $ |\xi| > A_{0} $, then $ u \in C^{0,1}(B_{r}\cap \{y_{n}>\varphi(y')\}) $ with the estimate
\begin{equation*}
  ||u||_{C^{0,1}(B_{r}\cap \{y_{n}>\varphi(y')\})} \leq \mathrm{C}(n, \lambda, \Lambda, i(\Phi), s(\Phi), r, \epsilon_{0}, \mathcal{M}_{1}, \mathcal{M}_{2}, \mathrm{Lip}_{g}(\partial \Omega)).
\end{equation*}
\end{proposition}

Before proving Proposition~\ref{Section3:Thm5}, we require the following auxiliary lemma.

\begin{lemma}\label{Section3:lemma6}
Suppose that the assumptions \hyperref[A1]{\bf (A1)}--\hyperref[A6]{\bf (A6)} hold. Let $ g$ be a Lipschitz continuous function on $ \partial \Omega $ and $ q \in \mathbb{R}^{n} $ with $ |q| =1 $. Then for every $r \in (0,1)$, there exists $ \widetilde{\mathrm{C}}_{\varphi, B_{r}} $, depending on $ ||f||_{L^{\infty}(B_{1}\cap \{y_{n}>\varphi(y')\})} $, $ \lambda, \Lambda $, $ \beta_{g} $, $ \varphi $, $ \mathrm{Lip}_{g}(\partial \Omega) $, such that for $ 0 \leq b < 1 $, any viscosity solution $ u $ of
\begin{equation*}
\left\{
     \begin{alignedat}{2}
         \Phi(|b Du+q|,y) F(D^{2}u,y)+ H(|b Du+q|, y)  & = f         \quad   &&   \text{in} \quad  B_{1} \cap \{y_{n} > \varphi(y')\}     ,    \\
          u   &  = g    \quad  &&   \text{on} \quad  B_{1} \cap \{y_{n} = \varphi(y')\} ,        \\
     \end{alignedat}
     \right.
\end{equation*}
and $\|u\|_{L^\infty(B_1 \cap \{y_n > \varphi(y')\})} \leq 1$, then
\[
|u(y', y_n) - g(y')| \leq \widetilde{\mathrm{C}}_{\varphi, B_{r}} d(y)   \quad   \text{in } B_r \cap \{y_n > \varphi(y')\}.
\]
\end{lemma}

The proof follows the same arguments as in Lemma \ref{Section3:lemma2}, see also \cite[Lemma 2.4]{BSRR23}. It is omitted here to avoid repetition.

Now we commence with the proof of Proposition~\ref{Section3:Thm5}.

\begin{proof}[Proof of Proposition~\ref{Section3:Thm5}]
Since the proof follows the same line as Proposition~\ref{Section3:Prop4}, we only highlight the differences. Here we define
\begin{equation*}
  \Phi(x,y):= u(x) - u(y) - L_{1} \omega(|x-y|)  - L_{2}(|x-x_{0}|^{2}+|y-x_{0}|^{2}),
\end{equation*}
for $ L_{1}, L_{2} >1 $, where
\[
 \omega(s) =
\begin{cases}
s - \omega_0 s^{1+\beta} & \text{if } 0 \leq s \leq s_0 := \left( \dfrac{1}{(1+\beta)\omega_0} \right)^{1/\beta}, \\[6pt]
\omega(s_0) & \text{if } s > s_0,
\end{cases}
\]
with $\beta \in (0, \theta_{F})$. Here we choose $\omega_0 \in \left(0, \dfrac{1}{1+\beta}\right)$ such that $s_0 \geq 1$. Our goal remains to show that \eqref{Section3:eq20} is true.

(I). If $ y \in B_{r_{1}} \cap \{y_{n} = \varphi(y')\} $, we observe that $ u $ is a viscosity solution to
\begin{equation}\label{Section3:eq25}
\left\{
     \begin{alignedat}{2}
      \widetilde{\Phi}(|b Du+e_{n}|, y)F(D^{2}u, y) +  \widetilde{H}(|b Du+e_{n}|, y)  & = \widetilde{f}(y)  \quad && \text{in} \quad B_{1} \cap \{y_{n}>\varphi(y')\},    \\
          u(y)  & = g(y)    \quad  &&   \text{on}   \ \  B_{1}  \cap \{y_{n} = \varphi(y')\}        \\
     \end{alignedat}
     \right.
\end{equation}
where
\begin{equation*}
 \widetilde{\Phi}(t,y):= \frac{\Phi(|\xi|t, y)}{\Phi(|\xi|, y)}, \quad
 \widetilde{H}(t,y):= \frac{H(|\xi|t, y)}{\Phi(|\xi|, y)}, \quad
 \widetilde{f}(t):=\frac{f(t)}{\Phi(|\xi|, y)},   \quad   b:= \frac{1}{|\xi|}  \quad   \text{and}  \quad e_{n}:= \frac{\xi}{|\xi|}.
\end{equation*}
In fact, if $ \psi \in C^{2}(B_{1} \cap \{y_{n}>\varphi(y')\}) $, and $ x_{0} \in B_{1} \cap \{y_{n}>\varphi(y')\} $ such that $ u - \psi $ has a local minimum at $ x_{0} $, then
\begin{align*}
\widetilde{\Phi}(|b D\psi & +e_{n}|, x_{0})F(D^{2}\psi, x_{0}) +  \widetilde{H}(|b D\psi+e_{n}|, x_{0})   \\
& = \Phi^{-1}(|\xi|, x_{0}) \bigg[\Phi(|D\psi +\xi|, x_{0})F(D^{2}\psi, x_{0}) + H(|D\psi +\xi|, x_{0})     \bigg]  \\
& \leq \Phi^{-1}(|\xi|, x_{0}) f(x_{0}).
\end{align*}
Consequently, $ u $ is a viscosity super-solution to \eqref{Section3:eq25}. Similarly, $ u $ is a viscosity sub-solution to \eqref{Section3:eq25}. In this case, we can show \eqref{Section3:eq20} holds by applying Lemma \ref{Section3:lemma6}.

(II). We next proceed with the contradiction argument from Proposition~\ref{Section3:Prop4}. The distinction emerges in the use of the definitions of limiting superjet and subjet:
\begin{equation*}
\left\{
     \begin{aligned}
& \Phi(|\xi_{\hat{x}}+\xi|, \hat{x}) F(\mathrm{X}, \hat{x}) + H(|\xi_{\hat{x}}+\xi|, \hat{x}) \geq f(\hat{x});      \\
& \Phi(|\xi_{\hat{y}}+\xi|, \hat{y}) F(\mathrm{Y}, \hat{y}) + H(|\xi_{\hat{y}} +\xi|, \hat{y}) \leq f(\hat{y}).
\end{aligned}
     \right.
\end{equation*}
Recall that $ |\xi_{\hat{x}}|, |\xi_{\hat{y}}| \leq 2L_{1} $, then we select $ A_{0}:= 3L_{1} $, one obtain
\begin{equation*}
 |\xi_{\hat{x}} +\xi|, |\xi_{\hat{y}} +\xi|  \geq A_{0} - |\xi_{\hat{x}}|   \geq  3L_{1} -  2L_{1} = L_{1}.
\end{equation*}
Combining this estimate with
\begin{align*}
 & -2L||f||_{\infty} \big(|\xi_{\hat{x}} +\xi|^{-i(\Phi)} + |\xi_{\hat{y}} +\xi|^{-i(\Phi)}    \big)   \\
   &\leq 2L(\mathcal{M}_{1}+\mathcal{M}_{2}) + \mathrm{C}_{1}  + L_{1}|\hat{x}-\hat{y}|^{\beta-1} \big[\mathrm{C}_{F} |\hat{x}-\hat{y}|^{\theta_{F}-\beta}- 4\lambda \beta(\beta+1)\omega_{0}    \big],
\end{align*}
where $ \mathrm{C}_{1}:=  \mathrm{C}_{F}(2L_{2}+1) + [\lambda+\Lambda(n-1)](4L_{2}+2) $.

It follows that
\begin{equation}\label{Section3:eq26}
  -2L||f||_{\infty} \leq 2L(\mathcal{M}_{1}+\mathcal{M}_{2}) + \mathrm{C}_{1}  + L_{1}|\hat{x}-\hat{y}|^{\beta-1} \big[\mathrm{C}_{F} |\hat{x}-\hat{y}|^{\theta_{F}-\beta}- 4\lambda \beta(\beta+1)\omega_{0}    \big].
\end{equation}
Choosing \( L_2 \geq 4\left(\frac{\mathrm{C}_F}{2\lambda\omega_0\beta(1+\beta)}\right)^{2/(\theta_{F}-\beta)} \), we note that \( \beta < \theta_{F} \), then
\[
\mathrm{C}_F |\hat{x} - \hat{y}|^{\theta_{F}-\beta} \leq \mathrm{C}_F\left(\frac{2}{\sqrt{L_2}}\right)^{\theta_{F}-\beta} \leq 2\lambda\omega_0\beta(1+\beta).
\]
Finally, \eqref{Section3:eq26} becomes
\begin{equation*}
2L_{1} \lambda \beta(\beta+1) \omega_{0} \leq \mathrm{C}_{1} + 2L (||f||_{\infty}+ \mathcal{M}_{1}+\mathcal{M}_{2}),
\end{equation*}
which is a contradiction, provided that $ L_{1} $ is large enough.

Hence the proof of Proposition~\ref{Section3:Thm5} is completed.
\end{proof}

Finally, we establish boundary H\"{o}lder regularity for viscosity solutions to \eqref{Section3:eq1} under the degenerate regime with small $ |\xi|$.

\begin{proposition}[{Boundary H\"{o}lder estimates for small $ |\xi|$: degenerate case}]\label{Section3:Thm7}
Let $ g$ be a Lipschitz continuous function. Suppose that \hyperref[A1]{\bf (A1)}--\hyperref[A6]{\bf (A6)} are in force with $ 0 < m \leq i(\Phi) $ and $ \nu_{0} = \nu_{1} =1 $. Let $ u $ be a viscosity solution to \eqref{Section3:eq1} with
\begin{equation*}
  ||u||_{L^{\infty}(B_{1}\cap \{y_{n}>\varphi(y')\})} \leq 1  \quad   \text{and}  \quad    ||f||_{L^{\infty}(B_{1}\cap \{y_{n}>\varphi(y')\})} \leq  \epsilon_{0}.
\end{equation*}
Then for all $ r \in (0,1) $, there exists $ A_{0} = A_{0}(n, \lambda, \Lambda, i(\Phi), s(\Phi), r, \epsilon_{0}, \mathcal{M}_{1}, \mathcal{M}_{2}, \mathrm{Lip}_{g}(\partial \Omega)) > 0 $, such that if $ |\xi| \leq A_{0} $, then $ u \in C^{0,\gamma}(B_{r}\cap \{y_{n}>\varphi(y')\}) $ for some $ \gamma \in (0,1) $ with the estimate
\begin{equation*}
  ||u||_{C^{0,\gamma}(B_{r}\cap \{y_{n}>\varphi(y')\})} \leq \mathrm{C}(n, \lambda, \Lambda, i(\Phi), s(\Phi), r, \epsilon_{0}, \mathcal{M}_{1}, \mathcal{M}_{2}, \mathrm{Lip}_{g}(\partial \Omega)).
\end{equation*}
\end{proposition}

The proof follows the same arguments as in Proposition \ref{Section3:Prop3}. It is omitted here to avoid repetition.

\begin{remark}\label{Section3:rmk3}
From Propositions \ref{Section3:Prop3} and \ref{Section3:Thm7}, we observe that the solution of \eqref{Section3:eq1} possesses H\"{o}lder regularity when $ \xi =0 $ and $ 0 < m \leq 1+i(\Phi)$. This plays a vital role in the proof Lemma \ref{Section4:lem1}, see Section \ref{Section 4} below.
\end{remark}

\section{Proof of Theorem \ref{Thm1}: global $ C^{1, \alpha}$ regularity}\label{Section 4}

In this section, our main goal is to prove Theorems~\ref{Thm1}. The proof of Theorem~\ref{Thm1} consists of two steps. Firstly, we need to establish a boundary version of the approximation lemma for \eqref{Section3:eq1}, see Lemma \ref{Section4:lem1} below. Secondly, by the virtue of delicate iteration and convergence analysis, we provide a complete proof of Theorem~\ref{Thm1}.

\begin{proposition}[{Pointwise boundary $ C^{1,\alpha}$ estimate: degenerate case}]
\label{Section4:prop1}
Suppose that the assumptions \hyperref[A1]{\bf (A1)}--\hyperref[A6]{\bf (A6)} are in force with $ i(\Phi) \geq 0 $ and $ \nu_{0} = \nu_{1} =1 $. Let $ \alpha' $ be chosen to fulfill
\begin{equation*}
  \alpha' \in  \bigg(0, \frac{1}{1+s(\Phi)}\bigg] \cap (0,\beta_{g}).
\end{equation*}
Then there exist constant $ \epsilon_{0} \in (0,1) $, $ 0< \rho < \frac{1}{2} $ and $ \mathrm{C}_{0} > 0 $ depending only on $ \alpha' $, $\lambda$, $\Lambda$, $n $, $ ||D^{2}\varphi||_{\infty}, ||g||_{C^{1,\beta_{g}}(\partial \Omega)} $, $ i(\Phi)$ and $ s(\Phi)$ such that for any $ \xi \in \mathbb{R}^{n} $ and a viscosity solution $ u $ of
\begin{equation*}
\left\{
     \begin{alignedat}{2}
         \Phi(|Du|,y) F(D^{2}u,y)+ H(|Du|, y)  & = f         \quad   &&   \text{in} \quad  B_{1} \cap \{y_{n} > \varphi(y')\}     ,    \\
          u   &  = g    \quad  &&   \text{on} \quad  B_{1} \cap \{y_{n} = \varphi(y')\} ,        \\
     \end{alignedat}
     \right.
\end{equation*}
the following holds: if
\begin{equation*}
  ||u||_{L^{\infty}(B_{1}\cap \{y_{n}>\varphi(y')\})} \leq 1
\end{equation*}
and
\begin{equation*}
  \max \bigg\{||f||_{L^{\infty}(B_{1}\cap \{y_{n}>\varphi(y')\})}, ||\mathrm{osc}_{F}||_{L^{\infty}(B_{1}\cap \{y_{n}>\varphi(y')\})}, \mathcal{M}_{1}, \mathcal{M}_{2} \big(1+|\xi|^{(m-i(\Phi))_{+}}\big)  \bigg\} \leq \epsilon_{0}.
\end{equation*}
Then there exists an affine function $ l(y) = a+ b\cdot y $ with $ |a| + |b| \leq \mathrm{C}_{0} $ such that for each $ 0 < r \leq \rho $,
\begin{equation*}
  ||u-l||_{L^{\infty}(B_{r}\cap \{y_{n}>\varphi(y')\})} \leq \mathrm{C}r^{1+\alpha'},
\end{equation*}
for some positive constant $ \mathrm{C} $.
\end{proposition}

\begin{lemma}[{Approximation lemma: degenerate case}]\label{Section4:lem1}
Suppose that the assumptions \hyperref[A1]{\bf (A1)}--\hyperref[A6]{\bf (A6)} are in force with $ i(\Phi) \geq 0 $ and $ \nu_{0} = \nu_{1} =1 $. Let $ \xi \in \mathbb{R}^{n} $ be an arbitrary vector and $ u $ be a viscosity solution to
\begin{equation*}
\left\{
     \begin{alignedat}{2}
         \Phi(|Du+\xi|,y) F(D^{2}u,y)+ H(|Du+\xi|, y)  & = f         \quad   &&   \text{in} \quad  B_{1} \cap \{y_{n} > \varphi(y')\}     ,    \\
          u   &  = g    \quad  &&   \text{on} \quad  B_{1} \cap \{y_{n} = \varphi(y')\} ,        \\
     \end{alignedat}
     \right.
\end{equation*}
satisfying
\begin{equation*}
  ||u||_{L^{\infty}(B_{1}\cap \{y_{n}>\varphi(y')\})} \leq 1, \quad ||g||_{C^{1,\beta_{g}}(B_{1} \cap \{y_{n} = \varphi(y')\} )} \leq 1.
\end{equation*}
Then for any $ \sigma > 0 $, there exists a constant $ \delta_{0} = \delta_{0}(n, \lambda, \Lambda, L, i(\Phi), \sigma) > 0 $ such that if
\begin{equation*}
  \max \bigg\{||f||_{L^{\infty}(B_{1}\cap \{y_{n}>\varphi(y')\})}, ||\mathrm{osc}_{F}||_{L^{\infty}(B_{1}\cap \{y_{n}>\varphi(y')\})}, \mathcal{M}_{1}, \mathcal{M}_{2} \big(1+|\xi|^{(m-i(\Phi))_{+}}\big)  \bigg\} \leq \delta_{0},
\end{equation*}
then we can find a function $ h$ of homogeneous $ (\lambda, \Lambda)$-uniform elliptic equation
\begin{equation}\label{Section4:eq1}
\left\{
     \begin{alignedat}{2}
       F(D^{2}h) & = 0       \quad    && \text{in} \ \ B_{\frac{4}{5}} \cap \{y_{n} > \varphi(y')\}    ,        \\
          h(y)  & = g(y)    \quad   &&   \text{on}   \ \  B_{\frac{4}{5}}  \cap \{y_{n} = \varphi(y')\},        \\
     \end{alignedat}
     \right.
\end{equation}
such that
\begin{equation*}
  ||u-h||_{L^{\infty}({B_{1/2}\cap \{y_{n}>\varphi(y')\}})} \leq \sigma.
\end{equation*}
\end{lemma}

\begin{proof}
We argue by contradiction, suppose that the conclusion does not hold, then there exist sequences $ \{\Phi_{j}\}_{j \in \mathbb{N}}, \{\xi_{j}\}_{j \in \mathbb{N}}, \{H_{j}\}_{j \in \mathbb{N}}, \{u_{j}\}_{j \in \mathbb{N}} $, $ \{F_{j}\}_{j \in \mathbb{N}}, \{f_{j}\}_{j \in \mathbb{N}}, \{g_{j}\}_{j \in \mathbb{N}}$, $ \{\mathcal{M}_{1,j} \}_{j \in \mathbb{N}}$, $ \{\mathcal{M}_{2,j}\}_{j \in \mathbb{N}} $ and a positive number $ \sigma_{0}  $ such that, for every $ j \in \mathbb{N} $, we have

\label{D1} $ {(D1)} $. $ F_{j}: \mathrm{Sym}(n) \rightarrow \mathbb{R}  $ is a $ (\lambda, \Lambda) $-elliptic operator;

\label{D2}  $ {(D2)} $. $ f_{j} \in C^{0}(B_{1}\cap \{y_{n}>\varphi(y')\}) $ with
$$ \max \{||f_{j}||_{L^{\infty}(B_{1}\cap \{y_{n}>\varphi(y')\})}, ||\mathrm{osc}_{F_{j}}||_{L^{\infty}(B_{1}\cap \{y_{n}>\varphi(y')\})} \} \leq 1/j; $$

\label{D3}  $ {(D3)} $. \(\Phi_j \in C^{0}([0, \infty) \times B_1, [0, \infty))\) such that the map \(t \mapsto \frac{\Phi_j(t,x)}{t^{i(\Phi)}}\) is almost non-decreasing and the map \(t \mapsto \frac{\Phi_j(t,x)}{t^{s(\Phi)}}\) is almost non-increasing with the same constant \(L \geq 1\), and \(\Phi_j(1,y) = 1\) for all \(y \in B_1\);

\label{D4}  $ {(D4)} $. \(H_j : \mathbb{R}^n \times B_1 \to \mathbb{R}\) is continuous and there exist constants \(\mathcal{M}_{1,j}, \mathcal{M}_{2,j} > 0\) such that
\begin{equation*}
  |H_{j}(t,x)| \leq \mathcal{M}_{1,j} + \mathcal{M}_{2,j}|t|^m \quad  \text{for every} \quad t \in \mathbb{R}^n, x \in B_{1},
\end{equation*}
 and
\begin{equation*}
   \max \{ \mathcal{M}_{1,j}, \mathcal{M}_{2,j} \big(1+|\xi_{j}|^{(m-i(\Phi))_{+}}\big)   \} \leq 1/j;
\end{equation*}

\label{D5} $ {(D5)} $. $ u_{j} \in C^{0}(B_{1}\cap \{y_{n}>\varphi(y')\}) $ with $ ||u_{j}||_{L^{\infty}(B_{1}\cap \{y_{n}>\varphi(y')\})} \leq 1 $ solves the equation
\begin{equation*}
\left\{
     \begin{alignedat}{2}
        \Phi_{j}(|Du_{j}+\xi_{j}|,y) F_{j}(D^{2}u_{j},y)+ H_{j}(|Du_{j}+\xi_{j}|, y) & = f_{j}(y)        \quad  &&  \text{in} \ \ B_{1} \cap \{y_{n} > \varphi(y')\}    ,    \\
          u_{j}(y)  & = g_{j}(y)    \quad  && \text{on}   \ \  B_{1}  \cap \{y_{n} = \varphi(y')\},        \\
     \end{alignedat}
     \right.
\end{equation*}
with $ ||g_{j}||_{C^{1,\beta_{g}}(B_{1} \cap \{y_{n} = \varphi(y')\} )} \leq 1 $, but
\begin{equation}\label{Section4:eq2}
||u_{j}-h||_{L^{\infty}({B_{1/2}\cap \{y_{n}>\varphi(y')\}})} \geq \sigma_{0}
\end{equation}
for any $ h $ fulfilling \eqref{Section4:eq1}.

The condition~\hyperref[D1]{(D1)} and Lemma~\ref{Section2:lem2} imply that $ F_{j} \rightarrow F_{\infty} $ for some uniformly elliptic operator $ F_{\infty}$. Analogously, from  \hyperref[D5]{(D5)} it is readily seen that $ g_{j} \rightarrow g_{\infty} $ uniformly. In addition, from \hyperref[D2]{(D2)} and \hyperref[D4]{(D4)}, we can easily see that $ f_{j} $ and $ \mathcal{M}_{k,j}, k=1,2 $ converge uniformly to $ 0$. In what follows, we distinguish two cases.

\smallskip
\noindent
\emph{Case 1}. If $ \{\xi_{j}\}_{j \in \mathbb{N}} $ is bounded sequence, then up to a subsequence, we have that $ \xi_{j} \rightarrow  \xi_{\infty}  $ for some $ \xi_{\infty} \in \mathbb{R}^{n} $. Now we take a sequence $ \widetilde{u_{j}}(y):= u_{j}(y) + y\cdot \xi_{j}  $, it follows that $ \widetilde{u_{j}} $ solves
\begin{equation*}
\left\{
     \begin{alignedat}{2}
       \Phi_{j}(|D\widetilde{u_{j}}|,y) F_{j}(D^{2}\widetilde{u_{j}},y)+ H_{j}(|\widetilde{u_{j}}|, y) & = f_{j}(y) \quad && \text{in} \ \ B_{1} \cap \{y_{n} > \varphi(y')\} , \\
      \widetilde{u_{j}}(y) & = \widetilde{g_{j}}(y) \quad && \text{on} \ \ B_{1} \cap \{y_{n} = \varphi(y')\} ,
     \end{alignedat}
     \right.
\end{equation*}
where $ \widetilde{g_{j}}(y):= g_{j}(y)  + y\cdot \xi_{j}  $. Applying Remark \ref{Section3:rmk3} for $ \widetilde{u_{j}} $ and then using Arzel$ \grave{a}$-Ascoli compactness criterion, we give $ u_{j} \rightarrow u_{\infty} $ uniformly in $ B_{r} \cap \{y_{n} > \varphi(y')\} $ for any $ 0 < r <1 $. By the stability result of viscosity solution (Lemma \ref{Section2:lem2}), $ u_{\infty} $ satisfies
 \begin{equation*}
\left\{
     \begin{alignedat}{2}
       F(D^{2}u_{\infty}) &= 0       \quad   &&  \text{in} \ \ B_{\frac{4}{5}} \cap \{y_{n} > \varphi(y')\}    ,        \\
          u_{\infty}(y)  &= g_{\infty}(y)    \quad  &&  \text{on}   \ \  B_{\frac{4}{5}}  \cap \{y_{n} = \varphi(y')\}.        \\
     \end{alignedat}
     \right.
\end{equation*}
 which leads to a contradiction with \eqref{Section4:eq2} provided that we choose $ h = u_{\infty} , g = g_{\infty}, F = F_{\infty} $.

\smallskip
\noindent
\emph{Case 2}. If $ \{\xi_{j}\}_{j \in \mathbb{N}} $ is unbounded sequence, we may assume that $ |\xi_{j}| \rightarrow \infty $ and $ |\xi_{j}| > A_{0} $, where $ A_{0} $ is the constant appearing in Proposition \ref{Section3:Thm5}. Applying Proposition \ref{Section3:Thm5} for $ u_{j}$ and exploiting the stability result of viscosity solution again (Lemma \ref{Section2:lem2}), we also obtain
\begin{equation*}
\left\{
     \begin{alignedat}{2}
       F(D^{2}u_{\infty}) &= 0       \quad   &&  \text{in} \ \ B_{\frac{4}{5}} \cap \{y_{n} > \varphi(y')\}    ,        \\
          u_{\infty}(y)  &= g_{\infty}(y)    \quad  &&  \text{on}   \ \  B_{\frac{4}{5}}  \cap \{y_{n} = \varphi(y')\},        \\
     \end{alignedat}
     \right.
\end{equation*}
which gives a contradiction as well, similar to that in {\em Case 1}.
\end{proof}

In the sequel, we establish a first step in a geometric regularity approach via a recursive iterative process.

\begin{lemma}[{First step}]
\label{Section4:lem2}
Let $ \alpha' $ be chosen to fulfill
\begin{equation*}
  \alpha' \in  \bigg(0, \frac{1}{1+s(\Phi)}\bigg] \cap (0,\beta_{g}).
\end{equation*}
Suppose that the assumptions \hyperref[A1]{\bf (A1)}--\hyperref[A6]{\bf (A6)} are in force with $ i(\Phi) \geq 0 $ and $ \nu_{0} = \nu_{1} =1 $, there exist $ \epsilon_{0} \in (0,1) $ and $ \rho \in (0,\frac{1}{2}) $ depending on $ n, \lambda, \Lambda, \alpha', L, m $ and $ i(\Phi) $ such that for any $ \xi \in \mathbb{R}^{n} $ and $ u $ a viscosity solution to \eqref{Section3:eq1}, the following holds: if
\begin{equation*}
  ||u||_{L^{\infty}(B_{1}\cap \{y_{n}>\varphi(y')\})} \leq 1
\end{equation*}
and
\begin{equation*}
  \max \bigg\{||f||_{L^{\infty}(B_{1}\cap \{y_{n}>\varphi(y')\})}, ||\mathrm{osc}_{F}||_{L^{\infty}(B_{1}\cap \{y_{n}>\varphi(y')\})}, \mathcal{M}_{1}, \mathcal{M}_{2} \big(1+|\xi|^{(m-i(\Phi))_{+}}\big)  \bigg\} \leq \epsilon_{0}.
\end{equation*}
then there exists an affine function $ l_{1}(y):= a_{1}+ b_{1}\cdot y $($ a \in \mathbb{R}$ and $ b \in \mathbb{R}^{n}$) such that
  \begin{equation*}
  ||u-l_{1}||_{L^{\infty}(B_{\rho}\cap \{y_{n}>\varphi(y')\})} \leq \rho^{1+\alpha'},
\end{equation*}
and
\begin{equation*}
  |a_{1}| + |b_{1}| \leq \mathrm{C}(n,\lambda, \Lambda).
\end{equation*}
\end{lemma}

\begin{proof}
From Lemma \ref{Section4:lem1}, we have that for $ \rho $ small,
\begin{equation}\label{Section4:eq17}
  ||u-h||_{L^{\infty}({B_{\rho}\cap \{y_{n}>\varphi(y')\}})} \leq  \sigma
\end{equation}
for some $ \sigma $, to be determined later. In terms of the pointwise boundary regularity of \eqref{Section4:eq1} (\cite[Theorem 2.1]{AS23}), there exist affine function $ l_{1} $ and universal constant $ \mathrm{C} >0 $ such that
\begin{equation}\label{Section4:eq18}
  \sup_{y \in B_{\rho}\cap \{y_{n}>\varphi(y')\}}|h(y) - l_{1}(y)| \leq \mathrm{C} \rho^{1+\beta_{g}},  \quad  l_{1}(y):= h(0) + Dh(0)\cdot y.
\end{equation}
To proceed, a combination of the triangle inequality with \eqref{Section4:eq17} and \eqref{Section4:eq18} yields
\begin{align}\label{Section4:eq19}
\begin{split}
\sup_{y \in B_{\rho}\cap \{y_{n}>\varphi(y')\}}|u(y) - l_{1}(y)| & \leq \sup_{y \in B_{\rho}\cap \{y_{n}>\varphi(y')\}}|u(y) - h(y)| + \sup_{y \in B_{\rho}\cap \{y_{n}>\varphi(y')\}}|h(y) - l_{1}(y)|  \\
& \leq  \sigma +  \mathrm{C} \rho^{1+\beta_{g}}.
\end{split}
\end{align}
Note that $ 0< \alpha' < \beta_{g} $, then we can choose $ 0 < \rho <1 $ small enough such that
\begin{equation}\label{Section4:eq20}
 \mathrm{C} \rho^{\beta_{g}-\alpha'} \leq \frac{1}{2}.
\end{equation}
Meanwhile, we set
$$ \sigma = \frac{1}{2} \rho^{1+\alpha'}, $$
which, together with \eqref{Section4:eq19} and \eqref{Section4:eq20}, reveals that
\begin{equation*}
  \sup_{y \in B_{\rho}\cap \{y_{n}>\varphi(y')\}}|u(y) - l_{1}(y)| \leq  \rho^{1+\alpha'}.
\end{equation*}
Here $ a_{1} = h(0) $ and $ b_{1} = Dh(0) $. Thus we finish the proof of Lemma \ref{Section4:lem2}.
\end{proof}

Next, we present $ j $th Iterative schemes to derive the desired sharp boundary $ C^{1,\alpha} $ regularity.

\begin{lemma}[{Iterative schemes--$j$th}]\label{Section4:lem3}
Let $ \alpha' $ be fixed as in Lemma \ref{Section4:lem2}. Suppose that the assumptions \hyperref[A1]{\bf (A1)}--\hyperref[A6]{\bf (A6)} are in force with $ i(\Phi) \geq 0 $ and $ \nu_{0} = \nu_{1} =1 $, there exist $ \epsilon_{0} \in (0,1) $ and $ \rho \in (0,\frac{1}{2}) $ depending on $ n, \lambda, \Lambda, \alpha', L, m $ and $ i(\Phi) $ such that for any $ \xi \in \mathbb{R}^{n} $ and $ u $ a viscosity solution to
\begin{equation*}
\left\{
     \begin{alignedat}{2}
         \Phi(|Du|,y) F(D^{2}u,y)+ H(|Du|, y)  & = f         \quad   &&   \text{in} \quad  B_{1} \cap \{y_{n} > \varphi(y')\}     ,    \\
          u   &  = g    \quad  &&   \text{on} \quad  B_{1} \cap \{y_{n} = \varphi(y')\} ,        \\
     \end{alignedat}
     \right.
\end{equation*}
with
\begin{equation*}
  ||u||_{L^{\infty}(B_{1}\cap \{y_{n}>\varphi(y')\})} \leq 1
\end{equation*}
and
\begin{equation*}
  \max \bigg\{||f||_{L^{\infty}(B_{1}\cap \{y_{n}>\varphi(y')\})}, ||\mathrm{osc}_{F}||_{L^{\infty}(B_{1}\cap \{y_{n}>\varphi(y')\})}, \mathcal{M}_{1}, \mathcal{M}_{2} \bigg\} \leq \epsilon_{0},
\end{equation*}
then for any $ j \in \mathbb{N} $, there exists a sequence of affine functions $ \{l_{j}(y)\}_{j \in \mathbb{N}} $, where $ l_{j}(y):= a_{j} + b_{j} \cdot y $, fulfilling
\begin{equation}\label{Section4:eq21}
  ||u-l_{j}||_{L^{\infty}(B_{\rho^{j}}\cap \{y_{n}>\varphi(y')\})}  \leq \rho^{j(1+\alpha')}
\end{equation}
and
\begin{equation}\label{Section4:eq22}
|a_{j} - a_{j-1}|  +   \rho^{j-1}|b_{j} - b_{j-1}| \leq \mathrm{C} \rho^{(j-1)(1+\alpha')}
\end{equation}
where $ \mathrm{C}= \mathrm{C}(n, \lambda, \Lambda) $.
\end{lemma}

\begin{proof}
First of all, such an assertion holds for $ j=1 $ by Lemma \ref{Section4:lem2}. Now we suppose that \eqref{Section4:eq21} and \eqref{Section4:eq22} hold true for $ j \geq 1 $, then we shall verify \eqref{Section4:eq21} and \eqref{Section4:eq22} for $ j+1 $. For this purpose, we define an auxiliary function $ u_{j}: B_{1}\cap \{y_{n}>\varphi(y')\} \rightarrow \mathbb{R} $ as
\begin{equation*}
  u_{j}(y):= \frac{u(\rho^{j}y)-l_{j}(\rho^{j}y)}{\rho^{j(1+\alpha')}},
\end{equation*}
then direct calculation yields that $ u_{j} $ solves
\begin{equation}
\label{Section4:eq23}
\left\{
     \begin{alignedat}{2}
         \Phi_{j}(|Du_{j}+\xi_{j}|,y) F_{j}(D^{2}u_{j},y)+ H_{j}(|Du_{j}+\xi_{j}|, y)  & = f_{j}         \quad   &&   \text{in} \quad  B_{1} \cap \{y_{n} > \varphi_{j}(y')\}     ,    \\
          u_{j}   &  = g_{j}    \quad  &&   \text{on} \quad  B_{1} \cap \{y_{n} = \varphi_{j}(y')\} ,         \\
     \end{alignedat}
     \right.
\end{equation}
where
\begin{equation*}
\left\{
     \begin{aligned}
& F_{j}(\mathrm{X}, y):= \rho^{j(1-\alpha')} F(\rho^{j(\alpha'-1)}\mathrm{X}, \rho^{j}y), \quad  \Phi(t,y):= \frac{\Phi(\rho^{j\alpha'}t, \rho^{j}y)}{\Phi(\rho^{j\alpha'}, \rho^{j}y)};      \\
& H_{j}(t,y):= \frac{\rho^{j(1-\alpha')}}{\Phi(\rho^{j\alpha'}, \rho^{j}y)} H(\rho^{j\alpha'}t, \rho^{j}y) ;    \\
& f_{j}(y):=\frac{\rho^{j(1-\alpha')}}{\Phi(\rho^{j\alpha'}, \rho^{j}y)} f(\rho^{j}y)     ,   \quad  g_{j}(y):= \frac{g(\rho^{j}y)-l_{j}(\rho^{j}y)}{\rho^{j(1+\alpha')}};    \\
& \xi_{j}:= b_{j} \rho^{-j\alpha'} , \quad \text{and}  \quad  \varphi_{j}(y'):= \rho^{-j}\varphi(\rho^{j}y') .
\end{aligned}
     \right.
\end{equation*}
One can readily check that
\begin{enumerate}[label=(\roman*)]
\item  $ F_{j} $ satisfies \hyperref[A1]{\bf (A1)} with the same constants $ (\lambda, \Lambda) $;    

\item By induction assumption, it can be seen that $ ||u_{j}||_{L^{\infty}(B_{1}\cap \{y_{n}>\varphi_{j}(y')\})} \leq 1$;

\item Simple calculation yields
\begin{align*}
   \mathrm{osc}_{F_{j}}(y,0) & = \sup_{\mathrm{M} \in \mathrm{Sym}(n)\setminus\{0\}} \frac{|F_{j}(\mathrm{M},y)- F_{j}(\mathrm{M}, 0)|}{||\mathrm{M}||}  \\
  & = \sup_{\mathrm{M} \in \mathrm{Sym}(n)\setminus\{0\}} \frac{|F(\rho^{j(1-\alpha')}\mathrm{M}, \rho^{j}y)-F(\rho^{j(1-\alpha')}\mathrm{M}, 0)|}{||\rho^{j(\alpha'-1)}\mathrm{M}||}  \\
&  = \mathrm{osc}_{F}(\rho^{j}y,0),
\end{align*}
which, together with the hypothesis of Lemma \ref{Section4:lem3}, leads to
\begin{align*}
||\mathrm{osc}_{F_{j}}||_{L^{\infty}(B_{1}\cap \{y_{n}>\varphi(y')\})} = ||\mathrm{osc}_{F}||_{L^{\infty}(B_{\rho^{j}}\cap \{y_{n}>\varphi(y')\})}  \leq  ||\mathrm{osc}_{F}||_{L^{\infty}(B_{1}\cap \{y_{n}>\varphi(y')\})} \leq \epsilon_{0};
\end{align*}

\item $ \Phi_{j} $ satisfies \hyperref[A3]{\bf (A3)} with the same constants $ (i(\Phi), s(\Phi)) $ and $ \Phi_{j}(y,1) \equiv 1 $;

\item By the virtue of the properties of $ \Phi_{j} $, \hyperref[A4]{\bf (A4)} and $ \rho \in (0,\frac{1}{2}) $, we derive
    \begin{equation*}
      |H_{j}(t,x)| \leq \frac{L\rho^{j(1-\alpha')}}{\rho^{j\alpha' s(\Phi)}} (\mathcal{M}_{1}+ \mathcal{M}_{2} \rho^{j \alpha' m}|t|^{m}) : = \mathcal{M}_{1;j}  +  \mathcal{M}_{2;j}|t|^{m}.
    \end{equation*}
Now we need to estimate the term $ \mathcal{M}_{2;j}(1+|\xi_{j}|^{(m-i(\Phi))_{+}})$. Recall that $ \xi_{j} = b_{j} \rho^{-j\alpha'} $, then
\begin{equation}\label{Section4:eq24}
  \mathcal{M}_{2;j}(1+|\xi_{j}|^{(m-i(\Phi))_{+}}) \leq L \epsilon_{0} \rho^{j[1-\alpha'(1+s(\Phi)-m)]-j\alpha'(m-i(\Phi))_{+}} + (|b_{j}|^{(m-i(\Phi))_{+}}+1).
\end{equation}
We distinguish two cases.

\smallskip
\noindent
\emph{Case 1}. If $ 0 < m \leq i(\Phi) $, using $ 1-\alpha'(1+s(\Phi)-m) > 0 $, then we have
\begin{equation}\label{Section4:eq25}
  \mathcal{M}_{2;j}(1+|\xi_{j}|^{(m-i(\Phi))_{+}}) \leq 2L \epsilon_{0} \rho^{j[1-\alpha'(1+s(\Phi)-m)]} \leq 2L \epsilon_{0};
\end{equation}
\smallskip
\noindent
\emph{Case 2}. If $ i(\Phi) < m \leq 1+i(\Phi)$, by using $ 1-\alpha'(1+s(\Phi)-i(\Phi)) > 0 $, then
\begin{equation}\label{Section4:eq26}
  \mathcal{M}_{2;j}(1+|\xi_{j}|^{(m-i(\Phi))_{+}}) \leq L \epsilon_{0} \rho^{j[1-\alpha'(1+s(\Phi)-i(\Phi))] }(1+|b_{j}|^{m-i(\Phi)})  \leq L \epsilon_{0}(1+|b_{j}|^{m-i(\Phi)}).
\end{equation}
Besides, from \eqref{Section4:eq22} and $ \rho < \frac{1}{2} $, it follows
\begin{equation*}
  |b_{j}| \leq |b_{0}|  +  \sum_{k=1}^{j}|b_{k}-b_{k-1}| \leq \mathrm{C} \sum_{k=1}^{j} \rho^{\alpha'(k-1)}  \leq \frac{\mathrm{C}}{1-\rho^{\alpha'}}  \leq 2\mathrm{C},
\end{equation*}
which together with \eqref{Section4:eq26} to obtain
\begin{equation}\label{Section4:eq27}
   \mathcal{M}_{2;j}(1+|\xi_{j}|^{(m-i(\Phi))_{+}}) \leq L \epsilon_{0}(1+(2\mathrm{C})^{m-i(\Phi)}).
\end{equation}
We combine \eqref{Section4:eq25} and \eqref{Section4:eq27} to give that for $ 0 < m \leq 1+i(\Phi) $,
\begin{equation*}
  \mathcal{M}_{2;j}(1+|\xi_{j}|^{(m-i(\Phi))_{+}}) \leq L \epsilon_{0}(1+(2\mathrm{C})^{(m-i(\Phi))_{+}}) \leq \delta_{0},
\end{equation*}
provided $ \epsilon_{0} $ is small enough, where $ \delta_{0} $ occurs in the statement of Lemma \ref{Section4:lem1};

\item Applying the property of $ \Phi $ and $ 0 < \alpha' \leq \frac{1}{1+ s(\Phi)} $, it follows
\begin{equation*}
  ||f_{j}||_{L^{\infty}(B_{1} \cap \{y_{n} > \varphi_{j}(y')\})}  \leq \frac{L\rho^{j(1-\alpha')}}{\rho^{j\alpha' s(\Phi)}} ||f||_{L^{\infty}(B_{1} \cap \{y_{n} > \varphi_{j}(y')\})} \leq L \epsilon_{0} \leq \delta_{0};
\end{equation*}

\item It is easy to see $ \mathcal{M}_{1;j} \leq L \rho^{j[1-\alpha'(1+s(\Phi))]} \mathcal{M}_{1} \leq L \epsilon_{0} \leq \delta_{0} $;

\item  $ ||D^{2}\varphi_{j}||_{L^{\infty}(B_{1}\cap \{y_{n}>\varphi_{j}(y')\})}  \leq \rho^{j}||D^{2}\varphi||_{L^{\infty}(B_{1}\cap \{y_{n}>\varphi(y')\})} \leq  ||D^{2}\varphi||_{L^{\infty}(B_{1}\cap \{y_{n}>\varphi(y')\})} $;

\item  A straightforward computation yields that for any $ y,z \in B_{1}  \cap \{y_{n} = \varphi_{j}(y')\} $,
\begin{align*}
|Dg_{j}(y) - Dg_{j}(z)| & = \rho^{-j\alpha'} |Dg(\rho^{j}y) - Dg(\rho^{j}z)|  \leq \rho^{j(\beta_{g}-\alpha')} ||g||_{C^{1,\beta_{g}}(B_{1}\cap \{y_{n}>\varphi(y')\})} |y-z|^{\beta_{g}}       \\
&  \leq ||g||_{C^{1,\beta_{g}}(B_{1}\cap \{y_{n}=\varphi(y')\})} |y-z|^{\beta_{g}}  \quad  \text{for} \  \rho  \ \text{small enough},
\end{align*}
then we have that
\begin{equation*}
  ||g_{j}||_{C^{1,\beta_{g}}(B_{1}\cap \{y_{n}>\varphi_{j}(y')\})} \leq ||g||_{C^{1,\beta_{g}}(B_{1}\cap \{y_{n}=\varphi(y')\})} \leq 1.
\end{equation*}
\end{enumerate}

Consequently, as discussed above, we arrive at
 \begin{equation*}
  \max \bigg\{||f_{j}||_{L^{\infty}(B_{1}\cap \{y_{n}>\varphi_{j}(y')\})}, ||\mathrm{osc}_{F_{j}}||_{L^{\infty}(B_{1}\cap \{y_{n}>\varphi_{j}(y')\})}, \mathcal{M}_{1}, \mathcal{M}_{2} \big(1+|\xi_{j}|^{(m-i(\Phi))_{+}}\big)  \bigg\} \leq \delta_{0},
\end{equation*}

Now we apply Lemma \ref{Section4:lem1} for $ u_{j} $ and then follow the argument in the first step (Lemma \ref{Section4:lem2}) to ensure the existence of an affine function $ \widehat{l}(y):= \widehat{a} + \widehat{b}\cdot y $ such that
\begin{equation*}
  ||u_{j}-\widehat{l}||_{L^{\infty}(B_{\rho}\cap \{y_{n}>\varphi(y')\})} \leq \rho^{1+\alpha'} \quad  \text{and}  \quad  |\widehat{a}| + |\widehat{b}|  \leq \mathrm{C}(n, \lambda, \Lambda).
\end{equation*}

By scaling back, we conclude that
\begin{equation*}
  ||u(y)-l_{j+1}(y)||_{L^{\infty}(B_{\rho^{j+1}}\cap \{y_{n}>\varphi(y')\})}  \leq \rho^{(j+1)(1+\alpha')},
\end{equation*}
where
\begin{equation*}
\left\{
     \begin{aligned}
& l_{j+1}(y):= l_{j}(y)  +  \rho^{j(1+\alpha')} \widehat{l}(\rho^{-j}y);      \\
& a_{j+1} :=a_{j} + \rho^{j(1+\alpha')} \widehat{a} ;    \\
& b_{j+1} :=b_{j} +\rho^{j\alpha'} \widehat{b}.
\end{aligned}
     \right.
\end{equation*}
This completes the proof of Lemma \ref{Section4:lem3}.
\end{proof}

Now we commence with the proof of Proposition \ref{Section4:prop1}.

\begin{proof}[{Proof of Proposition \ref{Section4:prop1}}]
By the standard argument, see \cite[Corollary 4.1]{BSRR23}, the proof of Proposition \ref{Section4:prop1} is complete.
\end{proof}

Finally, we can give

\begin{proof}[{Proof of Theorem \ref{Thm1}}]
The proof is divided into two cases.

\smallskip
\noindent
\emph{Case 1}. $ i(\Phi) \geq  0 $. The result of Theorem~\ref{Thm1} can be completed by combining Proposition \ref{Section4:prop1}, Theorem \ref{Section2:thm3} and the argument in the proof of \cite[Theorem 1.1]{AS23}.

\smallskip
\noindent
\emph{Case 2}. $ -1 < i(\Phi) <0 $. From Proposition \ref{Section3:Prop4}, it can be readily seen
\begin{equation*}
  ||u||_{C^{0,1}((B_{3/4}\cap \{y_{n}>\varphi(y')\}))}  \leq \mathrm{C},
\end{equation*}
for a universal constant $ \mathrm{C} $. By applying Lemma~\ref{Section2:lemma33}, then $ u $ is a viscosity solution to
\begin{equation*}
\left\{
     \begin{alignedat}{2}
         \widetilde{\Phi}(|Du|,y) F(D^{2}u,y)+ \widetilde{H}(|Du|, y)  & = \widetilde{f}         \quad   &&   \text{in} \quad  B_{3/4} \cap \{y_{n} > \varphi(y')\}     ,    \\
          u   &  = g    \quad  &&   \text{on} \quad  B_{3/4} \cap \{y_{n} = \varphi(y')\} ,        \\
     \end{alignedat}
     \right.
\end{equation*}
where
\begin{equation*}
\left\{
     \begin{aligned}
& \widetilde{\Phi}(t, y):= t^{-i(\Phi)} \Phi(t,y);      \\
& \widetilde{H}(t,y):= t^{-i(\Phi)} H(t, y);    \\
& \widetilde{f}(y):= t^{-i(\Phi)} f(y).
\end{aligned}
     \right.
\end{equation*}
Here $ \widetilde{\Phi} $ satisfies the assumption \hyperref[A3]{\bf (A3)} with the same constant $ L \geq 1 $, $ \widetilde{\Phi}(1,y) = 1 $ for all $ y \in B_{1} $, and
\begin{equation*}
  ||\widetilde{f}||_{L^{\infty}(B_{3/4}\cap \{y_{n}>\varphi(y')\})}  \leq  \mathrm{C}^{-i(\Phi)}||f||_{L^{\infty}(\Omega)}.
\end{equation*}
Thus, one can repeat the argument in the proof of Proposition \ref{Section4:prop1} to obtain the global $ C^{1,\alpha}$ estimate.
\end{proof}

\section{Proof of Theorem \ref{Thm2}: boundary $ C^{1, \beta'}$ regularity}\label{Section 5}
This section  is devoted to establishing the optimal boundary $ C^{1,\beta'} $ regularity of viscosity solutions to \eqref{Section1:eq5}. The proof is built upon a refined localized analysis and a geometric iteration scheme designed to capture the optimal oscillation decay of solutions near the boundary. Hereafter, for $ r > 0 $, we recall the notation
\begin{equation*}
  \Omega_{r}^{+}:= \Omega \cap B_{r}  \quad  \text{and}  \quad \Omega'_{r}:= \partial \Omega \cap B_{r}.
\end{equation*}

Here is the main result of this section.

\begin{theorem}
\label{Section5:Thm1}
Suppose that the assumptions \hyperref[A1]{\bf (A1)}, \hyperref[A2]{\bf (A2)}, \hyperref[A3a]{\bf (A3a)}, \hyperref[A4a]{\bf (A4a)} and \hyperref[A5a]{\bf (A5a)} hold. Let $ u $ be a viscosity solution to \eqref{Section1:eq5}, $ u(0) = 0 $, and $ \beta'$ be as in \eqref{Section1:eq6}. There exist small positive numbers $ \rho_{0} $ and $ \delta_{0} $, such that if
\begin{equation}\label{Section5:eq1}
  |Du(0)| \leq \delta_{0} \rho^{\beta'},
\end{equation}
for some $ 0 < \rho \leq \rho_{0} $, then it holds
\begin{equation}\label{Section5:eq2}
  \sup_{\Omega_{\rho}^{+}}|u(x)| \leq  \mathrm{C}A' \rho^{1+\beta'},
\end{equation}
where $ A' $ is as in Theorem \ref{Thm2}.
\end{theorem}

To prove this theorem, we first need the following approximation lemma involving boundary case,

\begin{lemma}\label{Section5:Lem1}
Under the hypotheses of Theorem \ref{Section5:Thm1}, let $ u  $ be a viscosity solution to \eqref{Section1:eq5}, $ u(0) = 0 $, and $ ||u||_{L^{\infty}(\Omega)} \leq 1 $, there exists a small constant $ \rho_{0} > 0  $ such that for every $ 0 < \rho < \rho_{0} $ we can find $ \delta > 0 $ for which
\begin{equation}
\label{Section5:lem5:eq1}
  \max \bigg\{ |Du(0)|, ||\mathrm{osc}_{F}||_{L^{\infty}(\Omega_{1}^{+})}, ||f||_{L^{\infty}(\Omega_{1}^{+})}, ||h||_{L^{\infty}(\Omega_{1}^{+})},    ||g||_{C^{1,\beta_{g}}(\Omega_{1}^{'})}      \bigg\}  \leq \delta,
\end{equation}
then there holds
\begin{equation*}
  \sup_{\Omega_{\rho}^{+}}|u(x)| \leq  \rho^{1+\beta'}.
\end{equation*}
\end{lemma}

\begin{proof}
We argue by contradiction, suppose that for some $ \rho > 0 $, there exist sequences $\{u_{j}\}_{j \in \mathbb{N}} $, $ \{F_{j}\}_{j \in \mathbb{N}}$, $\{f_{j}\}_{j \in \mathbb{N}}$, $\{g_{j}\}_{j \in \mathbb{N}} $, $ \{h_{j}\}_{j \in \mathbb{N}}$, $ \{p_{j}\}_{j \in \mathbb{N}}$, $ \{q_{j}\}_{j \in \mathbb{N}}$ and $ \{a_{j}\}_{j \in \mathbb{N}}$ such that $ u_{j} $ solves
\begin{equation}\label{Section5:eq3}
  \big[|Du_{j}|^{p_{j}(x)}+ a_{j}(x)|Du_{j}|^{q_{j}(x)}\big]F_{j}(D^{2}u_{j},x)+ h_{j}(x)|Du_{j}|^{m} = f_{j}(x)         \quad      \text{in} \ \ \Omega_{1}^{+},
\end{equation}
with $ u_{j}(0) = 0 $, and for each $ j \in \mathbb{N} $, it holds
\begin{equation}\label{Section5:eq4}
  \max \bigg\{ |Du_{j}(0)|, ||\mathrm{osc}_{F_{j}}||_{L^{\infty}(\Omega_{1}^{+})}, ||f_{j}||_{L^{\infty}(\Omega_{1}^{+})}, ||h_{j}||_{L^{\infty}(\Omega_{1}^{+})},    ||g_{j}||_{C^{1,\beta_{g}}(\Omega_{1}^{'})}      \bigg\}  \leq \frac{1}{j},
\end{equation}
but
\begin{equation}\label{Section5:eq5}
  \sup_{\Omega_{\rho}^{+}}|u_{j}(x)| >  \rho^{1+\beta'}.
\end{equation}
With the aid of Theorem \ref{Thm1} and Arzel$ \grave{a}$-Ascoli theorem, we know that $ u_{j} $ converges (up to a subsequence) locally uniformly in $ \Omega_{1}^{+} $ to a continuous function $ u_{\infty} $ in $ C^{1} $-topology. From \hyperref[A1]{\bf (A1)} and \eqref{Section5:eq4}, we see $ F_{j}(\mathrm{X}, x) \rightarrow F_{\infty}(\mathrm{X})$ locally uniformly in $ \mathrm{Sym}(n) \times \Omega_{1}^{+} $ for some uniformly elliptic operator $  F_{\infty} $. Using the stability of viscosity solution and cutting lemma (Lemma \ref{Section 2:lem1}), we get $ u_{\infty} $ solves
\begin{equation*}
\left\{
     \begin{alignedat}{2}
        F_{\infty}(D^{2}u_{\infty}) & = 0        \quad   &&   \text{in} \quad  \Omega_{2/3}^{+}    ,    \\
          u_{\infty}   &  = 0    \quad  &&   \text{on} \quad  \Omega_{2/3}^{'},        \\
          |Du_{\infty}(0)| &  = 0.    \\
     \end{alignedat}
     \right.
\end{equation*}
Hence, \cite[Theorem 2.1]{AS23} gives that for some $ \mathrm{C}_{0} >0 $, depending only on $ n, \lambda, \Lambda $, such that for any $ 0 < \rho \leq \rho_{0} $,
\begin{equation*}\label{}
  \sup_{\Omega_{\rho}^{+}}|u_{\infty}(x)| \leq \mathrm{C}_{0} \rho^{2}  \leq \frac{1}{2}\rho^{1+\beta'},
\end{equation*}
where we choose a suitable $ \rho $ satisfying $ 0 < \rho \leq (2C_{0})^{-\frac{1}{1-\beta'}} $. Consequently, for $ j$ large enough, we have
\begin{equation*}
  \sup_{\Omega_{\rho}^{+}}|u_{j}(x)| \leq \sup_{\Omega_{\rho}^{+}}|u_{j}(x)-u_{\infty}(x)|  + \sup_{\Omega_{\rho}^{+}}|u_{\infty}(x)| \leq  \frac{1}{2}\rho^{1+\beta'} + \frac{1}{2}\rho^{1+\beta'} = \rho^{1+\beta'},
\end{equation*}
which contradicts with \eqref{Section5:eq5}.
\end{proof}

We then iterate the preceding estimate so as to control the oscillation of the solutions on $ \rho_{0}$-adic balls.

\begin{proposition}\label{Section5:prop1}
Under the hypotheses of Theorem \ref{Section5:Thm1}, let $ u  $ be a viscosity solution to \eqref{Section1:eq5}, $ u(0) = 0 $, $ ||u||_{L^{\infty}(\Omega)} \leq 1 $ and $ \beta'$ be as in \eqref{Section1:eq6}. there exists a small constant $ \rho_{0} > 0  $ and $ \delta_{0} > 0 $ for which
\begin{equation}\label{Section5:eq6}
  \max \bigg\{||\mathrm{osc}_{F}||_{L^{\infty}(\Omega_{1}^{+})}, ||f||_{L^{\infty}(\Omega_{1}^{+})}, ||h||_{L^{\infty}(\Omega_{1}^{+})},    ||g||_{C^{1,\beta_{g}}(\Omega_{1}^{'})}      \bigg\}  \leq \delta_{0},
\end{equation}
and for some $ j \in \mathbb{N} $,
\begin{equation}\label{Section5:eq7}
|Du(0)| \leq  \delta_{0} \rho_{0}^{j \beta'},
\end{equation}
then
\begin{equation}\label{Section5:eq8}
   \sup_{\Omega_{\rho_{0}^{j}}^{+}}|u(x)| \leq  \rho_{0}^{j(1+\beta')}.
\end{equation}
\end{proposition}

\begin{proof}
Under assumption \eqref{Section5:eq6}, we use induction to prove that for any positive integer $ j $, validity of estimate \eqref{Section5:eq7} yields estimate \eqref{Section5:eq8}. The case $ j=1 $ follows directly from Lemma \ref{Section5:Lem1}. Suppose that \eqref{Section5:eq8} is true for $ j = k $, now we assert the same is true for $ j= k+1 $. In effect, we denote $ r_{k}:= \rho_{0}^{k} $ and an auxiliary function
\begin{equation*}
  u_{k}(x):= \frac{u(r_{k}x)}{r_{k}^{1+\beta'}}  \qquad   \text{in}  \qquad    \Omega_{1,k}^{+}:= B_{1} \cap \{ y \in \mathbb{R}^{n}: y_{n} > \varphi_{k}(y') \},
\end{equation*}
then a straightforward computation shows that $u_{k} $ solves
\begin{equation}
\label{Section5:eq9}
\left\{
     \begin{alignedat}{2}
           [|Du_{k}|^{p_{k}(y)}+a_{k}(y)|Du_{k}|^{q_{k}(y)}] F_{k}(D^{2}u_{k}, y) + h_{k}(y)|Du_{k}|^{m} & = f_{k}         \quad   &&   \text{in} \quad  \Omega_{1,k}^{+}     ,    \\
          u_{k}   &  = g_{k}    \quad  &&   \text{on} \quad \Omega_{1,k}^{'} ,         \\
     \end{alignedat}
     \right.
\end{equation}
where
\begin{equation*}
\left\{
     \begin{aligned}
& F_{k}(\mathrm{X}, y):= r_{k}^{1-\beta'} F(r_{k}^{\beta'-1}\mathrm{X}, r_{k}y);          \\
& h_{k}(y):=r_{k}^{1+\beta'[m-p_{k}(y)-1]}h(r_{k}y) ;    \\
& f_{k}(y):= r_{k}^{1-\beta'(1+p_{k}(y))} f(r_{k}y)     ,   \quad  g_{k}(y):= \frac{g(r_{k}y)}{r_{k}^{1+\beta'}};      \\
&  p_{k}(y):= p(r_{k}y),   \quad   q_{k}(y):= q(r_{k}y),  \quad  a_{k}(y):= a(r_{k}y) r_{k}^{\beta'[q_{k}(y)-p_{k}(y)]};     \\
&   \varphi_{k}(y'):= \frac{\varphi(r_{k}y')}{r_{k}}.
\end{aligned}
     \right.
\end{equation*}

Also, we observe that
\begin{enumerate}[label=(\roman*)]

\item $ F_{k} $ satisfies \hyperref[A1]{\bf (A1)} with the same constants, and
\begin{equation*}
  ||\mathrm{osc}_{F_{k}}||_{L^{\infty}(B_{1}\cap \{y_{n}>\varphi_{k}(y')\})}   \leq   ||\mathrm{osc}_{F}||_{L^{\infty}(B_{1}\cap \{y_{n}>\varphi(y')\})} \leq \delta_{0};
\end{equation*}

\item $ |Du_{k}(0)| = r_{k}^{-\beta'} |Du(0)| = \rho_{0}^{-k\beta'} |Du(0)|    \overset{\eqref{Section5:eq7}}{\leq}  \delta_{0};       $

\item $ ||f_{k}||_{L^{\infty}(\Omega_{1,k}^{+})} \leq  r_{k}^{1-\beta'(1+p_{k}(y))+ \theta_{f}} \mathcal{K}_{2}     \leq  \delta_{0}  $, for small $ \rho_{0} $, where we have used the fact:
    \begin{equation*}
     0<  \beta' \overset{\eqref{Section1:eq6}}{\leq} \frac{1+\theta_{f}}{1+p_{\text{max}}};
    \end{equation*}

\item Since $ \beta' \overset{\eqref{Section1:eq6}}{\leq}   \frac{1}{1+p_{\mathrm{max}}-m} $, we have
    $$ ||h_{k}||_{L^{\infty}(\Omega_{1,k}^{+})} \leq r_{k}^{1+\beta'[m-p_{k}(y)-1]} ||h||_{L^{\infty}(\Omega_{1}^{+})}  \leq  ||h||_{L^{\infty}(\Omega_{1}^{+})}  \leq \delta_{0};  $$

\item  It is easy to see $ 0 \leq a_{k} \in C^{0}(\Omega_{1,k}^{+}) $ and $ ||a_{k}||_{L^{\infty}(\Omega_{1,k}^{+})} \leq ||a||_{L^{\infty}(\Omega_{1}^{+})}$ (since $ q_{k} \geq p_{k}$);

\item  A straightforward computation yields that for any $ y,z \in \Omega_{1,k}^{+} $,
\begin{align*}
|Dg_{k}(y) - Dg_{k}(z)| & = \rho_{0}^{-k \beta'} |Dg(\rho_{0}^{k}y) - Dg(\rho_{0}^{k}z)|  \leq \rho_{0}^{k(\beta_{g}-\beta')} ||g||_{C^{1,\beta_{g}}(B_{1}\cap \{y_{n}>\varphi(y')\})} |y-z|^{\beta_{g}}         \\
&  \leq ||g||_{C^{1,\beta_{g}}(B_{1}\cap \{y_{n}=\varphi(y')\})} |y-z|^{\beta_{g}}  \quad  \text{for} \  \rho_{0}  \ \text{small enough},
\end{align*}
then we have that
\begin{equation*}
  ||g_{j}||_{C^{1,\beta_{g}}(B_{1}\cap \{y_{n}>\varphi_{j}(y')\})} \leq ||g||_{C^{1,\beta_{g}}(B_{1}\cap \{y_{n}=\varphi(y')\})} \leq \delta_{0}.
\end{equation*}
\end{enumerate}
Moreover, since \eqref{Section5:eq7} also holds for $ j =k +1 $, then it yields
\begin{equation*}
  |Du_{k}(0)| \leq \delta_{0} \rho_{0}^{(k+1)\beta'}  \leq \delta_{0}.
\end{equation*}
By taking $ \delta = \delta_{0} $, and applying Lemma \ref{Section5:Lem1} to $ u_{k} $, we obtain
\begin{equation*}
   \sup_{\Omega_{\rho_{0}^{k+1}}^{+}}|u(x)| \leq  \rho_{0}^{(k+1)(1+\beta')},
\end{equation*}
which implies that \eqref{Section5:eq8} holds for $ j =k+1 $. Hence we complete the proof of Proposition \ref{Section5:prop1}.
\end{proof}

With these preparations, we are ready to present the proof of Theorem \ref{Section5:Thm1}.

\begin{proof}[{Proof of Theorem \ref{Section5:Thm1}}]
Let $ \rho_{0} $ and $ \delta_{0} $ be as in Proposition \ref{Section5:prop1}. Assume that $ ||u||_{L^{\infty}(\Omega)} \leq 1 $ and the smallness conditions \eqref{Section5:eq6} hold. For given $ 0 < \rho \leq  \rho_{0} $, then there exists $ k \in \mathbb{N} $ such that $ \rho_{0}^{k+1} < \rho \leq \rho_{0}^{k} $. Hence it is easy to see that \eqref{Section5:eq1} implies $ |Du(0)| \leq  \delta_{0} \rho_{0}^{k \beta'} $. By applying Proposition \ref{Section5:prop1}, we have
 \begin{equation*}
   \sup_{\Omega_{\rho}^{+}}|u(x)| \leq \rho_{0}^{k(1+\beta')} = \rho_{0}^{-(1+\alpha)} \rho_{0}^{(k+1)(1+\beta')} \leq \mathrm{C} \rho^{1+\beta'}.
 \end{equation*}
which implies that Theorem \ref{Section5:prop1} is true in this particular scenario.

For the general case, for $ \delta_{0} $ as in \eqref{Section5:eq6}, we define
\begin{equation*}
  \widetilde{u}(x):= \frac{u(rx)}{K}
\end{equation*}
for a positive constant $ K \geq 1 \geq r $ to be determined later. Direct calculations yield that $ \widetilde{u} $ satisfies, in the viscosity sense, the following problem:
\begin{equation*}
\left\{
     \begin{alignedat}{2}
          \big[|D\widetilde{u}|^{\widetilde{p}(x)}  + \widetilde{a}(x)|D\widetilde{u}|^{\widetilde{q}(x)}       \big] \widetilde{F}(D^{2}\widetilde{u},x)+ \widetilde{h}(x)|D\widetilde{u}|^{m}   & = \widetilde{f}(x)            \quad   &&   \text{in} \quad   \Omega_{1,r}^{+}    ,    \\
             \widetilde{u}(x) &  = \widetilde{g}(x)     \quad  &&   \text{on} \quad  \Omega_{1,r}^{'}  ,        \\
     \end{alignedat}
     \right.
\end{equation*}
where
 \begin{equation*}
\left\{
     \begin{aligned}
& \widetilde{F}(\mathrm{X}, x):= \frac{r^{2}}{K} F\left(\frac{K}{r^{2}}\mathrm{X}, rx\right);          \\
& \widetilde{h}(x):=\frac{K^{m-1-\widetilde{p}(x)}}{r^{m-2-\widetilde{p}(x)}}       h(rx) ;    \\
& \widetilde{f}(x):= \frac{r^{2+\widetilde{p}(x)}}{K^{1+\widetilde{p}(x)}} f(rx)     ,   \quad  \widetilde{g}(x):= \frac{u(rx)}{K};      \\
&  \widetilde{p}(x):= p(rx),   \quad   \widetilde{q}(x):= q(rx),  \quad   \widetilde{a}(x):= a(rx) \left(  \frac{K}{r} \right)^{\widetilde{q}(x)-\widetilde{p}(x)};       \\
&  \Omega_{1,r}^{+}:= B_{1} \cap \{ x \in \mathbb{R}^{n}: x_{n} \geq r^{-1} \varphi(r x') \};  \\
&   \Omega_{1,r}^{'}:=  B_{1} \cap \{ x \in \mathbb{R}^{n}: x_{n} =  r^{-1} \varphi(r x') \}.
\end{aligned}
     \right.
\end{equation*}
Arguing as in Remark \ref{Section3:rmk1}, it follows that
 \begin{equation*}
  \max \bigg\{||\mathrm{osc}_{\widetilde{F}}||_{L^{\infty}(\Omega_{1}^{+})}, ||\widetilde{f}||_{L^{\infty}(\Omega_{1,r}^{+})}, ||\widetilde{h}||_{L^{\infty}(\Omega_{1,r}^{+})},    ||\widetilde{g}||_{C^{1,\beta_{g}}(\Omega_{1,r}^{'})}      \bigg\}  \leq \delta_{0},
\end{equation*}
provided we choose the constants $ K $ and $ r $ as below:
\[
K :=
\begin{cases}
1 + \|u\|_{L^\infty(\Omega_{1}^{+})} +  ||h||_{L^\infty(\Omega_{1}^{+})}^{\frac{1}{1+p_{\mathrm{min}}-m}}        +  \mathcal{K}_{2}^{\frac{1}{1+p_{\mathrm{min}}}} +   ||g||_{C^{1,\beta_{g}}(\Omega_{1}')}  & \text{for } \ m < 1 + p_{\mathrm{min}}, \\[1em]
1 + \|u\|_{L^\infty(\Omega_{1}^{+})} + ||g||_{C^{1,\beta_{g}}( \Omega_{1}')} & \text{for } \ m = 1 + p_{\mathrm{min}},
\end{cases}
\]
and
\[
r :=
\begin{cases}
\min\left\{1,  \left(\frac{\delta_{0}}{C_F}\right)^{\frac{1}{\theta_{F}}}, \delta_{0}^{\frac{1}{2+p_{\mathrm{min}}}}, \delta_{0}^{\frac{1}{2+p_{\mathrm{min}}-m}}\right\} & \text{for } \ m < 1 + p_{\mathrm{min}}, \\[1.5em]
\min\left\{1, \left(\frac{\delta_{0}}{C_F}\right)^{\frac{1}{\theta_{F}}}, \left(\frac{\delta_{0}}{||h||_{L^{\infty}(\Omega)}}\right)^{\frac{1}{2+p_{\mathrm{min}}}}, \frac{\delta_{0}}{\mathcal{K}_{2}}\right\} & \text{for } \ m = 1 + p_{\mathrm{min}}.
\end{cases}
\]

Thus we finish the proof of Theorem~\ref{Section5:Thm1}.
\end{proof}

In the end, we give a complete proof of Theorem~\ref{Thm2}.

\begin{proof}[{Proof of Theorem~\ref{Thm2}}]
By scaling argument as in Remark \ref{Section3:rmk1}, we may assume that $ u(0) = 0 $, $ ||u||_{L^{\infty}(B_{1})} \leq 1 $ and $ \max \big\{||\mathrm{osc}_{F}||_{L^{\infty}(\Omega_{1}^{+})}, ||f||_{L^{\infty}(\Omega_{1}^{+})}, ||h||_{L^{\infty}(\Omega_{1}^{+})},    ||g||_{C^{1,\beta_{g}}(\Omega_{1}^{'})}      \big\}  \leq \delta_{0} $. Let $ \delta_{0}, \rho_{0} > 0 $ be the numbers granted by Theorem~\ref{Section5:Thm1}. We distinguish two cases:

\emph{Case 1}: $ |Du(0)| \leq \delta_{0} \rho^{\beta'} $. We define
\begin{equation}\label{Section5:eq3}
  \iota:= \bigg( \frac{|Du(0)|}{\delta_{0}}   \bigg)^{\frac{1}{\beta'}}.
\end{equation}
If $  \iota \leq t \leq \rho \leq \rho_{0}$, it follows from \eqref{Section5:eq3} that
\begin{equation}\label{Section5:eq4}
  |Du(0)| \leq \iota^{\beta'} \delta_{0}  \leq \delta_{0} t^{\beta'}.
\end{equation}
Then by Theorem~\ref{Section5:Thm1}, one obtain
\begin{equation*}
  \sup_{\Omega_{t}^{+}} |u(x)|  \leq \mathrm{C} t^{1+\beta'},
\end{equation*}
which, together with \eqref{Section5:eq4}, yields
\begin{equation*}
   \sup_{\Omega_{t}^{+}} |u(x)-Du(0)\cdot x| \leq (\mathrm{C}+\delta_{0})t^{1+\beta'}.
\end{equation*}

On the other hand, if $ 0 < t < \iota \leq \rho_{0} $, in this case we consider a new scaling function
\begin{equation*}
  u_{\iota}(x):= \frac{u(\iota x)}{\iota^{1+\beta'}} \quad   \text{in}  \quad  \Omega_{1, \iota}^{+}:= B_{1} \cap \{ x \in \mathbb{R}^{n}: x_{n} \geq \iota^{-1} \varphi(\iota x') \}.
\end{equation*}
Now we claim
$$ \sup_{x \in B_{1}} |u_{\iota}(x)| \leq   \mathrm{C} $$
for a universal constant $ \mathrm{C} > 0 $. In fact, from \eqref{Section5:eq3}, we have $ |Du(0)| = \delta_{0} \iota^{\beta}$, then using Theorem~\ref{Section5:Thm1}, it follows
\begin{equation*}
  \sup_{x \in \Omega_{1, \iota}^{+}} |u_{\iota}(x)| = \sup_{x \in \Omega_{1, \iota}^{+}} \bigg|\frac{u(\iota x)}{\iota^{1+\beta}}\bigg| \leq \mathrm{C},
\end{equation*}
which shows the claim.

Furthermore, a routine calculation reveals that $ u_{\iota} $ satisfies, in the viscosity sense, the following problem:
\begin{equation}
\label{Section5:eq120}
\left\{
     \begin{alignedat}{2}
          \big[|Du_{\iota}|^{p_{\iota}(x)}  + a_{\iota}(x)|Du_{\iota}|^{q_{\iota}(x)}       \big] F_{\iota}(D^{2}u_{\iota},x)+ h_{\iota}(x)|Du_{\iota}|^{m}   & = f_{\iota}(x)            \quad   &&   \text{in} \quad   \Omega_{1,r}^{+}    ,    \\
             u_{\iota} &  = g_{\iota}     \quad  &&   \text{on} \quad  \Omega_{1,r}^{'}  ,        \\
     \end{alignedat}
     \right.
\end{equation}
where
 \begin{equation*}
\left\{
     \begin{aligned}
& F_{\iota}(\mathrm{X}, x):= \iota^{1-\beta'} F(\iota^{\beta'-1}\mathrm{X}, \iota x);      \\
& h_{\iota}(x):= \iota^{1+\beta'[m-1-p_{\iota}(x)]} h(\iota x)     ;    \\
&  f_{\iota}(x):= \iota^{1-\beta'(1+p_{\iota}(x))} f(\iota x),  \quad    g_{\iota}(x):= \frac{u(\iota x)}{\iota^{1+\beta'}};    \\
&  p_{\iota}(x):= p(\iota x),   \quad   q_{\iota}(x):= q(\iota x),  \quad   \text{and}   \quad     a_{\iota}(x):=a(\iota x) \iota^{\beta'[q_{\iota}(x)-p_{\iota}(x)]}.
\end{aligned}
     \right.
\end{equation*}
We use \hyperref[A5a]{\bf (A5a)}, it follows that
\begin{equation*}
  |f_{\iota}(x)|  \leq  \mathcal{K}_{2} \iota^{1+\theta_{f}-\beta'(1+p_{\iota}(x))}|x|^{\theta_{f}}  \leq \mathcal{K}_{2},
\end{equation*}
where we have used that fact $ 1+\theta_{f}-\beta'(1+p_{\mathrm{max}}) > 0 $ and $ \iota \in (0,1) $. Moreover, $ F_{\iota} $ still fulfills the conditions \hyperref[A1]{\bf (A1)} and \hyperref[A2]{\bf (A2)}. Then by using Theorem~\ref{Section2:thm3}, we have $ u_{\iota} \in C^{1, \beta'}_{\text{loc}}$. In addition, recall $ |D u_{\iota}(0)| = \delta_{0}$, then there exists a small constant $ r $, independent of $ \iota $, it infers
\begin{equation*}
 \frac{\delta_{0}}{2} \leq  |Du_{\iota}(x)| \leq  2 \delta_{0}  \quad  \text{in}  \quad  \Omega_{r, \delta_{0}}^{+}.
\end{equation*}

We may now rewrite \eqref{Section5:eq120} as
\begin{equation*}
   F_{\iota}(D^{2}u_{\iota}, x) = \frac{f_{\iota}(x)-h_{\iota}(x)|Du_{\iota}|^{m}   }{|Du_{\iota}|^{p_{\iota}(x)}  + a_{\iota}(x)|Du_{\iota}|^{q_{\iota}(x)}}:= \widetilde{\widetilde{f}}_{\iota}(x)  \quad  \text{in}  \quad  \Omega_{r, \delta_{0}}^{+}.
\end{equation*}
That is, $ u_{\iota} $ solves a fully nonlinear uniformly elliptic equation with a bounded source term $ \widetilde{\widetilde{f}}_{\iota}$ on the right-hand side. Consequently, using the regularity result established by \cite[Theorem 2.1]{AS23}, then
\begin{equation*}
  \sup_{x \in \Omega_{\delta, \delta_{0}}^{+}}|u_{\iota}(x) -  D u_{\iota}(0) \cdot x| \leq \mathrm{C} \delta^{1+\beta_{g}}
\end{equation*}
for every $  0 \leq \delta \leq \frac{r}{2} $. Rescaling back leads to
\begin{equation*}
  \sup_{x \in \Omega_{t}^{+}} |u(x) - Du(0) \cdot x| \leq  \mathrm{C} t^{1+\beta'}
\end{equation*}
for every $ 0 < t \leq \frac{\iota r}{2} $. When $ \frac{\iota r}{2} < t < r $, using the result $ t = \iota $ in Case 1, we have
\begin{align*}
  \sup_{x \in \Omega_{t}^{+}} |u(x) - Du(0) \cdot x| \leq  \sup_{x \in \Omega_{\iota}^{+}} |u(x) - Du(0) \cdot x| & \leq (\mathrm{C}+\delta_{0})\iota^{1+\beta'}   \\
&  \leq (\mathrm{C}+\delta_{0})\left(\frac{2}{r}\right)^{1+\beta'} t^{1+\beta'}.
\end{align*}
This completes the proof of \emph{Case 1}.

\emph{Case 2}: $ |Du(0)| \geq \delta_{0} \rho^{\beta'} $. In this scenario, we consider a scaling function
\begin{equation*}
  \widehat{u}(x):= \frac{\delta_{0} \rho^{\beta'}}{|Du(0)|} u(x),
\end{equation*}
and it is easy to see $ D \widehat{u}(0) = \delta_{0} \rho^{\beta'} $, which reduces to the \emph{Case 1}.

Combining the two cases above, we complete the proof of Theorem~\ref{Thm2}.
\end{proof}

\section{Proof of Theorem \ref{Thm3}: non-degeneracy}
\label{Section 6}
In this section, we are concerned to prove Theorem~\ref{Thm3} by constructing a delicate auxiliary function and using the general comparison principle.

We begin by recalling the following comparison principle, which will play a crucial role in the proof that follows.

\begin{proposition}[{Comparison principle II}]
\label{Section8:prop1}
Suppose that the assumptions \hyperref[A1]{\bf (A1)}, \hyperref[A2]{\bf (A2)}, \hyperref[A3a]{\bf (A3a)} and \hyperref[A4a]{\bf (A4a)} hold. Let $ u $ be a lower semi-continuous bounded super-solution of
\begin{equation*}
  \big[|Du|^{p(x)}+ a(x)|Du|^{q(x)}\big]F(D^{2}u,x)+ h(x)|Du|^{m}  = f(x)    \quad   \text{in}   \quad  B_{1}
\end{equation*}
and $ v $ be is a upper semi-continuous bounded sub-solution of
\begin{equation*}
  \big[|Dv|^{p(x)}+ a(x)|Dv|^{q(x)}\big]F(D^{2}v,x)+ h(x)|Dv|^{m}  = g(x)    \quad   \text{in}   \quad  B_{1}.
\end{equation*}
Assume that either $ u $ or $ v $ is Lipschitz continuous and that $ f < g $ in $ \overline{B_{1}} $. If $ v \leq u $ on $ \partial B_{1} $, then $ v \leq u $ in $ B_{1} $.
\end{proposition}

Since the proof is similar to that of Proposition~\ref{Section3:prop1} and is completely standard, we omit it and refer the reader to \cite[Proposition 4.1]{BD16}.

With Proposition~\ref{Section8:prop1} in hand, we give the proof of Theorem~\ref{Thm3}.

\begin{proof}[Proof of Theorem \ref{Thm3}]
Firstly, for $ x_{0} \in \Gamma' \Subset \Omega $, we define a auxiliary function as follows
\begin{equation*}
  u_{r, x_{0}}(x):= \frac{u(x_{0}+r x)-u(x_{0})+\epsilon}{r^{\frac{2+p_{\mathrm{min}}}{1+p_{\mathrm{min}}}}},   \quad   x \in B_{1}(0)
\end{equation*}
with a small $ \epsilon > 0 $. Direct calculation yields that $  u_{r, x_{0}}  $ is a viscosity solution to
\begin{equation}\label{Section8:eq1}
  [|Du_{r, x_{0}}|^{\widetilde{p}(x)} + \widetilde{a}(x)|Du_{r, x_{0}}|^{\widetilde{q}(x)} ] \widetilde{F}(D^{2}u_{r,x_{0}}, x)  +  \widetilde{h}(x) |Du_{r, x_{0}}|^{m} = \widetilde{f}(x)   \quad   \text{in}  \quad  B_{1}(0),
\end{equation}
where
\begin{equation*}
\left\{
     \begin{aligned}
& \widetilde{F}(\mathrm{X}, x):= r^{2-\beta''} F(r^{\beta''-2}\mathrm{X}, x_{0}+rx),   \quad  \beta'':=\frac{2+p_{\mathrm{min}}}{1+p_{\mathrm{min}}};         \\
& \widetilde{h}(x):= h(x_{0}+rx) r^{(\beta''-1)[m-\widetilde{p}(x)]+2-\beta''} ;    \\
& \widetilde{f}(x):= f(x_{0}+rx) r^{(2-\beta'')-(\beta''-1)\widetilde{p}(x)}     ;    \\
&  \widetilde{a}(x):= a(x_{0}+rx) r^{(\beta''-1)[\widetilde{q}(x)-\widetilde{p}(x)]};   \\
& \widetilde{p}(x):= p(x_{0}+rx)  \quad   \text{and}   \quad  \widetilde{q}(x):= q(x_{0}+rx).
\end{aligned}
     \right.
\end{equation*}

We now choose a comparison function:
\begin{equation*}
  \Theta(x):= \mathrm{C}_{0} |x|^{\beta''},
\end{equation*}
with the constant $ \mathrm{C}_{0} > 0 $ chosen to satisfy
\begin{equation}\label{Section8:eq2}
  [|D\Theta|^{\widetilde{p}(x)} + \widetilde{a}(x)|D \Theta|^{\widetilde{q}(x)} ] \widetilde{F}(D^{2}\Theta, x)  +  \widetilde{h}(x) |D\Theta|^{m} <  \widetilde{f}(x)   \quad   \text{in}  \quad  B_{1}(0).
\end{equation}
We compute
\begin{align*}
& D_{i}\Theta=\mathrm{C}_{0} \beta''|x|^{\beta''-2}x_{i}; \quad  i=1,2,\cdots, n  \\
& D_{ij} \Theta= \mathrm{C}_{0}\beta''\bigg[\delta_{ij}+(\beta''-2)\frac{x_{i}x_{j}}{|x|^{2}}   \bigg]|x|^{\beta''-2}, \quad i,j=1,2,\cdots, n   \\
& \quad \quad  := \mathrm{C}_{0}\beta'' |x|^{\beta''-2} B_{ij},
\end{align*}
where
\begin{equation*}
     B_{ij}:= \bigg[\delta_{ij}+(\beta''-2)\frac{x_{i}x_{j}}{|x|^{2}}\bigg].
\end{equation*}

 Observed that Hessian matrix $ (B_{ij})_{i,j=1}^{n} $ has a simple eigenvalue $ \beta''-1 $ and an eigenvalue $ 1 $ with multiplicity $ n-1 $. This along with \hyperref[A1]{\bf (A1)} lead to
\begin{equation}\label{Section8:eq3}
  \widetilde{F}(D^{2}\Theta, x) \leq \mathscr{P}^{+}_{\lambda,\Lambda}(D^{2}\Theta) \leq \Lambda (n+\beta''-2)\mathrm{C}_{0}\beta''|x|^{\beta''-2}.
\end{equation}

Next, we calculate the degenerate term
\begin{align}\label{Section8:eq4}
\begin{split}
& [|D\Theta|^{\widetilde{p}(x)} + \widetilde{a}(x)|D \Theta|^{\widetilde{q}(x)} ]   \\
& \leq  \bigg[(\mathrm{C}_{0}\beta'')^{\widetilde{p}(x)}  + ||a||_{L^{\infty}(B_{1})} r^{(\beta''-1)[\widetilde{q}(x)-\widetilde{p}(x)]}(\mathrm{C}_{0}\beta'')^{\widetilde{q}(x)}|x|^{(\beta''-1)[\widetilde{q}(x)-\widetilde{p}(x)]}    \bigg]|x|^{(\beta''-1)\widetilde{p}(x)}
\end{split}
\end{align}
and Hamiltonian terms
\begin{align*}
|D\Theta|^{m}|\widetilde{h}(x)|  \leq ||h||_{\infty} r^{(\beta''-1)[m-\widetilde{p}(x)]+2-\beta'} (\mathrm{C}_{0}\beta'')^{m}|x|^{(\beta''-1)m}
\end{align*}
which, together with \eqref{Section8:eq3} and \eqref{Section8:eq4}, yields
\begin{align}\label{Section8:eq5}
\begin{split}
& [|D\Theta|^{\widetilde{p}(x)} + \widetilde{a}(x)|D \Theta|^{\widetilde{q}(x)} ] \widetilde{F}(D^{2}\Theta, x)  +  \widetilde{h}(x) |D\Theta|^{m}      \\
& \leq   \bigg[(\mathrm{C}_{0}\beta'')^{\widetilde{p}(x)}  + ||a||_{L^{\infty}(B_{1})} r^{(\beta''-1)[\widetilde{q}(x)-\widetilde{p}(x)]}(\mathrm{C}_{0}\beta'')^{\widetilde{q}(x)}|x|^{(\beta''-1)[\widetilde{q}(x)-\widetilde{p}(x)]}    \bigg]\Lambda (n+\beta''-2)        \\
 &    \quad       \times \mathrm{C}_{0}\beta'' |x|^{(\beta''-1)\widetilde{p}(x)+(\beta''-2)}   + ||h||_{L^{\infty}(B_{1})} r^{(\beta''-1)[m-\widetilde{p}(x)]+2-\beta''} (\mathrm{C}_{0}\beta'')^{m}|x|^{(\beta''-1)m}         \\
 &  \leq  \bigg[ (\mathrm{C}_{0}\beta'')^{\widetilde{p}(x)}  + ||a||_{\infty}  (\mathrm{C}_{0}\beta'')^{\widetilde{q}(x)}     \bigg]\Lambda (n+\beta''-2)\mathrm{C}_{0} \beta''+ ||h||_{L^{\infty}(B_{1})}(\mathrm{C}_{0} \beta'')^{m}
 \end{split}
\end{align}
where we have used three inequalities:
\begin{equation*}
\left\{
     \begin{aligned}
& (\beta''-1)[\widetilde{q}(x)-\widetilde{p}(x)]  \geq 0;      \\
&  (\beta''-1)\widetilde{p}(x) +(\beta''-2) = 0;    \\
& (\beta''-1)[m-\widetilde{p}(x)]+2-\beta''  \geq 0  \quad  \left(\text{since} \  \beta'' \leq \frac{2+p_{\mathrm{min}}-m}{1+p_{\mathrm{max}}-m}\right)
\end{aligned}
     \right.
\end{equation*}
in the second inequality.

Now we distinguish two cases:

\emph{Case 1}. $ \mathrm{C}_{0} \geq 1 $. It is immediately seen that
\begin{align*}
\text{Left of} \ \eqref{Section8:eq5} \leq \bigg[ (\mathrm{C}_{0}\beta'')^{p_{\mathrm{max}}} + ||a||_{L^{\infty}(B_{1})}  (\mathrm{C}_{0}\beta'')^{q_{\mathrm{max}}}   \bigg] \Lambda (n+\beta''-2)\mathrm{C}_{0} \beta''+ ||h||_{L^{\infty}(B_{1})}(\mathrm{C}_{0} \beta'')^{m}
\end{align*}

At this point, we define a function $ \widetilde{g}: [0, \infty) \rightarrow \mathbb{R} $ given by
\begin{equation*}
  \widetilde{g}(t):= \Lambda (n+\beta''-2) \bigg[ 1+||a||_{L^{\infty}(B_{1})} t^{q_{\mathrm{max}}-p_{\mathrm{max}}}  (\beta'')^{q_{\mathrm{max}}-p_{\mathrm{max}}}   \bigg] t^{1+p_{\mathrm{max}}}(\beta'')^{1+p_{\mathrm{max}}} + ||h||_{L^{\infty}(B_{1})}(\beta'')^{m} t^{m}  - \inf_{x \in \Omega} f(x).
 \end{equation*}
Observing that $ \widetilde{g}(0) < 0 $ and $ \widetilde{g}(t) \rightarrow \infty $ as $ t \rightarrow \infty $, by the Intermediate Value Theorem, there exists at least one root for $ \widetilde{g}(t) = 0 $ in $ (0, \infty) $. Let $ T' $ denote the smallest such root, then one may select a $ \mathrm{C}_{0} \in (0, T') $ such that $ \widetilde{g}(\mathrm{C}_{0}) < 0  $, namely,
\begin{equation*}
   [|D\Theta|^{\widetilde{p}(x)} + \widetilde{a}(x)|D \Theta|^{\widetilde{q}(x)} ] \widetilde{F}(D^{2}\Theta, x)  +  \widetilde{h}(x) |D\Theta|^{m} <  \widetilde{f}(x)   \quad   \text{in}  \quad  B_{1}(0).
\end{equation*}
In this setting, we prove \eqref{Section8:eq2}.

\emph{Case 2}. $ \mathrm{C}_{0} < 1 $. Analogous to the argument in \emph{Case 1}, it is straightforward to verify that
\begin{align*}
\text{Left of} \ \eqref{Section8:eq5} \leq \bigg[ (\mathrm{C}_{0})^{p_{\mathrm{min}}}(\beta'')^{p_{\mathrm{max}}} + ||a||_{L^{\infty}(B_{1})}  (\mathrm{C}_{0})^{q_{\mathrm{min}}}(\beta'')^{q_{\mathrm{max}}}   \bigg] \Lambda (n+\beta''-2)\mathrm{C}_{0} \beta''+ ||h||_{L^{\infty}(B_{1})}(\mathrm{C}_{0} \beta'')^{m}.
\end{align*}

In this setting, we define a function $ \widetilde{\widetilde{g}}: [0, \infty) \rightarrow \mathbb{R} $ as follows
\begin{equation*}
  \widetilde{\widetilde{g}}(t):= \Lambda (n+\beta''-2) \beta'' \bigg[t^{1+p_{\mathrm{min}}} (\beta'')^{p_{\mathrm{max}}} + ||a||_{L^{\infty}(B_{1})} t^{1+q_{\mathrm{min}}}(\beta'')^{q_{\mathrm{max}}}    \bigg] + ||h||_{L^{\infty}(B_{1})} (\beta'')^{m} t^{m}  - \inf_{x \in \Omega} f(x).
\end{equation*}
The remaining argument proceeds as in \emph{Case 1}, hence we also obtain \eqref{Section8:eq2}.

Finally, if $ u_{r, x_{0}} \leq  \Theta $ on $ \partial B_{1}(0) $, then the comparison principle (Proposition \ref{Section8:prop1}) ensures that $ u_{r, x_{0}} \leq  \Theta $ in $ B_{1}(0)$,  contradicting $ u_{r, x_{0}}(0) > 0 $. Consequently, we can find a point $ z \in \partial B_{1}(0) $ satisfying
\begin{equation*}
u_{r,x_0}(z) > \Theta(z) = \mathrm{C}_{0}.
\end{equation*}
By rescaling and letting $ \epsilon \rightarrow 0 $,  the proof is complete.
\end{proof}

\section{Applications}\label{Section 7}
In this section, we discuss how our main regularity results can be applied to derive sharp global estimates for two cornerstone operators: the $ \infty $-Laplacian and the $ p$-Laplacian. By assuming a suitable viscosity curvature condition on the level sets of the solution, we provide a new perspective for addressing these long-standing regularity conjectures.

\subsection{Global $ C^{1, \frac{1}{3}} $ regularity for infinity-Poisson}
The infinity Laplacian operator, defined as
\begin{equation*}
  \Delta_{\infty} u:= \Sigma_{i,j=1}^{n} u_{i} u_{j} u_{ij} = (Du)^{\mathrm{T}} D^{2}u Du,
\end{equation*}
arises naturally in the study of best Lipschitz extensions and has been a subject of intense investigation since the pioneering work of Aronsson in the 1960s.

A celebrated open conjecture in this field is whether infinity-harmonic functions satisfy a universal $C^{1, \frac{1}{3}}$ regularity estimate. This threshold is suggested by Aronsson's famous example $a(x, y):= x^{\frac{4}{3}} - y^{\frac{4}{3}}$, which is infinity-harmonic but possesses a gradient that is only $C^{0, \frac{1}{3}}$ continuous. While significant breakthroughs have been achieved, such as $C^1$ regularity in the plane by Savin \cite{Sa05}, and in two-dimensional case, Evans and Savin \cite{ESa08} sharpened the result to $C^{1,\alpha}$ for some small $\alpha > 0$. Furthermore, everywhere differentiability in all dimensions by Evans and Smart \cite{ES11}, the sharp $C^{1, \alpha}$ regularity in high dimensions remains one of the most challenging open problems.

In 2015, Ara\'{u}jo-Ricarte-Teixeira \cite{ART15} obtained sharp interior $C^{1, \frac{1}{3}}$ estimates for infinity-harmonic functions by assuming specific structural profiles, such as {\it separable variables} (cf. \cite[Proposition 4.2]{ART15}) or {\it radial asymptotic behavior near singular points} (cf. \cite[Theorem 4.3]{ART15}). These assumptions allow the operator to be treated under a power-type scaling regime.

While significant progress has been made regarding interior regularity, the boundary regularity theory for the infinity Laplacian remains markedly underdeveloped, with only a few results currently available in the literature. Seminal work by Wang-Yu \cite{WY12} established the first boundary regularity results for infinity harmonic functions in high dimensions, assuming $C^1$ regularity for both the domain and the boundary data. Subsequently, Hong \cite{H13} relaxed these requirements, weakening the $C^1$ assumptions on the domain and data to mere everywhere differentiability. However, the optimal $C^{1, \frac{1}{3}}$ regularity up to the boundary in arbitrary dimensions remains unknown.

In contrast to these structural profiles \cite[Proposition 4.2, Theorem 4.3]{ART15}, in this paper, we provide an intrinsic geometric approach based on the properties of the solution's level sets. To formulate this geometric information in the viscosity setting, we introduce the following nonzero-gradient curvature condition.

\begin{definition}[Nonzero-gradient viscosity curvature condition]\label{Section7:def-curvature}

Let $u\in C^{0}(B_1)$ and $k\in C^{0}(B_1)\cap L^\infty(B_1)$.
For $\varphi\in C^2(B_1)$ and a point $x_0 \in B_{1}$ with
$D\varphi(x_0)\neq0$, write
\[
\mathcal K_\varphi(x_0)
:=
\frac{1}{|D\varphi(x_0)|}
\operatorname{tr}\!\left[
\left(
\textbf{I}\mathrm{d_{n}}-
\frac{
D\varphi(x_0)\otimes D\varphi(x_0)
}{
|D\varphi(x_0)|^2
}
\right)
D^2\varphi(x_0)
\right].
\]
We say that $u$ satisfies the curvature condition associated with $k$
if the following two requirements hold:
\begin{enumerate}[label=(\roman*)]

\item Whenever $u-\varphi$ has a local maximum at $x_0$ and
$D\varphi(x_0)\neq0$, one has
\[
\mathcal K_\varphi(x_0)\ge k(x_0).
\]

\item Whenever $u-\varphi$ has a local minimum at $x_0$ and
$D\varphi(x_0)\neq0$, one has
\[
\mathcal K_\varphi(x_0)\le k(x_0).
\]

\end{enumerate}
\end{definition}

\begin{remark}
This condition imposes no requirement on tests with zero gradient.
The infinity-Poisson or $p$-Poisson equation is understood
in the usual viscosity sense, including tests with zero gradient.
\end{remark}

\begin{remark}
At points where $u$ is $C^2$ and $Du\neq0$, the preceding condition
is equivalent to
\[
\operatorname{div}\!\left(\frac{Du}{|Du|}\right)=k(x).
\]
Indeed, at an upper contact point,
$D\varphi=Du$ and $D^2\varphi-D^2u$ is positive semi-definite.
Writing $e:=Du/|Du|$, we obtain
\[
\mathcal K_\varphi-\mathcal K_u
=
\frac{1}{|Du|}
\operatorname{tr}\!\left[
(\textbf{I}\mathrm{d_{n}}-e\otimes e)(D^2\varphi-D^2u)
\right]
\ge0.
\]
The inequality is reversed at a lower contact point.
Here mean curvature denotes the sum of the principal curvatures,
with orientation determined by $ Du/|Du| $.

\end{remark}

Now we give

\begin{theorem}
\label{Section 7:thm1}
Let $u\in C^{0}(\overline{B_1})$ be a viscosity solution of
\begin{equation}\label{Section7:eq1}
\begin{cases}
\Delta_{\infty}u=f(x) &  \quad   \text{in }   \quad B_{1},\\
u=g & \quad   \text{on } \quad     \partial B_{1},
\end{cases}
\end{equation}
where
\[
f\in C^{0}(B_1)\cap L^\infty(B_1),
\quad
g\in C^{1,1}(\partial B_1).
\]
Assume that there exists
$k\in C^{0}(B_1)\cap L^\infty(B_1)$ such that
$u$ satisfies Definition~\ref{Section7:def-curvature}.
Then
\[
u\in C^{1,\frac13}(\overline{B_1}).
\]

\end{theorem}

\begin{proof}

We first prove that $u$ is a viscosity solution of
\begin{equation}\label{Section7:thm-eq1}
|Du|^2\Delta u-k(x)|Du|^3 = f(x)
\quad   \text{in}  \quad  B_1.
\end{equation}
Here we only prove the sub-solution case, and the remaining case follows similarly. Indeed, let $\varphi\in C^2(B_1)$ touch $u$ from above at $x_0\in B_1$.
The viscosity sub-solution property of the infinity-Poisson equation gives
\[
\Delta_\infty\varphi(x_0)\ge f(x_0).
\]
If $D\varphi(x_0)\neq0$, the geometric identity
\[
|D\varphi|^2\Delta\varphi
=
\Delta_\infty\varphi
+
|D\varphi|^3\mathcal K_\varphi
\]
and Definition~\ref{Section7:def-curvature} yield
\[
\begin{aligned}
&|D\varphi(x_0)|^2\Delta\varphi(x_0)
-k(x_0)|D\varphi(x_0)|^3\\
&\quad=
\Delta_\infty\varphi(x_0)
+
|D\varphi(x_0)|^3
\bigl(
\mathcal K_\varphi(x_0)-k(x_0)
\bigr)\\
&\quad\ge f(x_0).
\end{aligned}
\]
If $D\varphi(x_0)=0$, then
$\Delta_\infty\varphi(x_0)=0\ge f(x_0)$, which is exactly the
required sub-solution inequality for~\eqref{Section7:thm-eq1}.
Hence we shown that $ u $ is a viscosity sub-solution to \eqref{Section7:thm-eq1}.

Noticing that \eqref{Section7:thm-eq1} has the structure of the general equation in Corollary~\ref{Section1:coro1}, with
\[
\Phi(t,x)=t^2,
\qquad
F(\mathrm{M},x)=\operatorname{tr}\mathrm{M},
\qquad
H(|t|,x)=-k(x)|t|^3, \qquad  k\in C^{0}(B_1)\cap L^\infty(B_1),
\]
thus Corollary~\ref{Section1:coro1} gives
\begin{equation*}
  u\in C^{1,\frac13}(\overline{B_1}).
\end{equation*}
\end{proof}

\subsection{Global $ C^{1,\frac{1}{p-1}} $ regularity for $ p$-Poisson ($ p > 2$)}

The regularity theory for the $p$-Laplacian has been extensively
developed over the past decades. It is classical that solutions to the $p$-Poisson equation
\[
-\Delta_p u=f
\]
with bounded source terms are locally of class $C^{1,\alpha}$ for
some $\alpha=\alpha(n,p)\in(0,1)$. A natural radial example shows
that, when $p>2$, the exponent $\frac{1}{p-1}$ is the optimal
regularity threshold for bounded right-hand sides. This motivates
the so-called $C^{p'}$-regularity conjecture, which predicts
\[
u\in C_{\mathrm{loc}}^{1,\frac{1}{p-1}}.
\]

For $n=2$, Ara\'ujo, Teixeira and Urbano \cite{ATU17} established the
optimal interior $C^{1,\frac{1}{p-1}}$ regularity for the $p$-Poisson
equation with bounded source terms. For $n\geq 2$, they further
obtained in \cite{ATU18} the same regularity for radially symmetric
solutions and for solutions with no saddle critical points, while
the general higher-dimensional case remains open.

Concerning boundary regularity, Lieberman \cite{L88} established
global $C^{1,\alpha}$ estimates for a broad class of degenerate
elliptic equations under suitable assumptions on the domain and the
boundary data. For $n=2$, Haque \cite{H18} obtained sharp boundary
regularity for generalized $p$-Poisson equations with $f\in L^q$,
$q>2$, under both Dirichlet and Neumann boundary conditions, with
the optimal exponent
\[
\alpha=\frac{1-\frac{2}{q}}{p-1},
\]
which reduces to $\frac{1}{p-1}$ when $f\in L^\infty$. In higher
dimensions, the corresponding sharp global regularity is not known
in general. In contrast to the structural conditions considered in
\cite{ATU18}, we formulate a new geometric condition, and obtain the global optimal regularity
\[
u\in C^{1,\frac{1}{p-1}}(\overline{B_1}).
\]

Now we state the main result in this section.

\begin{theorem}\label{Section 7:thm2}
Let $ u \in C^{0}(\overline{B_{1}})$ be a viscosity solution to
\begin{equation*}
\left\{
     \begin{alignedat}{2}
           \Delta_{p} u & = f(x)         \quad   &&   \text{in} \quad B_{1},    \\
          u(x)   &  = g(x)   \quad  &&   \text{on} \ \  \partial B_{1},        \\
     \end{alignedat}
     \right.
\end{equation*}
where
\[
f\in C^{0}(B_1)\cap L^\infty(B_1),
\qquad
g\in C^{1,1}(\partial B_1).
\]
Assume that $u$ satisfies Definition~\ref{Section7:def-curvature}
for some $k\in C^{0}(B_1)\cap L^\infty(B_1)$.
Then
\[
u\in C^{1,\frac{1}{p-1}}(\overline{B_1}).
\]

\end{theorem}

\begin{proof}

We first prove that $u$ is a viscosity solution of
\begin{equation}\label{Section7:thm-eq2}
|Du|^{p-2}\Delta u
-\frac{p-2}{p-1}k(x)|Du|^{p-1}
=\frac{f(x)}{p-1}
\quad  \text{in }  \quad  B_1.
\end{equation}
We only prove the sub-solution case, since the super-solution case follows by reversing the corresponding inequalities.
Let $\varphi\in C^2(B_1)$ touch $u$ from above at $x_0\in B_1$.
The viscosity sub-solution property of the $p$-Poisson equation gives
\[
\Delta_p\varphi(x_0)\ge f(x_0).
\]
If $D\varphi(x_0)\neq0$, we use the identities
\[
\Delta_p\varphi
=
|D\varphi|^{p-2}\Delta\varphi
+
(p-2)|D\varphi|^{p-4}\Delta_\infty\varphi
\]
and
\[
\Delta_\infty\varphi
=
|D\varphi|^2\Delta\varphi
-
|D\varphi|^3\mathcal K_\varphi
\]
to obtain
\[
\Delta_p\varphi
=
(p-1)|D\varphi|^{p-2}\Delta\varphi
-
(p-2)|D\varphi|^{p-1}\mathcal K_\varphi.
\]
Since $p>2$, Definition~\ref{Section7:def-curvature} yields
\[
\begin{aligned}
&|D\varphi(x_0)|^{p-2}\Delta\varphi(x_0)
-\frac{p-2}{p-1}k(x_0)|D\varphi(x_0)|^{p-1}\\
&\quad=
\frac{\Delta_p\varphi(x_0)}{p-1}
+
\frac{p-2}{p-1}|D\varphi(x_0)|^{p-1}
\bigl(
\mathcal K_\varphi(x_0)-k(x_0)
\bigr)\\
&\quad\ge\frac{f(x_0)}{p-1}.
\end{aligned}
\]
If $D\varphi(x_0)=0$, then $\Delta_p\varphi(x_0)=0$ because
$p>2$. Indeed, the map $q\mapsto |q|^{p-2}q$ is differentiable
at $q=0$, with derivative zero.
Consequently, the original sub-solution inequality gives
\[
0\ge f(x_0),
\]
which implies exactly the desired sub-solution inequality for
\eqref{Section7:thm-eq2}. Thus $u$ is a viscosity sub-solution of \eqref{Section7:thm-eq2}. The analogous argument proves the super-solution property.

Observed that equation~\eqref{Section7:thm-eq2} fits the framework of Corollary~\ref{Section1:coro1}, with
\[
\Phi(t,x)=t^{p-2},
\qquad
F(\mathrm M,x)=\operatorname{tr}\mathrm M,
\qquad
H(|t|,x)=-\frac{p-2}{p-1}k(x)|t|^{p-1},
\]
and right-hand side $\frac{f}{p-1}$.
Since $k\in C^{0}(B_1)\cap L^\infty(B_1)$, the Hamiltonian term is
continuous and satisfies
\[
|H(|t|,x)|
\le
\frac{p-2}{p-1}\|k\|_{L^\infty(B_1)}t^{p-1}.
\]
Thus \hyperref[A4]{\bf (A4)} holds with
\[
i(\Phi)=s(\Phi)=p-2>0,
\qquad
m=p-1=1+i(\Phi).
\]
By applying
Corollary~\ref{Section1:coro1}, we have
\begin{equation*}
u\in C^{1,\frac{1}{p-1}}(\overline{B_1}).
\end{equation*}\end{proof}

The following example shows that the H\"{o}lder regularity exponents for the gradient in Theorems~\ref{Section 7:thm1} and~\ref{Section 7:thm2} cannot be improved. The solutions have constant right-hand sides and $ C^{1,1}$ boundary data, and satisfy the viscosity curvature condition with $ k \equiv 0 $.

\begin{example}
Let $n\geq2$, and set
\[
t:=1-x_n,   \quad    u_s(x)=t^s,  \quad   1<s<2.
\]
Since $t>0$ in $B_1$, then the function $u_s$ is smooth and has nonzero gradient throughout $B_1$. Direct calculation gives
\[
Du_s=-s t^{s-1}e_n,  \quad
D^2u_s=s(s-1)t^{s-2}e_n\otimes e_n.
\]
Consequently,
\[
\Delta_\infty u_s=s^3(s-1)t^{3s-4},
\]
and, for $ p>2 $,
\[
\Delta_pu_s
=(p-1)s^{p-1}(s-1)t^{(s-1)(p-1)-1}.
\]
In particular,
\[
\Delta_\infty u_{\frac{4}{3}}=\frac{64}{81},
\quad
\Delta_pu_{\frac{p}{p-1}}
=\left(\frac{p}{p-1}\right)^{p-1}.    
\]
Thus both choices have continuous bounded source terms.

We next verify the viscosity curvature condition with $k=0$.
The level sets of $ u_s $ are hyperplanes, and $ \frac{Du_s}{|Du_s|}=-e_n$. If test function $ \varphi $ touches $u_s$ from above at an interior point, then $D\varphi=Du_s$ and
$D^2\varphi-D^2u_s\geq0$. Therefore,
\[
\mathcal K_\varphi
=\frac{1}{|Du_s|}
\operatorname{tr}\!\left[
(\textbf{I}\mathrm{d_{n}}-e_n\otimes e_n)(D^2\varphi-D^2u_s)
\right]\geq0.
\]
At a lower contact point, the inequality is reversed.
Hence $u_s$ satisfies Definition~\ref{Section7:def-curvature} with $ k\equiv0 $.

The boundary datum $ g_s= u_s|_{\partial B_1}$ belongs to
$C^{1,1}(\partial B_1)$. Indeed, near the north pole $e_n$,
the sphere is parametrized by
$x_n=\sqrt{1-|x'|^2}$, and
\begin{equation}\label{Section7:thm-eq33}
g_s(x')
=|x'|^{2s}
\left(1+\sqrt{1-|x'|^2}\right)^{-s}.
\end{equation}
Since $2s>2$, \eqref{Section7:thm-eq33} extends as a $C^2$ function at $x'=0$. Away from $e_n$, the boundary datum is smooth. Thus $g_s\in C^2(\partial B_1)\subset C^{1,1}(\partial B_1)$.

Finally, we see that $u_s\in C^{1,s-1}(\overline{B_1})$ since 
\[
|Du_s(x)-Du_s(y)|
\leq s|x-y|^{s-1}.
\]
On the other hand, along $x_t=(1-t)e_n$, we have 
\[
Du_s(e_n)=0, \quad
|Du_s(x_t)-Du_s(e_n)|=s t^{s-1}.
\]
For every $ s-1<\gamma\leq1 $, a routine calculation reveals that    
\[
\frac{|Du_s(x_t)-Du_s(e_n)|}
{|x_t-e_n|^\gamma}
=s t^{s-1-\gamma}\longrightarrow\infty
\quad  \text{as } t \rightarrow  0.
\]
Hence $u_s\notin C^{1,\gamma}(\overline{B_1})$ for such
$\gamma$. Taking $s= \frac{4}{3}$ and $s= \frac{p}{p-1}$ proves that
the exponents $ \frac{1}{3}$ and $\frac{1}{p-1}$ in Theorems~\ref{Section 7:thm1} and~\ref{Section 7:thm2}, respectively, cannot be improved.        
\end{example}

\end{document}